\RequirePackage[l2tabu, orthodox]{nag}		% checks for outdated latex commands ... you can remove this if you want
\documentclass[10pt]{article}

\usepackage[T1]{fontenc}
\usepackage[frenchmath]{mathastext}							% sets default math font to the default text font (in italics)										% improves spacing
\usepackage{natbib}       									
\usepackage{amsmath}                        % need for subequations
\numberwithin{equation}{section}
\usepackage{graphicx} 											% if you compile LaTeX => PS => PDF, then uncomment this and uncomment next two lines
\usepackage{adjustbox,float}
\usepackage{enumitem}												% allows you to customize spacing of lists
\usepackage{mdwlist}												% enumerate* itemize* are compact
\usepackage[dvipsnames]{xcolor}							% gives you the colors you will use in the next line
\usepackage[plainpages=false, pdfpagelabels]{hyperref} % enables and customizes hyperlinks
\hypersetup{
colorlinks   = true,
citecolor    = RoyalBlue, %NavyBlue,
linkcolor    = RubineRed, %Maroon,
urlcolor     = Turquoise
}
\usepackage{amssymb}                        % gives you \mathbb{} font
\usepackage[mathscr]{eucal}                 % gives you \mathscr font
\usepackage{dsfont}													% gives you \mathds{} font
\usepackage[paperwidth=8.5in,paperheight=11in,top=1.25in, bottom=1.25in, left=1.0in, right=1.0in]{geometry} %use with 10pt font
\usepackage{mathtools}                      % need for `show only references'
\mathtoolsset{showonlyrefs=true}          % only equations which are labeled AND referenced will numbered use \eqref{} instead of (\ref{})
\usepackage{amsthm}                         % need for theorem-proof environment
\allowdisplaybreaks                       % allows page breaks for long equations...you can prevent a page-break with \\*
\theoremstyle{plain}
\newtheorem{theorem}{Theorem}
\numberwithin{theorem}{section}
\newtheorem{lemma}{Lemma}       	% [theorem] ==> theorems and lemmas will share a counter
\numberwithin{lemma}{section}
\newtheorem{proposition}{Proposition}
\numberwithin{proposition}{section}
\newtheorem{corollary}{Corollary}
\numberwithin{corollary}{section}
\theoremstyle{definition}
\newtheorem{definition}{Definition}
\numberwithin{definition}{section}

\numberwithin{example}{section}
\newtheorem{remark}{Remark}
\numberwithin{remark}{section}

\numberwithin{assumption}{section}

\newcommand\Eb{\mathds{E}}

\newcommand\Rb{\mathds{R}}

\newcommand\Ac{\mathscr{A}}

\newcommand\Dc{\mathscr{D}}

\newcommand\Lc{\mathscr{L}}

\newcommand\Pc{\mathscr{P}}

\newcommand\eps{\varepsilon}

\newcommand\gam{\gamma}

\newcommand\lam{\lambda}
\newcommand\del{\delta}

\newcommand\yo{\overline{y}}

\newcommand\Wh{\widehat{W}}
\newcommand\wh{\widehat{w}}

\newcommand\Vh{\widehat{V}}

\newcommand\Act{\widetilde{\Ac}}

\newcommand\Ct{\widetilde{C}}
\newcommand\ct{\tilde{c}}

\newcommand\htt{\widetilde{h}}

\newcommand\rt{\widetilde{r}}

\newcommand\Xt{\widetilde{X}}

\newcommand\Zt{\widetilde{Z}}

\newcommand\ut{\widetilde{u}}

\newcommand\Vt{\widetilde{V}}

\newcommand\wt{\tilde{w}}

\newcommand\deltat{\tilde{\del}}

\newcommand\dd{\mathrm{d}}
\newcommand\ee{\mathrm{e}}

\newcommand{\cs}{{c^*}}
\newcommand{\cso}{c^*}

\newcommand\xu{x_{s}}
\newcommand{\xs}{{x_{\alpha}}}
\newcommand{\xa}{{x_{\alpha}}}
\newcommand{\xH}{x_\text{h}}

\newcommand{\ys}{{y_{\alpha}}}

\newcommand \al {\alpha}
\newcommand \vp {v}
\newcommand \vpt {v}

\providecommand{\keywords}[1]{\textbf{\textit{Keywords: }} #1}

\begin{document}

\title{Optimal Consumption under a Habit-Formation\\Constraint: the Deterministic Case}

\author{
	Bahman Angoshtari
	\thanks{Department of Mathematics, University of Miami.  \textbf{e-mail}: \url{b.angoshtari@math.miami.edu}}
	\and
	Erhan Bayraktar
	\thanks{Department of Mathematics, University of Michigan.  \textbf{e-mail}: \url{erhan@umich.edu} E. Bayraktar is supported in part by the Susan M. Smith Professorship.}
	\and
	Virginia R. Young
	\thanks{Department of Mathematics, University of Michigan.  \textbf{e-mail}: \url{vryoung@umich.edu} V. R. Young is supported in part by the Cecil J. and Ethel M. Nesbitt Professorship.}
}

\date{This version: \today}

\maketitle

\begin{abstract}
	We formulate and solve a deterministic optimal consumption problem to maximize the discounted CRRA utility of an individual's consumption-to-habit process assuming she only invests in a riskless market and that she is unwilling to consume at a rate below a certain proportion $\alpha\in(0,1]$ of her consumption habit. Increasing $\alpha$, increases the degree of addictiveness of habit formation, with $\alpha=0$ (respectively, $\al=1$) corresponding to non-addictive (respectively, completely addictive) model. 	
	We derive the optimal consumption policies explicitly in terms of the solution of a nonlinear free-boundary problem, which we analyze in detail. Impatient individuals (or, equivalently, those with more addictive habits)  always consume above the minimum rate; thus, they eventually attain the minimum wealth-to-habit ratio.  Patient individuals (or, equivalently, those with less addictive habits) consume at the minimum rate if their wealth-to-habit ratio is below a threshold, and above it otherwise.  By consuming patiently, these individuals maintain a wealth-to-habit ratio that is greater than the minimum acceptable level. Additionally, we prove that the optimal consumption path is hump-shaped if the initial wealth-to-habit ratio is either: (1) larger than a high threshold; or (2) below a low threshold and the agent is more risk seeking (that is, less risk averse). Thus, we provide a simple explanation for the consumption hump observed by various empirical studies.
\end{abstract}


%-----------------------------------------------------------------------------------
%
%       SECTION: 		Introduction
%
%-----------------------------------------------------------------------------------

\section{Introduction}

It has been long known that the classical time-separable preferences of \cite{Merton1969} is at odds with empirical observations; see \cite{GrossmanShiller1980}, \cite{mehra1985equity}, and the references therein.  To address these shortcomings, researchers developed model of habit-formation models in the late 1960's; see, for example, \cite{Pollak1970} and \cite{RyderHeal1973}.  In these models, time-inseparability is introduced through an explicit dependence of the consumption utility function on the consumption habit, which is defined as a running average of past consumption. \cite{sundaresan1989} solved an infinite-horizon, optimal investment and consumption problem assuming a habit-formation power utility function and a geometric Brownian motion price process; he demonstrated the smoothness of the optimal consumption process relative to that of \cite{Merton1969}.  \cite{Constantinides1990} generalized the specification of the consumption habit process and provided an explanation for the equity premium puzzle. Under a more general habit-formation utility and market model, \cite{DetempleZapatero1991} and  \cite{DetempleZapatero1992} derive conditions under which optimal policies exist and characterize the optimal consumption policy in terms of an unknown stochastic process that arises from the martingale representation theorem.

The above studies largely assume \emph{addictive} habit formation, in the sense that they explicitly or implicitly assume that the individual is unwilling to consume at a rate below her consumption habit.\footnote{Our definition of addictive and nonaddictive models follows \cite{DetempleZapatero1991} (see Example 3.2 on page 1639) and \cite{DetempleKaratzas2003} (see top of page 266).}  \cite{DetempleKaratzas2003} adapted a habit-formation specification that allows for \emph{non-addictive} habit where consumption can fall below the individual's consumption habit. For more recent studies on continuous-time optimal consumption models with habit-formation preferences see \cite{MUNK2008}, \cite{EnglezosKaratzas2009}, \cite{Muraviev2011}, \cite{2015Yu}, and \cite{KraftMunkSeirfriedWagner2017}, among others.

We consider an infinite-horizon, optimal consumption problem for an individual who forms a consumption habit. The novelty of our approach is in introducing a consumption habit-formation constraint that controls the level of addictiveness of the habit-formation mechanism. In particular, we assume that the individual is unwilling to consume at a rate that is below a certain proportion $\al\in(0,1]$ of her consumption habit. Setting $\al=0$ (resp.\ $\al=1$) leads to a non-addictive (resp.\ addictive) habit formation. Choosing $\al\in(0,1)$ leads to partially addictive models, with the level of addictiveness increasing in $\al$. In contrast to the existing habit-formation literature, our constraint cannot be incorporated in the objective function through infinite marginal utility, and alters the analysis of the corresponding optimal control problem.

We assume the individual funds her consumption solely through risk-free investment.  To avoid bankruptcy, we show that the wealth-to-habit ratio must always be above a certain ``\emph{safe level}'' $\xu$. We, then, formulate and solve a deterministic control problem to maximize the discounted CRRA utility of the consumption-to-habit process. We show that there exists a threshold $\xa$ such that if the ratio of wealth-to-habit is above (resp.\ below) $\xa$, it is optimal to consume at a rate greater than (resp.\ equal to) the minimum acceptable rate imposed by the habit-formation constraint. Furthermore, the individual optimally consumes in such a way that her wealth-to-habit ratio attains a specific value. We find a significant difference between impatient individuals (those whose utility discount rate is above a certain threshold) and patient individuals (those with utility discount rate below the said threshold). Impatient individuals always consume above the minimum rate (that is, $\xa = \xu$) and, thereby, eventually attain the minimum wealth-to-habit ratio $\xu$, while patient individuals might consume at the minimum rate (that is, $\xa > \xu$) and, thereby, attain a wealth-to-habit ratio greater than the minimum acceptable level. This patient vs.\ impatient dichotomy can also be interpreted as high-addictive habits vs.\ low-addictive habits. In particular, an impatient (resp.\ patient) individual has an $\al$ that is above (resp.\ below) a certain threshold, and thus has a more (resp.\ less) addictive habit.  On the technical side, we obtain explicit results in terms of the solution of a nonlinear free-boundary problem, which we analyze in detail.

Various empirical studies indicate that consumption spending of individuals usually have a hump-shaped pattern, with spending typically increasing until the age of about 50 and then decreasing; see \cite{Thurow1969} for an early study and \cite{fernandezVillaverdeKrueger2007} for a more recent one. In a recent article, \cite{KraftMunkSeirfriedWagner2017} provided a theoretical justification for the consumption hump through an agent's habit formation. They considered a finite-horizon, optimal consumption model with addictive habit formation, in which the agent only invests in a risk-free market and with the objective of maximizing the discounted CRRA utility of the difference between consumption and the consumption habit. They derived the optimal consumption policy in closed form and provided sufficient conditions for the presence of a consumption hump in the asymptotic case of large investment horizon.

We compliment the study of \cite{KraftMunkSeirfriedWagner2017} by providing necessary and sufficient conditions for the presence of a consumption hump in our model; see Proposition \ref{prop:inverseU} below.  As in their paper, we show that a consumption hump can exist only if the individual's subjective utility discount rate is higher than the risk-free rate. We find that a consumption hump exists if the individuals initial wealth-to-habit is higher than a specific threshold, which we characterize as the solution of a certain algebraic equation (see \eqref{eq:xH} below).  Additionally, we find that a consumption hump can occur for individuals with low wealth-to-habit ratio (of around the level $\xa$ mentioned above), but only if their risk aversion rate is low (more specifically, their risk-aversion must be below that of a log-utility investor). Thus, our justification for a consumption hump is either: (1) excess initial wealth (relative to habit); or (2) lack of wealth and low risk aversion. Our first condition (with high wealth-to-consumption ratio) is similar to the condition provided by \cite{KraftMunkSeirfriedWagner2017}, while the second scenario (with low wealth-to-consumption ratio and risk aversion) is new.  

To the best of our knowledge, our paper is the first that incorporates the degree of addictiveness of habit formation via an admissibility constraint that cannot be incorporated in the objective function through infinite marginal utility.  It should be mentioned that there is a related literature on optimal consumption models with ratcheting and drawdown constraints; see, \cite{Dybvig1995}, \cite{ElieTouzi2008}, \cite{jeon2018portfolio}, \cite{2019Roche},  \cite{AngoshtariBayraktarYoung2019}, \cite{AlbrecherAzcueMuler2020}, and  \cite{AlbrecherAzcueMuler2022} among others.  In these studies, the individual is forbidden to consume below a proportion $\al$ of the running maximum of her past consumption. In parallel to addictive and non-addictive habit formation, the case $\al=1$ corresponds to the ratcheting constraint, while $\al\in(0,1)$ represents the drawdown constraint. There is, however, significant differences between the above studies and ours. Indeed, in contrast to habit formation based on average past consumption, drawdown and ratcheting constraints represent a severe form of habit formation for which the effect of past consumption does not ``fade away'' with time, and the habit process cannot be reduced by lowering the rate of consumption.  Furthermore, drawdown and ratcheting constraints lead to singular control, while our setting leads to regular control.

In a companion article \cite{AngoshtariBayraktarYoung2021}, we extend our model to the case when the agent invests in a risky asset as well as the risk-free asset, which leads to a stochastic control problem. The results presented herein for the deterministic case is not a special case of our other paper, however.  On the technical side, the analysis of the stochastic control problem relies on randomness of the model and degenerates once the risky asset is removed. Furthermore, the results presented here rely on analysis of a single ODE, while the stochastic case reduces to a coupled system of first-order ODEs with a free boundary, whose analysis is much more intricate. Thus, our deterministic model in the current paper is more tractable and amenable to extensions such as equilibrium modeling. On the economic side, we don't see structural differences between patient and impatient individual in the stochastic case, in the sense that, for the optimal consumption policy, we always have $\xa>\xu$ (that is, the individual consumes patiently) regardless of the value of the utility discount rate. Finally, our goal in the current paper is to explain the consumption hump, while our stochastic model explains the equity premium puzzle.

The rest of the paper is organized as follows. In Section \ref{sec:2}, we introduce the consumption habit process, derive its basic properties, and define our optimal consumption problem. In Section \ref{sec:3}, we formulate the Hamilton-Jacobi-Bellman (HJB) free-boundary-problem and solve it semi-explicitly by applying the Legendre transform. This section includes the main result of the paper, namely, Theorem \ref{thm:VF-riskless}, in which we verify that the solution of the HJB free-boundary-problem yields the value function and the optimal consumption policy. Furthermore, Proposition \ref{prop:inverseU} provides necessary and sufficient conditions for the presence of a consumption hump. In Subsection \ref{sub:Log}, we derive the optimal policy for logarithmic utility and, in Subsection \ref{sub:Income}, we investigate the effect of adding income at a constant rate. In Section \ref{sec:4}, we illustrate the optimal consumption and wealth process and their sensitivity to some of the model parameters through several numerical examples. We conclude in Section \ref{sec:5}. Longer proofs are included in the appendices.

%-----------------------------------------------------------------------------------
%
%       SECTION: 		Market Model
%
%-----------------------------------------------------------------------------------

\section{Problem formulation}\label{sec:2}

We consider an individual who invests in a riskless asset with a fixed interest rate $r>0$ and who consumes in order to maximize her utility of lifetime consumption. Let $C(t)\ge 0$ denote the individual's consumption rate at time $t\ge0$, such that $\int_0^t C(u)\dd u$ is the total consumption over the time interval $[0,t]$. Then, her wealth process $W = \{W(t)\}_{t\ge0}$ satisfies
\begin{align}\label{eq:wealth}
	\frac{\dd W(t)}{\dd t} &= r\,W(t) - C(t),
\end{align}
for $t\ge0$, with $W(0)=w>0$.

For a given consumption process $C = \{C(t)\}_{t \ge 0}$, we define the individual's \emph{habit process} (that is, consumption habit) as the process $Z = \{Z(t)\}_{t \ge 0}$ given by
\begin{align}\label{eq:Z}
	Z(t) = \ee^{-\rho\,t}\left(z + \int_0^{t} \rho\,\ee^{\rho\,u} C(u) \dd u\right);\quad t\ge0,
\end{align}
which has the following equivalent differential form,
\begin{align}\label{eq:Z2}
	\begin{cases}
		\frac{\dd Z(t)}{\dd t} = -\rho\big(Z(t) - C(t)\big);\quad t\ge0,\\
		Z(0)=z.
	\end{cases}
\end{align}
Here, $\rho>0$ is a constant, and $z>0$ represents the initial consumption habit of the individual.  The parameter $\rho$ determines how much current habit is influenced by the recent rate of consumption relative to the consumption rate farther in the past.  As $\rho$ increases, more weight is given to recent consumption.  In the limiting cases, $\rho = 0$ implies $Z(t) = z$ for all $t \ge 0$, and $\rho = \infty$ implies $Z(t) = C(t)$ for all $t \ge 0$.

For $t>0$, the consumption habit $Z(t)$ given by \eqref{eq:Z} is the exponentially-weighted moving average of past consumption $C(s)$, $s<t$. To see this, assume the individual lived (and consumed) over the time period $(-\infty, t)$. Let $z$ be the exponentially-weighted average of her consumption rate before time zero, that is, $z = \int_{-\infty}^0 \rho\,\ee^{\rho u} C(u) \dd u$.
(Note that $\int_{-\infty}^0 \rho \, \ee^{\rho u} \dd u = 1$.)  By substituting for $z$ in \eqref{eq:Z}, we obtain
\begin{align}
	Z(t) &= \int_{-\infty}^0 \rho\,\ee^{-\rho(t-u)} C(u)\dd u + \int_0^{t} \rho\,\ee^{-\rho(t-u)} C(u)\dd u\\
	&= \int_{-\infty}^{t} \rho\,\ee^{-\rho(t-u)} C(u)\dd u,
\end{align}
with $\int_{-\infty}^{t} \rho\,\ee^{-\rho(t-u)} \dd u = 1$.
Thus, $Z(t)$ is the exponentially-weighted moving average of $C(s)$, $s<t$, as claimed.

We consider a \emph{consumption habit formation} for the individual by assuming that, at any time $t\ge0$, she is unwilling to consume at a rate that is below a certain proportion of her  habit $Z(t)$. In particular, we impose the following constraint on the individual's consumption process
\begin{align}\label{eq:Habit}
	C(t) \ge \al\,Z(t);\quad t\ge0,
\end{align}
in which $\al \in (0,1]$ is a constant that measures the individual's tolerance for her current consumption to drop below her habit.  The larger the value of $\al$, the less tolerant the individual is in allowing her current consumption to fall below her habit. Note that the consumption habit process $\{Z(t)\}_{t\ge0}$, depends on $z$ and on the consumption process $\{C(t)\}_{t\ge0}$. To ease the notational burden, however, we write $Z(t)$ instead of the more accurate notation $Z_{z,C(\cdot)}(t)$. 

We assume that the individual consumes in such a way to avoid bankruptcy.  The following lemma provides the corresponding necessary and sufficient condition, namely, that the wealth-to-habit ratio must be above a ``\emph{safe level}'' $\xu$ given by
\begin{align}\label{eq:xu}
	\xu=\xu(\al):=\frac{\al}{r+\rho(1-\al)}\,,
\end{align}
for $\al\in[0,1]$. Note that $\xu(\al)$ is strictly increasing in $\al$, $\xu(0)=0$, and $\xu(1)=1/r$.  

\begin{lemma}\label{lem:NoRuin}
	Assume that $C:\Rb_+\to \Rb_+$ is a measurable function satisfying \eqref{eq:Habit}, in which $\{Z(t)\}_{t\ge0}$ is given by \eqref{eq:Z}. Define the wealth process $\{W(t)\}_{t\ge0}$ by \eqref{eq:wealth}. Then, $W(t)>0$ for all $t\ge0$ if and only if
	\begin{align}\label{eq:NoRuin}
		\frac{W(t)}{Z(t)} \ge \xu,
	\end{align}
	for all $t\ge0$.
\end{lemma}
\begin{proof}
	See Appendix \ref{app:NoRuin}.
\end{proof}

We can interpret \eqref{eq:NoRuin} by observing how it changes with respect to the parameters $\al$, $\rho$, and $r$.  First, $\xu$ increases with $\al$, which means that to avoid ruin, the individual's wealth-to-habit ratio needs to be larger with increasing $\al$.  This relationship makes sense because as $\al$ increases, the individual is less tolerant about allowing her current consumption to fall relative to her habit.   Second, $\xu$ decreases with increasing $\rho$, and increasing $\rho$ means that more weight is given to recent consumption in measuring the habit.  Thus, as $\rho$ increases, past consumption has less effect on current consumption via the habit, and the wealth-to-habit ratio does not need to be as large to avoid ruin.  Third, $\xu$ decreases with increasing $r$, and increasing $r$ means that the individual can earn more money in the riskless asset to fund her consumption; thus, it makes sense that increasing $r$ implies that the wealth-to-habit ratio does not need to be as large to avoid ruin.

Inequality \eqref{eq:NoRuin} implies that the highest initial consumption habit that the individual can \emph{afford} with an initial wealth $w$ is $z = w/\xu$.  Equivalently, \eqref{eq:NoRuin} tell us that the minimum initial wealth that the individual \emph{needs} to afford an initial consumption habit of $z$ is $w = \xu\, z$. In other words, \eqref{eq:NoRuin} characterizes \emph{affordable} consumption habits given the individual's wealth.

Note that as $\al \to 0^+$, the requirement for consumption \eqref{eq:Habit} becomes $C(t) \ge 0$, and inequality \eqref{eq:NoRuin} becomes moot, which we expect because this limiting case is the setting considered by \cite{Merton1969}.  It is also worth noting that, in the special case of $\al = 1$, the requirement for consumption \eqref{eq:Habit} becomes $C(t) \ge Z(t)$, and inequality \eqref{eq:NoRuin} becomes $rW(t) \ge Z(t)$, which is consistent with the feasibility condition adapted by \cite{Dybvig1995}, namely, that $rW(t) \ge C(t-)$. Note, also, that although both of the aforementioned studies consider risky investment in addition to the riskless investment, their no-bankruptcy conditions compares with ours because these conditions are derived using arguments that rely solely on riskless investments.

We define the set of \emph{admissible} investment and consumption policies as those that avoid bankruptcy while satisfying the individual's consumption habit-formation constraint.

\begin{definition}\label{def:Admissible0}
	Let $\widetilde{\Ac}(w,z)$ be the set of all measurable functions $C:\Rb_+\to \Rb_+$ such that conditions \eqref{eq:Habit} and \eqref{eq:NoRuin} hold, namely, $C(t) \ge \al Z(t)$, and $W(t)\ge \xu Z(t)$ for all $t\ge0$, in which $W$ and $Z$ are given by \eqref{eq:wealth} and \eqref{eq:Z}, respectively. \qed
\end{definition}

Next, we formulate the individual's lifetime consumption and investment problem as a control problem. For any admissible consumption policy $C\in \Act$, let us introduce the \emph{wealth-to-habit} process
\begin{align}\label{eq:X}
	X(t) := \frac{W(t)}{Z(t)};\quad t\ge0,
\end{align}
and note that, by \eqref{eq:wealth} and \eqref{eq:Z2},
\begin{align}\label{eq:X-riskless}
	X(t) = x + \int_0^t\Big[(r+\rho)X(u) - \big(1+\rho X(u)\big) c(u)\Big]\dd u;\quad t\ge0,
\end{align}
in which we have defined the \emph{consumption-to-habit} process $c = \{c(t)\}_{t \ge 0}$ by $c(t) :=C(t)/Z(t)$. We define the set of \emph{admissible} consumption-to-habit policies as follows.
\begin{definition}\label{def:Admissible2}
	Let $\Ac(x)$ be the set of all measurable functions $c:\Rb_+\to[\al,+\infty)$ such that $X(t) \ge \xu$ for all $t\ge0$, in which $X(t)$ is given by \eqref{eq:X-riskless}.\qed
\end{definition}

As the following proposition states, our two definitions of admissible policies are equivalent in the sense that any admissible consumption policy corresponds to an admissible relative consumption policy and vice versa. Its proof is elementary and, thus, omitted.

\begin{proposition}\label{prop:Admiss}
	Let $w,z > 0$ be the initial wealth and habit, respectively, and define $x:=w/z$. Assume that $C \in \Act(w,z)$ and let $Z$ be given by \eqref{eq:Z}. Then, we have  $c:=C/Z\in\Ac(x)$.
	Conversely, assume that $c\in\Ac(x)$, and let $W$ be the solution of
	\begin{align}
		\begin{cases}\displaystyle
			\frac{\dd W(t)}{\dd t} = W(t)\left(r - \frac{c(t)}{X(t)}\right);\quad t\ge0,\\
			W(0)=w,
		\end{cases}
	\end{align}
	in which $X$ is given by \eqref{eq:X-riskless}. We, then, have $C:= c W / X \in \Act(w,z)$.   \qed
\end{proposition}
%\begin{proof}
%	\red{To be added}.
%\end{proof}

We assume that the individual values her consumption relative to her habit. In particular, for a given consumption process $C$, the expected utility of her lifetime consumption is given by\footnote{Such a multiplicative habit-formation preference is common in the literature. See, for instance, \cite{Abel1990}. See page 322 of \cite{KraftMunkSeirfriedWagner2017} for a more complete list of references.}
\begin{align}\label{eq:obj}
	\Eb\left(\int_0^{\tau_d} \frac{1}{1-\gam}\left(\frac{C(t)}{Z(t)}\right)^{1-\gam}\, \ee^{-\deltat\,t}\,\dd t\right)
	= \int_0^{+\infty} \frac{1}{1-\gam}\left(\frac{C(t)}{Z(t)}\right)^{1-\gam}\, \ee^{-(\deltat + \lam)\,t}\,\dd t,
\end{align}
in which $\deltat>0$ is the individual's subjective time preference, $\gam>0$ (with $\gam \ne 1$) is her (constant) relative risk aversion, and $\tau_d$ is the random time of her death, which we assume is exponentially distributed with mean $1/\lam > 0$.  %, independent of the Brownian motion.
In light of Proposition \ref{prop:Admiss}, the individual's optimal investment-consumption problem is, thus, formulated by the following control problem:
\begin{align}\label{eq:VF-riskless}
	V(x) := \sup_{c(\cdot)\in\Ac(x)} \int_{0}^{+\infty}\ee^{-\del t} \frac{\big(c(t)\big)^{1-\gam}}{1-\gam}\, \dd t;\quad x\ge \xu.
\end{align}
in which $\del = \deltat + \lam$.

\section{The optimal consumption policy}\label{sec:3}

The Hamilton-Jacobi-Bellman (HJB) equation corresponding to the control problem in \eqref{eq:VF-riskless} is as follows
\begin{align}\label{eq:HJB-riskless}
	-\del \vpt(x) + (r+\rho) x \vpt'(x) + \sup_{c\ge\al} \left\{\frac{c^{1-\gam}}{1-\gam} - (1+\rho x)c \vpt'(x)\right\} = 0;\quad x\ge \xu.
\end{align}
In the rest of this section, we construct a classical solution of this differential equation; then, in the proof of Theorem \ref{thm:VF-riskless}, we verify that the constructed solution equals the value function $V$ in \eqref{eq:VF-riskless}. In Subsection \ref{sub:InverseU}, we provide sharp conditions for presence of consumption hump in our model. In Subsection \ref{sub:Log}, we solve the case of logarithmic utility function. In Subsection \ref{sub:Income}, we investigate the effect of adding a constant rate of income to our model. 

To construct a candidate solution, we hypothesize that  the optimal consumption policy has the following form. There exists a critical level of wealth-to-habit ratio $\xs\ge\xu$ such that:
\begin{itemize}
	\item[$(a)$] If $\xu\le X(t)\le\xs$, it is optimal to consume at the minimum rate $c(t) = \al$,
	
	\item[$(b)$] If $X(t)> \xs$, it is optimal to consume more than the minimum rate.
\end{itemize}

Next, we drive a set of conditions by assuming that a solution $v(x)$ of \eqref{eq:HJB-riskless} is consistent with the ansatz $(a)$ and $(b)$ above. Note that the optimizer $c^*$ in \eqref{eq:HJB-riskless} is given by\footnote{Here, we are assuming that $v'(x)>0$, which is verified by Proposition \ref{prop:HJB-riskless}.}
\begin{align}\label{eq:cs-cand-riskless}
	\cso(x) :=
	\begin{cases}
		\displaystyle
		\al;&\quad (1+\rho x)\vpt'(x) \ge \al^{-\gam},\vspace{1ex}\\
		\displaystyle
		\Big((1+\rho x) \vpt'(x) \Big)^{-\frac{1}{\gam}};&\quad 0<(1+\rho x)\vpt'(x)< \al^{-\gam}.
	\end{cases}
\end{align}
For ansatz $(a)$ and $(b)$ to be true, we must have
\begin{align}\label{eq:VE-riskless}
	\begin{cases}
		(1+\rho x)\vpt'(x) \ge \al^{-\gam};&\quad \xu\le x\le \xs,\vspace{1ex}\\
		0<(1+\rho x)\vpt'(x) < \al^{-\gam};&\quad x>\xs.
	\end{cases}
\end{align}
Under these conditions, \eqref{eq:HJB-riskless} becomes the free-boundary problem (FBP)
\begin{align}\label{eq:FBP-riskless}
	\begin{cases}
		\displaystyle
		-\al\left(\frac{x}{\xu}-1\right) \vpt'(x)  + \del \vpt(x) = \frac{\al^{1-\gam}}{1-\gam};
		&\quad \xu\le x \le \xs,\vspace{1ex}\\
		\displaystyle
		- (r+\rho) x \vpt'(x) + \del \vpt(x) = \frac{\gam}{1-\gam}\big((1+\rho x)\vpt'(x)\big)^{1-\frac{1}{\gam}};
		&\quad x > \xs,\vspace{1ex}\\
		(1+\rho\xs)\vpt'(\xs) = \al^{-\gam},
	\end{cases}
\end{align}
in which $\xs\ge \xu$ is an unknown free boundary.

\begin{remark}
	It is possible to directly provide the solution of the FBP \eqref{eq:FBP-riskless}. Here, we have chosen an indirect approach through the Legendre transform \eqref{eq:convexcon} below. We have three reasons for doing this. Firstly, the analysis of \eqref{eq:FBP-riskless} is more natural after applying the Legendre transform, since the verification argument relies on several properties of the solution which are expressed in terms of the derivative of the value function at certain boundaries. For instance, for the case $0<\del<r+\rho(1-\al)$, we have $V'(x_s^+)=+\infty$; while  $V'(x_s)=\al^{-\gam}/(1+\rho\xu)$ if $\del>r+\rho(1-\al)$. By applying the Legendre transform, the derivative $V'(x)=y$ becomes the independent variable and it will be easier to see these conditions. By working directly with FBP \eqref{eq:FBP-riskless}, one still needs to establish these additional properties which will require equivalent argument as those presented for the dual value function below. Secondly, our current arguments are parallel to the arguments in our companion work \cite{AngoshtariBayraktarYoung2021}. Therein, using the Legendre transform is necessary for linearizing the terms involving the second derivative of the value function. Our current presentation makes it easier to compare the two papers. Finally, we have found that the auxiliary ODEs \eqref{eq:y-ODE-impatient} and \eqref{eq:y-FBP} are numerically more stable than their counterparts obtained by directly working with \eqref{eq:FBP-riskless}. Although there are ways to properly deal with the numerical instability of such ODEs, this further highlights the fact that applying the Legendre transform is appropriate here.\qed
\end{remark}

To solve FBP \eqref{eq:FBP-riskless}, we define the convex conjugate $u$ given by
\begin{align}\label{eq:convexcon}
	u(y) := \sup_{x\ge\xu}\big\{v(x)-xy\big\};\quad 0<y\le \yo := \lim_{x\to\xu^+} \vp'(x)\in\Rb\cup\{+\infty\},
\end{align}
in which we have assumed that $\vpt$ is strictly increasing and concave, an assumption that will be verified in Proposition \ref{prop:HJB-riskless}. Assume that $I(\cdot)$ is the inverse of $v'(\cdot)$, that is, $v'\big(I(y)\big) = y$, $0<y\le \yo$. We, then, have
\begin{align}\label{eq:Legendre}
	v\big(I(y)\big) = u(y) - y u'(y),\quad I(y) = -u'(y),\quad\text{and}\quad v''\big(I(y)\big)=-\frac{1}{u''(y)},
\end{align}
for $0<y\le\yo$. By using these relationships and substituting $x = -u'(y)$, the FBP \eqref{eq:FBP-riskless} becomes the following FBP:
\begin{align}
	&\displaystyle\label{eq:FBPDual1-riskless}
	\big(r+\rho(1-\al)-\del\big)y u'(y) + \del u(y) = \frac{\al^{1-\gam}}{1-\gam} -\al y;
	\qquad \ys\le y \le \yo,\vspace{1ex}\\
	&\displaystyle\label{eq:FBPDual2-riskless}
	\big(r+\rho -\del\big)y u'(y) + \del u(y) = \frac{\gam}{1-\gam}\big(y-\rho y u'(y)\big)^{1-\frac{1}{\gam}};
	\qquad 0< y < \ys,\vspace{1ex}\\
	&\displaystyle\label{eq:FBPDual3-riskless}
	\lim_{y\to\yo^{\,-}} u'(y) = -\xu,&\vspace{1ex}\\
	\intertext{and}
	\label{eq:FBPDual4-riskless}
	&\ys-\rho \ys u'(\ys) = \al^{-\gam},
\end{align}
in which $\yo=\lim_{x\to\xu^-}v'(x)$ and $\ys=v'(\xs)$ are unknown free boundaries. Here, we include the possibility of $\yo=+\infty$. Our goal is to find a decreasing and strictly convex $u$ satisfying \eqref{eq:FBPDual1-riskless}--\eqref{eq:FBPDual4-riskless} which, by inverting \eqref{eq:convexcon}, yields an increasing and strictly concave $v$ satisfying \eqref{eq:FBP-riskless}.  Specifically, $v$ is given by $v(x) = u(y) - y u'(y)$, in which $0 < y \le \yo$ uniquely solves $x = -u'(y)$ for $x \ge \xu$.

Note that if $\del=r+\rho(1-\al)$, then, by \eqref{eq:FBPDual1-riskless},
\begin{align}
	u(y) =\frac{\al^{1-\gam}}{\del(1-\gam)} -\frac{\al}{\del} y;
	\qquad \ys< y \le \yo.
\end{align}
However, this $u$ is not strictly convex and contradicts \eqref{eq:convexcon} unless $\ys=\yo$ (meaning that \eqref{eq:FBPDual1-riskless} is moot).  If $\ys=\yo$, then \eqref{eq:FBPDual3-riskless} and \eqref{eq:FBPDual4-riskless} yield $\ys=\al^{-\gam}/(1+\rho\xu)$.  In the following proposition, we show that, if $\del\ge r+\rho(1-\al)$, there exists a decreasing and convex solution of \eqref{eq:FBPDual2-riskless}--\eqref{eq:FBPDual4-riskless} with $\yo = \ys=\al^{-\gam}/(1+\rho\xu)$.

\begin{proposition}\label{prop:u-solution-impatient}
	Assume $\del\ge r+\rho(1-\al)$. Then:
	
	\noindent $(i)$ There is a strictly increasing function $y:(0,\al^{-\gam}]\to\big(0,\al^{-\gam}/(1+\rho\xu)\big]$ satisfying
	\begin{align}\label{eq:y-ODE-impatient}
		\begin{cases}\displaystyle
			y'(\psi) = \frac{\frac{\rho}{r+\rho}\left(\frac{r+\rho-\del}{\rho} -\psi^{-\frac{1}{\gam}}\right)y(\psi)}{y(\psi)-\frac{\del}{r+\rho}\psi};\quad 0<\psi\le\al^{-\gam}, \vspace{0.5em} \\
			y\left(\al^{-\gam}\right) = \dfrac{\al^{-\gam}}{1+\rho\xu}.
		\end{cases}
	\end{align}
	Furthermore, $0<y(\psi)<\del\psi/(r+\rho)$ for $0<\psi<\al^{-\gam}$.
	
	\noindent $(ii)$ A strictly decreasing and strictly convex solution of the FBP \eqref{eq:FBPDual2-riskless}--\eqref{eq:FBPDual4-riskless} is given by $\yo=\ys = \frac{\al^{-\gam}}{1+\rho\xu}$, and
	\begin{align}\label{eq:u-impatient}
		u(y) &= \frac{1}{\del}\left[\frac{\gam}{1-\gam}\big(\psi(y)\big)^{1-\frac{1}{\gam}}+\frac{r+\rho-\del}{\rho}\big(\psi(y)-y\big)\right],
	\end{align}\vspace{1ex}
	for $0<y\le \frac{\al^{-\gam}}{1+\rho\xu}$, in which $\psi = \psi(y)$ is the $($strictly increasing$)$ inverse of $y = y(\psi)$ in $(i)$.
\end{proposition}

\begin{proof}
	See Appendix \ref{app:u-solution-impatient}.
\end{proof}

Next, we consider the FBP \eqref{eq:FBPDual1-riskless}--\eqref{eq:FBPDual4-riskless} when $0<\del<r+\rho(1-\al)$. For this case, we find that $0<\ys<\yo=+\infty$. The following proposition provides the solution for this case.

\begin{proposition}\label{prop:u-solution-patient}
	Assume $0<\del< r+\rho(1-\al)$. Define the constants $\psi_0:=\left(\frac{r+\rho-\del}{\rho}\right)^{-\gam}\in(0,\al^{-\gam})$ and $y_0:=\frac{\del\psi_0}{r+\rho} \in \big(0,\frac{\al^{-\gam}}{1+\rho\xu} \big)$. Then:\vspace{1ex}
	
	\noindent $(i)$ There exist a constant $\ys \in \big(y_0,\frac{\al^{-\gam}}{1+\rho\xu} \big)$ and a strictly increasing function $y:(0,\al^{-\gam}]\to\big(0,\ys\big]$ satisfying\footnote{This statement is non-trivial since we are looking for a \textbf{strictly increasing} solution on $(0,\al^{-\gam}]$. See Figure \ref{fig:y_pat}.}
	\begin{align}\label{eq:y-FBP}
		\begin{cases}\displaystyle
			y'(\psi) = \frac{\frac{\rho}{r+\rho}\left(\frac{r+\rho-\del}{\rho} -\psi^{-\frac{1}{\gam}}\right)y(\psi)}{y(\psi)-\frac{\del}{r+\rho}\psi};\quad 0<\psi\le\al^{-\gam},\\
			y\left(\al^{-\gam}\right) = \ys.
		\end{cases}
	\end{align}
	Furthermore, $\max\left(0, \psi-\frac{\rho}{r+\rho}\psi^{1-\frac{1}{\gam}}\right) < y(\psi) < \frac{\del}{r+\rho}\psi$ for $0<\psi<\psi_0$, and $\frac{\del}{r+\rho}\psi < y(\psi) < \psi-\frac{\rho}{r+\rho}\psi^{1-\frac{1}{\gam}}$ for $\psi_0<\psi<\al^{-\gam}$.
	
	\noindent $(ii)$ A strictly decreasing and strictly convex solution of the FBP \eqref{eq:FBPDual1-riskless}--\eqref{eq:FBPDual4-riskless} is given by $\yo=+\infty$, $\ys$ as in $(i)$,
	\begin{align}\label{eq:u_patient_sol1}
		u(y) = \frac{r+\rho(1-\al) - \del}{\del\rho}\big(\al^{-\gam} - \ys(1+\rho\xu)\big)
		\left(\frac{y}{\ys}\right)^{-\frac{\del}{r+\rho(1-\al)-\del}}-\xu y+\frac{\al^{1-\gam}}{\del(1-\gam)};\quad y>\ys,
	\end{align}
	and
	\begin{align}\label{eq:u_patient_sol2}
		u(y) = \frac{\gam}{\del(1-\gam)}\big(\psi(y)\big)^{1-\frac{1}{\gam}}
		+\frac{r+\rho-\del}{\rho\del}\big(\psi(y)-y\big);\quad 0<y\le \ys,
	\end{align}\vspace{1ex}
	in which $\psi = \psi(y)$ is the $($strictly increasing$)$ inverse of $y = y(\psi)$ in $(i)$.
\end{proposition}

\begin{proof}
	See Appendix \ref{app:u-solution-patient}.
\end{proof}

Propositions \ref{prop:u-solution-impatient} and \ref{prop:u-solution-patient} yield a decreasing and convex solution $\big(\yo, \ys, u(y)\big)$ of the FBP \eqref{eq:FBPDual1-riskless}--\eqref{eq:FBPDual4-riskless}. By reversing the transformation in \eqref{eq:convexcon}, we obtain an increasing and concave solution $\big(\xs, v(x)\big)$ of the FBP \eqref{eq:FBP-riskless}. We state this result as the following proposition.

\begin{proposition}\label{prop:HJB-riskless}
	Let $\yo$, $\ys$, $\psi$, and $u$ be as in Proposition {\rm \ref{prop:u-solution-impatient}} $($if $\del\ge r+\rho(1-\al))$ or Proposition {\rm \ref{prop:u-solution-patient}} $($if $0<\del<r+\rho(1-\al))$. Let $J:(-\infty, -\xu)\to(0,\yo)$ be the inverse of $u'$, that is, $u'\big(J(\xi)\big)=\xi$ for $\xi<-\xu$. Define
	\begin{align}
		\xs &:= -u'(\ys) = \frac{\al^{-\gam}}{\rho\ys}-\frac{1}{\rho},\\
		\vpt(x) &:= u\big(J(-x)\big) + x J(-x);\quad x>\xu,\\
		\vpt(\xu) &= \lim_{x\to\xu^+} \Big(u\big(J(-x)\big) + x J(-x)\Big),
	\end{align}
	and
	\begin{align}\label{eq:CS-sol-riskless}
		\cso(x) :=
		\begin{cases}
			\al; &\quad \xu\le x\le \xs,\vspace{1ex}\\
			\displaystyle
			\Big(\psi\big(J(-x)\big)\Big)^{-\frac{1}{\gam}};&\quad x>\xs.
		\end{cases}
	\end{align}
	Then, $\xs$, $\vpt$, and $\cso$ satisfy \eqref{eq:cs-cand-riskless}, \eqref{eq:VE-riskless}, and \eqref{eq:FBP-riskless}. In particular, $\vpt = \vpt(x)$ is increasing and strictly concave with respect to $x$. Furthermore, $\xs=\xu$ $($resp.\ $\xu<\xs<\frac{r+\rho-\del}{\del\rho})$ if $\del\ge r+\rho(1-\al)$ $($resp.\ $0<\del <r+\rho(1-\al))$. 
\end{proposition}

\begin{proof}
	We, first, prove that $J(\cdot)$ is an increasing function such that $\lim_{\xi\to-\infty} J(\xi)=0$ and $\lim_{\xi\to-\xu^{-}}J(\xi)=\yo$. By Propositions \ref{prop:u-solution-impatient} and \ref{prop:u-solution-patient}, $u(\cdot)$ is convex and thus, $u'(\cdot)$ is strictly increasing. Therefore, its inverse $J(\xi)$ exists and is strictly increasing. That $\lim_{\xi\to-\xu^{-}}J(\xi)=\yo$ follows from $\lim_{y\to\yo^{\hspace{0.2ex}-}}u'(y) = -\xu$ by \eqref{eq:FBPDual3-riskless}. Finally, to show $\lim_{\xi\to-\infty} J(\xi)=0$, it suffices to show $\lim_{y\to 0^+} u'(y)=-\infty$. On the contrary, suppose that $\lim_{y\to 0^+} u'(y)\ne-\infty$. Because $u'(\cdot)$ is strictly increasing, we must have $\lim_{y\to 0^+} u'(y) = K$ for some constant $K<-\xu=\lim_{y\to\yo^{\hspace{0.2ex}-}}u'(y)$. Note that for a sufficiently small $\eps>0$, we have in both Propositions \ref{prop:u-solution-impatient} and \ref{prop:u-solution-patient} that $u'(y) = \frac{1}{\rho}-\frac{\psi(y)}{\rho y}$ for $0<y<\eps$. Therefore, we must have
	\begin{align}\label{eq:contradiction}
		\lim_{y\to0^+}\frac{\psi(y)}{y} = 1-\rho K >1.
	\end{align}
	On the other hand, since $\lim_{y\to0^+}\psi(y)=0$, by L'H\^{o}pital's rule, \eqref{eq:psi-TVP}, and \eqref{eq:psi-FBP}, one obtains that
	\begin{align}
		\lim_{y\to0^+}\frac{\psi(y)}{y} = \lim_{y\to0^+} \psi'(y) = 
		\lim_{y\to0^+} \frac{1-\frac{\del}{r+\rho}\frac{\psi(y)}{y}}{\frac{\rho}{r+\rho}\left(\frac{r+\rho-\del}{\rho} -\big(\psi(y)\big)^{-\frac{1}{\gam}}\right)}=0,
	\end{align}
	which contradicts \eqref{eq:contradiction}. Thus, $\displaystyle\lim_{y\to 0^+} u'(y)=-\infty$ which, in turn, yields that $\displaystyle\lim_{\xi\to-\infty} J(\xi)=0$.
	
	It is, then, straightforward to prove that $\xs$, $\vpt$, and $\cso$ satisfy \eqref{eq:cs-cand-riskless}, \eqref{eq:VE-riskless}, and \eqref{eq:FBP-riskless} by reversing the transformation \eqref{eq:convexcon} and by using the fact that $\yo$, $\ys$, and $u$ solve FBP \eqref{eq:FBPDual1-riskless}--\eqref{eq:FBPDual4-riskless}. That $\vpt$ is increasing and strictly concave follows from \eqref{eq:Legendre} since $u$ is decreasing and strictly convex as established by Propositions \ref{prop:u-solution-impatient} and \ref{prop:u-solution-patient}. Finally, the statement about $\xs$ follows from the properties of $\ys$ in Propositions \ref{prop:u-solution-impatient} and \ref{prop:u-solution-patient}.
\end{proof}

We now state the main result of this section regarding the solution of the control problem \eqref{eq:VF-riskless}.

\begin{theorem}\label{thm:VF-riskless}
	Let $v$ and $\cso$ be as in Proposition {\rm \ref{prop:HJB-riskless}}. We, then, have $V(x) = \vpt(x)$ for $x\ge\xu$. Furthermore, an optimal consumption-to-habit policy is given by $\big\{\cso\big(X^*(t)\big) \big\}_{t\ge0}$, in which $X^* = \big\{X^*(t) \big\}_{t \ge 0}$ solves
	\begin{align}\label{eq:XStar-riskless}
		\begin{cases}
			\displaystyle
			\frac{\dd}{\dd t}X^*(t) = (r+\rho)X^*(t) - \big(1+\rho X^*(t)\big) \cso\big(X^*(t)\big);\quad t\ge0, \\
			X^*(0) = x.
		\end{cases}
	\end{align}
\end{theorem}
\begin{proof}
	We complete the proof in two steps. In the first step, we show that $\vpt(x)\ge V(x)$. In the second step, we prove that $t\mapsto\cso\big(X^*(t)\big)$ is admissible and that $\vpt(x) = \int_0^{+\infty} \ee^{-\del\,t} \cso\big(X^*(t)\big)^{1-\gam}/(1-\gam)\dd t$, which implies that $\vpt(x)\le V(x)$. The statement of the theorem, then, follows from these two steps.\vspace{1em}
	
	\noindent \textbf{Step 1:} Let $c(\cdot)\in\Ac_0$, and let $\{X(t)\}_{t \ge 0}$ be the corresponding wealth-to-habit process given by \eqref{eq:X-riskless}. We, then, have
	\begin{align}\label{eq:Dynkin-riskless}
		\ee^{-\del T}\vpt\big(X(T)\big) + \int_0^T \ee^{-\del t}\frac{c(t)^{1-\gam}}{1-\gam}\,\dd t
		= \vpt(x) + \int_0^T \ee^{-\del t} \Lc_{c(t)}\vpt\big(X(t)\big) \dd t;\quad T>0,
	\end{align}
	in which we have defined the operator $\Lc_{c}$, for any $c \ge \al$, by
	\begin{align}
		\Lc_{c} \vpt(x) := -\del \vpt(x) + (r+\rho) x \vpt'(x) + \frac{c^{1-\gam}}{1-\gam} - (1+\rho x)c \vpt'(x).
	\end{align}
	Because $\vpt(\cdot)$ satisfies the HJB equation \eqref{eq:HJB-riskless}, we have $\Lc_{c(t)} \vpt\big(X(t)\big)\le 0$ for all $t\ge 0$. Therefore, from \eqref{eq:Dynkin-riskless}, it follows that
	\begin{align}\label{eq:Dynkin-riskless2}
		\ee^{-\del T} \, \vpt\big(X(T)\big) + \int_0^T \ee^{-\del t} \, \frac{c(t)^{1-\gam}}{1-\gam}\,\dd t
		\le \vpt(x);\quad T\ge0.
	\end{align}
	By letting $T\to+\infty$ in \eqref{eq:Dynkin-riskless2} and using Lemma \ref{lem:Transversality-riskless} in Appendix \ref{app:Transversality-riskless}, we obtain that $\vpt(x)\ge \int_0^\infty \ee^{-\del t} \, \frac{c(t)^{1-\gam}}{1-\gam}\,\dd t$. Finally, by taking the supremum over all $c(\cdot)\in\Ac_0$, we obtain that $\vpt(x)\ge V(x)$.\vspace{1em}
	
	\noindent\textbf{Step 2:} %We need the following result. \red{[correct here]}
	First, we show that \eqref{eq:XStar-riskless} has a solution $X^*:[0,+\infty)\to[\xu,+\infty)$. Note that if $x=\xu$, the unique solution of \eqref{eq:XStar-riskless} is $X^*\equiv\xu$. Similarly, by using Lemma \ref{lem:CS-comparison}(ii) in Appendix \ref{app:Transversality-riskless}, if $0< \del < r+\rho(1-\al)$ and if $x=x_0:=\frac{r+\rho-\del}{\del\rho}$, then the unique solution of \eqref{eq:XStar-riskless} is $X^*\equiv x_0$. Next, consider the case $x\notin\{\xu,x_0\}$, for which \eqref{eq:XStar-riskless} has a non-constant solution. We consider three sub-cases:
	\begin{itemize}
		\item[(a)] Suppose $\del\ge r+\rho(1-\al)$ and $x>\xu$. For $T>0$, define the region $\Dc(T):=\{(t,X^*):0\le t\le T, \, \xu\le X^*\le x\}$. By the classical existence and uniqueness theorem for first-order ODEs, \eqref{eq:XStar-riskless} has a unique solution in $\Dc(T)$ because the ODE is Lipschitz with respect to $x$ over $\Dc(T)$; denote this solution by $X^*:[0,T_1] \to [\xu,x]$ for some $T_1\in(0, T]$.  By Lemma \ref{lem:CS-comparison}(i), $X^*(\cdot)$ is strictly decreasing:
		\begin{align}
			\frac{\dd}{\dd t}X^*(t) = (r+\rho)X^*(t) - \big(1+\rho X^*(t)\big) \cso\big(X^*(t)\big) <0;\quad 0\le t\le T_1.
		\end{align}
		We claim that the solution that starts from the top-left corner can only exit from the right edge of $\Dc(T)$, that is, $T_1 = T$.  Indeed, the solution cannot exit from the top boundary since it starts from the top left corner and is strictly decreasing in  $\Dc(T)$; 	therefore, $X^*(t)<x$ for all $0 \le t \le T_1$.  Furthermore, we must have $X^*(T_1) > \xu$ because $X^*(T_1)=\xu$ contradicts the uniqueness of the solution of the following terminal-value problem:
		\begin{align}
			\begin{cases}
				\displaystyle
				\frac{\dd}{\dd t}\Xt^*(t) = (r+\rho)\Xt^*(t) - \big(1+\rho \Xt^*(t)\big) \cso\big(\Xt^*(t)\big);\quad 0<t\le T_1, \\
				\Xt^*(T_1) = \xu,
			\end{cases}
		\end{align}
		which has the unique solution $\Xt^*\equiv\xu$. So, the solution has to exit from the right boundary of  $\Dc(T)$, which implies that $T = T_1$ and $\xu < X^*(t) < x$.  Because the choice of $T$ is arbitrary, we can conclude that, for this sub-case, \eqref{eq:XStar-riskless} has a unique decreasing solution $X^*:[0,+\infty)\to (\xu,x]$.
		
		\item[(b)] Suppose $0<\del < r+\rho(1-\al)$ and $x>x_0$. For $T>0$, define the region $\Dc(T):=\{(t,X^*):0\le t\le T, \, x_0 \le X^*\le x\}$. As in the argument for sub-case (a), we deduce that \eqref{eq:XStar-riskless} has a unique decreasing solution $X^*:[0,+\infty)\to (x_0,x]$.
		
		\item[(c)] Suppose $0<\del < r+\rho(1-\al)$ and $\xu<x<x_0$. For $T>0$, define the region $\Dc(T):=\{(t,X^*):0\le t\le T, \, x \le X^*\le x_0\}$.  As in the arguments for sub-cases (a) and (b), we deduce that \eqref{eq:XStar-riskless} has a unique increasing solution $X^*:[0,+\infty)\to [x, x_0)$.
	\end{itemize}
	We have, thereby, shown that \eqref{eq:XStar-riskless} has a solution $X^*(t)\ge\xu$.  Because $\cso(x)\ge \al$, it follows that $\cso\big(X^*(\cdot)\big)\in\Ac_0$ and $X^*(\cdot)$ is the corresponding wealth-to-habit process.
	
	It only remains to show that $\vpt(x) = \int_0^{+\infty} \ee^{-\del\,t} \cso\big(X^*(t)\big)^{1-\gam}/(1-\gam) \, \dd t$. To this end, we repeat the argument in Step 1 of the proof with $c(\cdot)$ and $X(\cdot)$ replaced by $\cso\big(X^*(\cdot)\big)$ and $X^*(\cdot)$, respectively. In particular, because $\vpt$ and $\cso$ satisfy \eqref{eq:cs-cand-riskless}, \eqref{eq:VE-riskless}, and \eqref{eq:FBP-riskless}, we have
	\begin{align}
		\Lc_{\cso(x)} \vpt(x) := -\del \vpt(x) + (r+\rho) x \vpt'(x) + \frac{\cso(x)^{1-\gam}}{1-\gam} - (1+\rho x)\cso(x) \vpt'(x) = 0,
	\end{align}
	for $x>\xu$. Equation \eqref{eq:Dynkin-riskless}, then, becomes
	\begin{align}
		\vpt(x) = \ee^{-\delta\,T}\vpt\big(X^*(T)\big) + \int_0^T \ee^{-\del\,t}\frac{\cso\big(X^*(t)\big)^{1-\gam}}{1-\gam} \,\dd t;\quad T>0.
	\end{align}
	Finally, by taking the limit as $T\to+\infty$ and by using Lemma \ref{lem:Transversality-riskless}, we obtain $\vpt(x) = \int_0^{+\infty} \ee^{-\del\,t}\frac{\cso(X^*(t))^{1-\gam}}{1-\gam} \, \dd t$.
\end{proof}

In the proof of Theorem \ref{thm:VF-riskless}, we also established the following results regarding the behavior of the optimal wealth-to-habit and consumption-to-habit processes. In its statement, $x_0$ and $c_0$ are the constants defined in Lemma \ref{lem:CS-comparison} of Appendix \ref{app:Transversality-riskless}, namely,
\begin{align}\label{eq:x0c0}
	x_0:=\frac{r+\rho-\del}{\del\rho},
	\quad\text{and}\quad
	c_0:=\frac{r+\rho-\del}{\rho}.
\end{align}

\begin{corollary}\label{cor:optimconsum_riskless}
	The optimal wealth-to-habit process $\{X^*(t)\}_{t \ge 0}$ and the optimal consumption-to-habit process $\big\{\cso\big(X^*(t)\big) \big\}_{t \ge 0}$ satisfy the following properties. 
	\begin{enumerate}
		\item[$(i)$] If $x=\xu$, then $X^*(t)=\xu$ and $\cso\big(X^*(t)\big)=\al$ for all $t\ge0$.
		\item[$(ii)$] If $\del\ge r+\rho(1-\al)$ and $x>\xu$, then $X^*(t)$ is a decreasing function, $\lim \limits_{t\to+\infty}X^*(t) = \xu$, and $\lim \limits_{t\to+\infty}\cso\big(X^*(t)\big) = \al$.
		\item[$(iii)$] If $0<\del< r+\rho(1-\al)$ and $x>x_0$, then $X^*(t)$ is a decreasing function, $\lim \limits_{t\to+\infty}X^*(t) = x_0$, and $\lim \limits_{t\to+\infty}\cso\big(X^*(t)\big) = c_0$.
		\item[$(iv)$] If $0<\del< r+\rho(1-\al)$ and $x=x_0$, then $X^*(t)=x_0$ and $\cso\big(X^*(t)\big)=c_0$ for all $t\ge0$.
		\item[$(v)$] If $0<\del< r+\rho(1-\al)$ and $\xu<x<x_0$, then $X^*(t)$ is an increasing function, $\lim \limits_{t\to+\infty}X^*(t) = x_0$, and $\lim \limits_{t\to+\infty}\cso\big(X^*(t)\big) = c_0$.  \qed
	\end{enumerate}  
\end{corollary}
\vspace{1em}

\begin{remark}
	The monotonic patterns of the consumption-to-habit and wealth-to-habit ratios are a consequence of our multiplicative habit-formation utility function of \eqref{eq:obj}. They are present even if we remove the habit-formation constraint. To see this, let $\al\to0^+$ in our model and note that $c_0$ and $x_0$ in Corollary \eqref{cor:optimconsum_riskless}.$(iii)-(v)$ do not depend on $\al$. Thus, as $\al\to 0^+$, the limiting behavior (at $t=+\infty$) of the optimal consumption policy should be the same as what is stated in Corollary 3.1.$(iii)-(v)$. In particular, if $w/z>x_0$ (i.e. the initial wealth-to-habit ratio is above $x_0$), then the individual starts with an initial consumption-to-habit $\cs(w/z)>c_0$ and, as $t\to+\infty$, reduces her consumption-to-habit (respectively, wealth-to-habit) to its limiting value of $c_0$ (respectively, $x_0$). Similarly, if $w/z<x_0$ (i.e. the initial wealth-to-habit ratio is below $x_0$), then the individual starts with a initial consumption-to-habit $\cs(w/z)<c_0$ and, as $t\to+\infty$, increases her consumption-to-habit (respectively, wealth-to-habit) to its limiting value of $c_0$ (respectively, $x_0$). Finally, if $w/z=x_0$, then the optimal consumption policy is to consume at a rate such that $c^*_t=C^*_t/Z^*_t\equiv c_0$. Thus, the monotoneicity of $X^*(t)$ and $c^*\big(X^*(t)\big)$ follows from the multiplicative habit-formation utility rather than the habit formation constrain \eqref{eq:Habit}.\qed
\end{remark}

\begin{remark}
	The previous remark also shed more light on the difference between patient and impatient optimal consumption patterns. Consider an individual with $\al=0$. As discussed in the previous remark, she optimally consumes to attain long-term consumption-to-habit level of $c_0$ and investment-to-habit level of $x_0$. Now, let us increase her $\al$ above 0 while keeping all the other parameters fixed. By Lemma \ref{lem:NoRuin}, a non-zero $\al$ would impose a safe level $W_t/Z_t\ge \xu$ in order to keep wealth positive. Then, the question is whether the individual can still attain her long-term consumption-to-habit $c_0$ and investment-to-habit $x_0$ now that she is subject to the habit-formation constraint. The answer depends on whether or not $x_0>\xu$, that is, if the long-term investment-to-habit ratio is sustainable under the habit-formation constraint. By \eqref{eq:xu} and \eqref{eq:x0c0}, we have that $\xu\le x_0~\Longleftrightarrow~ \del<r+\rho(1-\al)$. Thus, patient individuals can reach their long term goals (i.e. the wealth-to-habit ratio of $x_0$ and the consumption-to-habit ratio of $c_0$). If their initial wealth-to-habit ratio is above $x_0$ (i.e. $w/z\ge x_0$), then their consumption policy will be identical to the individual without habit-formation constraint (i.e. $\al=0$). If $w/z<x_0$, then they still optimally consume in a way to reach their long-term goals $(x_0,c_0)$ as in Corollary 3.1.$(v)$. However, their consumption policy is different from the case $\al=0$ in that if their wealth is low ($W^*_t/Z^*_t<x_\al$), they keep consuming at the minimum rate $C^*_t=\al Z^*_t$.
	while impatient individuals cannot. Finally, for impatient individual, $\del>r+\rho(1-\al)$ which implies $x_0<\xu$. Thus, these individual cannot attain their long-term goals of reaching a wealth-to-habit ratio of $x_0$ and a consumption-to-habit ratio of $c_0$. Since they have to start with a  wealth-to-habit ratio $w/z>\xu>x_0$, it is expected for them to try to reach the lowest possible wealth-to-habit ratio of $\xu$ (instead of their lower ideal goal $x_0$). This explains their optimal policy as in Corollary 3.1.$(ii)$.
	\qed
\end{remark}

In the following three subsection, we provide sharp conditions for presence of consumption hump, consider the case of logarithmic utility, and investigate the effect of adding a constant rate of income to our model.

\subsection{Necessary and sufficient condition for consumption hump}\label{sub:InverseU}
As discussed in the introduction, several empirical studies have observed that the consumption spending of individuals has a hump pattern, that is, the (absolute) consumption rate first increases until it reaches a maximum at about the age of 50, and it decreases during the remaining life of the individual.

In this subsection, we provide necessary and sufficient conditions for the presence of such a consumption hump in our model. In particular, as the following proposition states, a consumption hump can only exist if the wealth-to-habit ratio is either $(i)$ larger than a threshold $\xH>max\{x_0, \xu\}$; or $(ii)$ smaller than a threshold $\xH' \in (\xa, x_0)$. In the statement of the proposition, let $C^*(t):=\cs\big(X^*(t)\big) Z^*(t)$ be the optimal absolute consumption rate at time $t\ge 0$, in which the habit process $Z^*(t)$, $t\ge0$, solves
\begin{align}\label{eq:ZSTAR2}
	\begin{cases}
		\frac{\dd Z^*(t)}{\dd t} = -\rho Z^*(t)\big[1 - \cs\big(X^*(t)\big)\big];\quad t\ge0,\\
		Z(0)=z.
	\end{cases}
\end{align}

\begin{proposition}\label{prop:inverseU}
	The optimal absolute consumption rate $t\mapsto C^*(t)$ is hump-shaped if and only if one of the following two cases hold:
	\begin{enumerate}
		\item[$(i)$] $r<\del$ and $w>\xH z$, in which $\xH$ is the unique constant in the interval $\big[max\{x_0, \xu\}, +\infty\big)$ satisfying the equation
		\begin{align}\label{eq:xH}
			1+\frac{\del}{\gam}\frac{\xH-x_0}{1+\rho \xH} -\cs(\xH) = 0.
		\end{align}
		In this case, $t\mapsto C^*(t)$ is hump-shaped and attains its maximum at time $t=\tau_h$, in which $\tau_h>0$ is the unique time such that $X^*(\tau_h)=\xH$.
		
		\item[$(ii)$] $r<\del<r+\rho(1-\al)$, $1+\frac{\del}{\gam}\frac{\xa-x_0}{1+\rho \xa} -\al < 0$, and $\xa z<w<\xH' z$, in which $\xH'$ is the unique constant in the interval $(\xa, x_0)$ satisfying
		\begin{align}\label{eq:xHp}
			1+\frac{\del}{\gam}\frac{\xH'-x_0}{1+\rho \xH'} -\cs(\xH') = 0.
		\end{align}
		In this case, $t\mapsto C^*(t)$ is also hump-shaped and attains its maximum at time $t=\tau_h'$, in which $\tau_h'>0$ is the unique time such that $X^*(\tau_h')=\xH'$.
	\end{enumerate}
	In particular, Conditions $(i)$ and $(ii)$ fail if $\del\le r$, and Condition $(ii)$ fails if $\gam > 1-\frac{\del-r}{\rho(1-\al)}$.
\end{proposition}
\begin{proof}
	See Appendix \ref{app:inverseU}.
\end{proof}

\cite{KraftMunkSeirfriedWagner2017} provided a nice interpretation for the presence of a consumption hump in a habit-formation model such as ours. At the initial time, the individual likes to increase her habit if she can afford it. However, starting with a high initial rate of consumption would lead to a high consumption habit and would diminish her utility of consumption (relative to habit). Instead, the individual starts with a lower rate of consumption and puts aside wealth for her future consumption. As time passes, less wealth is needed to fund future consumption, allowing the individual to increase her rate of consumption. At a certain age, however, the individual's impatience outweighs her concern for future habit, meaning that she prefers consuming more at that point, even if it leads to higher level of habit (and lower level of utility) later. %Thus, she would decrease her consumption rate for the rest of her life.
In our model, as Proposition \ref{prop:inverseU} indicates, such a scenario applies to all individuals if they have a large enough initial wealth. Additionally, an individual with low level of initial wealth behaves in this way if the conditions in Proposition \ref{prop:inverseU}.(ii) hold, which can only be the case if $\gam<1-\frac{\del-r}{\rho(1-\al)}<1$.

\begin{remark}\label{rem:Kraft}
	\cite{KraftMunkSeirfriedWagner2017} provided a sufficient condition, in the form of three inequalities, for presence of consumption hump in their model. The first inequality (i.e. (24) therein) is $r<\del$. The third inequality (i.e. (26) therein) requires the initial wealth to be larger than some threshold (see equation (29) in \cite{KraftMunkSeirfriedWagner2017} and their discussion thereafter). The remaining inequality (i.e. (25) therein) is $\beta-\frac{\del-r}{\gam}>\al>0$, in which $\beta$ (respectively $\al$) corresponds to $\rho$ (respectively, $\al\rho$) in our model.\footnote{The habit process $h(t)$ in \cite{KraftMunkSeirfriedWagner2017} satisfies $\dd h(t) = [\al c(t)-\beta h(t)]\dd t$. The counterpart of $h(t)$ in our model is $\tilde{Z}(t) = \al Z(t)$. From \eqref{eq:Z2}, we obtain that $\dd \tilde{Z}_t = \al \dd Z_t = -\al\rho (Z(t) - C(t)) \dd t = \big(\al\rho C(t)-\rho \tilde{Z}(t)\big)\dd t$. So, $\beta$ (respectively, $\al$) in \cite{KraftMunkSeirfriedWagner2017} is $\rho$ (respectively, $\al\rho$) in our model.} 
	Therefore, Condition (25) of \cite{KraftMunkSeirfriedWagner2017} becomes 
	$\rho - \frac{\del-r}{\gam}>\al\rho$ or, equivalently, 
	$\gam>\frac{\del-r}{\rho(1-\al)}$.
	Thus, the sufficient conditions in \cite{KraftMunkSeirfriedWagner2017} are similar to Condition (i) of Proposition \ref{prop:inverseU} in that they both require $r<\del$ and that the initial wealth $w$ be larger than some threshold. Our Condition (i) is a weaker condition, however, since we don't require the extra lower bound on the risk-aversion parameter $\gam$ as in Condition (25) of \cite{KraftMunkSeirfriedWagner2017}.\footnote{Note that \cite{KraftMunkSeirfriedWagner2017} did not provide any economic explanation for Condition (25). It seems that this condition is adapted to simplify their arguments.}
		
	Condition $(ii)$ of Proposition \ref{prop:inverseU} has a different nature from those provided by \cite{KraftMunkSeirfriedWagner2017}. Firstly, it requires a ``moderate'' level of wealth, that is, a wealth that is not too small (i.e. $w>\xa z$) nor too large (i.e. $w<x_h' z$). This is in contrast to condition (26) of \cite{KraftMunkSeirfriedWagner2017} that only imposes a lower threshold for wealth. Secondly, our Condition $(ii)$ only holds for small values of the risk-aversion parameter $\gam$, as the last statement of Proposition \ref{prop:inverseU} states that Condition $(ii)$ implies that $0<\gam<1-\frac{\del-r}{\rho(1-\al)}<1$. This is in contrast with Condition (25) of \cite{KraftMunkSeirfriedWagner2017} that imposes a lower bound on $\gam$, namely, that $\gam>\frac{\del-r}{\rho(1-\al)}$.
	
	The numerical examples in Section \ref{sec:4} highlight another distinction between the two conditions. Condition $(i)$ (and those in \cite{KraftMunkSeirfriedWagner2017}) show consumption hump for \emph{rich and risk-averse} individuals, while Condition $(ii)$ is relevant for \emph{poor and more risk-seeking} (but patient) individuals. As Figure \ref{fig:Merton_Kraft} in Section \ref{sec:4} indicates (see also Figure 2 in \cite{KraftMunkSeirfriedWagner2017}), Condition $(i)$ holds for an wealth-to-habit ratio of $w/z=370/3.92=94.4$, and risk-aversion $\gam=4$. In contrast, Figure \ref{fig:conshump_lowWealth} of Section \ref{sec:4} shows that Condition $(ii)$ holds for an initial wealth-to-habit ratio of $w/z=2.6$ and risk-aversion $\gam=0.05$.
	
	Finally, note that our conditions are necessary and sufficient, while \cite{KraftMunkSeirfriedWagner2017} only provided sufficient conditions for consumption hump in their model. Thus, the fact that Condition $(ii)$ in Proposition 3.4 does not reconcile with those provided by \cite{KraftMunkSeirfriedWagner2017} does not mean that there is an inconsistency between our results and theirs. Indeed, there may still be scenarios in \cite{KraftMunkSeirfriedWagner2017} with hump-shaped consumption that is not included in their sufficient conditions. Furthermore, the utility functions between the two models are different (we use a multiplicative utility while \cite{KraftMunkSeirfriedWagner2017} used a classical habit-formation utility), which may also lead to different conditions for presence of consumption hump.\qed
\end{remark}

\subsection{The logarithmic utility function}\label{sub:Log}
In this subsection, we consider the following stochastic control problem,
\begin{align}\label{eq:VF_log}
	V_{\log}(x) := \sup_{c(\cdot)\in\Ac(x)} \int_{0}^{+\infty}\ee^{-\del t} \log\big(c(t)\big)\, \dd t;\quad x\ge \xu.
\end{align}
That is, we replace the power utility function in \eqref{eq:VF-riskless} with a logarithmic utility function. As we will show, $V_{\log}$ and the corresponding optimal consumption policy are obtained from our earlier results by taking the limit $\gam\to 1$.

For $\gam\in(0,1)\cup(1,+\infty)$, define the value function
\begin{align}\label{eq:VF_alt}
	\Vt_{\gam}(x) := \sup_{c(\cdot)\in\Ac(x)} \int_{0}^{+\infty}\ee^{-\del t} \frac{\big(c(t)\big)^{1-\gam}-1}{1-\gam}\, \dd t;\quad x\ge \xu.
\end{align}
Note that $V_{\log}=\lim_{\gam\to1} \Vt_\gam$. By comparing $\Vt_\gam$ with $V$ given by \eqref{eq:VF-riskless}, we obtain that 
\begin{align}\label{eq:VF_VFalt}
	\Vt_\gam(x)=V(x)-\int_0^{+\infty}\frac{\ee^{-\del t}}{1-\gam}\dd t = V(x) -\frac{1}{\del(1-\gam)};\quad x\ge \xu,
\end{align}
and that the consumption policy $\cs$ in \eqref{eq:CS-sol-riskless} is also the optimal policy for \eqref{eq:VF_alt}. 
Let $\ut_\gam$ (respectively, $u_{\log}$) be the convex conjugate of $\Vt$ (respectively, $V_{\log}$), namely,
\begin{align}
	\ut_\gam(y) :=\sup_{x\ge \xu}\{\Vt_\gam(x)-xy\};\quad 0<y<\lim_{x\to\xu^+}\Vt'(x)=\lim_{x\to\xu^+}V'(x)=\yo,
\end{align}
and
\begin{align}
	\ut_{\log}(y) :=\sup_{x\ge \xu}\{V_{\log}(x)-xy\};\quad 0<y<\yo_{\log}:=\lim_{x\to\xu^+}V_{\log}'(x)=\lim_{\gam\to 1}\yo.
\end{align}
Recall that, by Propositions \ref{prop:u-solution-impatient} and \ref{prop:u-solution-patient}, we have $\yo=\frac{\al^{-\gam}}{1+\rho \xu}$ (respectively, $\yo=+\infty$) for impatient (respectively, patient) individuals. 
From \eqref{eq:VF_VFalt}, it follows that
\begin{align}\label{eq:ut_u}
	\ut_\gam(y)=\sup_{x\ge \xs}\{V(x)-xy\} -\frac{1}{\del(1-\gam)} = u(y) - \frac{1}{\del(1-\gam)},
	\quad 0<y\le \yo,
\end{align}
in which $u$ is the convex conjugate of $V$. We then obtain the following Corollary of Proposition \ref{prop:u-solution-impatient} for an impatient individual with logarithmic utility.

\begin{corollary}\label{cor:u-solution-impatient-log}
	Assume $\del\ge r+\rho(1-\al)$. There is a strictly increasing function $y:(0,\al^{-1}]\to\big(0,\al^{-1}/(1+\rho\xu)\big]$ satisfying
	\begin{align}\label{eq:y-ODE-impatient-log}
		\begin{cases}\displaystyle
			y'(\psi) = \frac{\frac{\rho}{r+\rho}\left(\frac{r+\rho-\del}{\rho} -\psi^{-1}\right)y(\psi)}{y(\psi)-\frac{\del}{r+\rho}\psi};\quad 0<\psi\le\al^{-1}, \vspace{0.5em} \\
			y\left(\al^{-1}\right) = \dfrac{\al^{-1}}{1+\rho\xu}=\frac{1}{\al}-\frac{\rho}{r+\rho}.
		\end{cases}
	\end{align}
	Furthermore, $0<y(\psi)<\del\psi/(r+\rho)$ for $0<\psi<\al^{-1}$. Finally,
	\begin{align}\label{eq:u-impatient-log}
		u_{\log}(y)=\frac{1}{\del}\left[-\log\big(\psi(y)\big)+\frac{r+\rho-\del}{\rho}\big(\psi(y)-y\big)\right],\quad 0<y\le \yo:=\frac{\al^{-1}}{1+\rho\xu},
	\end{align}
	in which $\psi = \psi(y)$ is the $($strictly increasing$)$ inverse of $y = y(\psi)$.\qed\vspace{1em}
\end{corollary}
\begin{proof}
	 Existence of $y(\psi)$ and its properties is obtained by setting $\gam=1$ in Proposition \ref{prop:u-solution-impatient}.$(i)$ (note that the proof of Proposition \ref{prop:u-solution-impatient}.$(i)$ is valid for $\gam=1$). From \eqref{eq:ut_u} and Proposition \ref{prop:u-solution-impatient}.$(ii)$, it then follows that
	 \begin{align}
	 	\ut_\gam(y)&=\frac{1}{\del}\left[\frac{\gam}{1-\gam}\big(\psi(y)\big)^{1-\frac{1}{\gam}}+\frac{r+\rho-\del}{\rho}\big(\psi(y)-y\big)\right]
	 	- \frac{1}{\del(1-\gam)}\\
	 	&=\frac{1}{\del}\left[-\frac{\big(\psi(y)\big)^{1-\frac{1}{\gam}} - 1}{1-\frac{1}{\gam}}+\frac{r+\rho-\del}{\rho}\big(\psi(y)-y\big)\right],
	 \end{align}
	 for $0<y\le \frac{\al^{-\gam}}{1+\rho\xu}$. Taking the limit $\gam\to1$ then yields \eqref{eq:u-impatient-log}.
\end{proof}

With a similar argument, we obtain the following corollary of Proposition \ref{prop:u-solution-patient} for a patient individual with logarithmic utility function. Its proof is omitted since it is similar to the previous proof.

\begin{corollary}\label{cor:u-solution-patient-log}
	Assume $0<\del< r+\rho(1-\al)$. Define the constants $\psi_0:=\frac{\rho}{r+\rho-\del}\in(0,\al^{-\gam})$ and $y_0:=\frac{\del\psi_0}{r+\rho} \in \big(0,\frac{\al^{-1}}{1+\rho\xu} \big)$. There exist a constant $\ys \in \big(y_0,\frac{\al^{-1}}{1+\rho\xu} \big)$ and a strictly increasing function $y:(0,\al^{-1}]\to\big(0,\ys\big]$ satisfying
	\begin{align}\label{eq:y-FBP-log}
		\begin{cases}\displaystyle
			y'(\psi) = \frac{\frac{\rho}{r+\rho}\left(\frac{r+\rho-\del}{\rho} -\frac{1}{\psi}\right)y(\psi)}{y(\psi)-\frac{\del}{r+\rho}\psi};\quad 0<\psi\le\al^{-1},\\
			y\left(\al^{-1}\right) = \ys.
		\end{cases}
	\end{align}
	Furthermore, $\max\left(0, \psi-\frac{\rho}{r+\rho}\right) < y(\psi) < \frac{\del}{r+\rho}\psi$ for $0<\psi<\psi_0$, and $\frac{\del}{r+\rho}\psi < y(\psi) < \psi-\frac{\rho}{r+\rho}$ for $\psi_0<\psi<\al^{-1}$. Finally, 
	\begin{align}\label{eq:u_patient_sol1-log}
		u_{\log}(y)	= \frac{r+\rho(1-\al) - \del}{\del\rho}\big(\al^{-1} - \ys(1+\rho\xu)\big)
		\left(\frac{y}{\ys}\right)^{-\frac{\del}{r+\rho(1-\al)-\del}}
		-\xu y+\frac{1}{\del}\log(\al);\quad y>\ys,
	\end{align}
	and
	\begin{align}\label{eq:u_patient_sol2-log}
		u_{\log}(y) = -\frac{1}{\del}\log\big(\psi(y)\big)
		+\frac{r+\rho-\del}{\rho\del}\big(\psi(y)-y\big);\quad 0<y\le \ys.
	\end{align}\vspace{1ex}
	in which $\psi = \psi(y)$ is the $($strictly increasing$)$ inverse of $y = y(\psi)$.\qed
\end{corollary}

With $u_{\log}$ at hand, we may then obtain $V_{\log}$ by inverting the  as in Proposition 3.3. The verification argument is very similar to the proof of Theorem 3.1. We thus get the following result.
\begin{theorem}\label{prop:HJB-riskless-log}
	If $\del\ge r+\rho(1-\al))$, let $\yo=\ys=\frac{\al^{-1}}{1+\rho\xu}$ and let $\psi$ and $u_{\log}$ be as in Corollary  \ref{cor:u-solution-impatient-log}. If $0<\del< r+\rho(1-\al))$, let $\yo=+\infty$ and let $\ys$, $\psi$ and $u_{\log}$ be as in Corollary \ref{cor:u-solution-patient-log}. Furthermore, define $J:(-\infty, -\xu)\to(0,\yo)$ to be the inverse of $u'_{\log}$, that is, $u_{\log}'\big(J(\xi)\big)=\xi$ for $\xi<-\xu$. Then, the value function $V_{\log}$ in \eqref{eq:VF_log} is
	\begin{align}
		V_{\log}(x) =
		\begin{cases}
			u_{\log}\big(J(-x)\big) + x J(-x);&\quad x>\xu,\vspace{1ex}\\
			\displaystyle
			\lim_{x\to\xu^+} \Big(u_{\log}\big(J(-x)\big) + x J(-x)\Big);&\quad x=\xu.
		\end{cases}
	\end{align}
	Furthermore, define $\xs := -u_{\log}'(\ys) = \frac{\al^{-1}}{\rho\ys}-\frac{1}{\rho}$ and
	\begin{align}\label{eq:CS-sol-riskless-log}
		\cso(x) :=
		\begin{cases}
			\al; &\quad \xu\le x\le \xs,\vspace{1ex}\\
			\displaystyle
			\frac{1}{\psi\big(J(-x)\big)};&\quad x>\xs.
		\end{cases}
	\end{align}
	Then, an optimal consumption-to-habit policy is given by $\big\{\cso\big(X^*(t)\big) \big\}_{t\ge0}$, in which $X^* = \big\{X^*(t) \big\}_{t \ge 0}$ solves
	$\frac{\dd}{\dd t}X^*(t) = (r+\rho)X^*(t) - \big(1+\rho X^*(t)\big) \cso\big(X^*(t)\big)$, $t\ge0$, with $X^*(0) = x$.\qed
\end{theorem}

\subsection{Optimal consumption policy under constant income and borrowing limit}\label{sub:Income}

In this subsection, we investigate the effect of adding a fixed lifetime income and a fixed borrowing limit to our model. In the setting of Section \ref{sec:2}, assume that the individual receives a fixed income rate $\eta\ge 0$. Let $\Wh(t)$ denote the individual's wealth at time $t\ge0$. Given a consumption process $\{C(t)\}_{t\ge0}$, we then have
\begin{align}\label{eq:wealth-income}
	\frac{\dd \Wh(t)}{\dd t} &= r\,\Wh(t) - C(t) + \eta,
\end{align}
for $t\ge0$, with $\Wh(0)=\wh$ denoting the individual's initial wealth. Assume further that the individual can borrow against her future income. Since the present value of her future income at time $t$ is $\eta/r=\int_t^{+\infty} \eta\ee^{-r (s-t}\dd s$, it makes sense to set the borrowing limit to $\eta/r$. In other words, to avoid bankruptcy, we must have
\begin{align}\label{eq:no_bankruptcy_income}
	\Wh(t)\ge -\frac{\eta}{r}, \quad t\ge0.
\end{align}
Let us first see how the addition of income and borrowing limit changes the safe level given by Lemma \ref{lem:NoRuin}. To this end, consider a new process, which we call ``\emph{effective wealth}'', given by $W(t):=\Wh(t)+\eta/r$, $t\ge0$. From \eqref{eq:wealth-income}, it follows that $\{W(t)\}_{t\ge0}$ satisfy \eqref{eq:wealth} with $W(0)=w:=\wh+\eta/r$. Furthermore, the no-bankruptcy condition \eqref{eq:no_bankruptcy_income} yields that $W(t)\ge 0$ for $t\ge0$. Therefore, by the change-of-variable  $W(t):=\Wh(t)+\eta/r$, we recover the setting of Section \ref{sec:2}. In other words, all of our results are applicable to the above setting once we replace wealth $\Wh$ with effective wealth $W\equiv\Wh+\eta/r$. For instance, Lemma \ref{lem:NoRuin} yields the following corollary.
\begin{corollary}\label{cor:NoRuin-income}
	Assume that $C:\Rb_+\to \Rb_+$ is a measurable function satisfying \eqref{eq:Habit}, in which $\{Z(t)\}_{t\ge0}$ is given by \eqref{eq:Z}. Define the wealth process $\{\Wh(t)\}_{t\ge0}$ by \eqref{eq:wealth-income}. Then, $\Wh(t)>\eta/r$ for all $t\ge0$ if and only if
	\begin{align}\label{eq:NoRuin-incom}
		\Wh(t) \ge \xu Z(t) - \frac{\eta}{r},
	\end{align}
	for all $t\ge0$, with $\xu$ given by \eqref{eq:xu}.
\end{corollary}
As expected, the addition of income and borrowing limit reduces the safe level of wealth. We then define the set of \emph{admissible} investment and consumption policies as those that avoid bankruptcy while satisfying the individual's consumption habit-formation constraint.

\begin{definition}\label{def:Admissible0-income}
	Let $\widehat{\Ac}(\wh,z)$ be the set of all measurable functions $C:\Rb_+\to \Rb_+$ 
	such that $C(t) \ge \al Z(t)$ and $\Wh(t)\ge \xu Z(t)-\eta/r$ for all $t\ge0$, 
	in which $\Wh$ and $Z$ are given by \eqref{eq:wealth-income} and \eqref{eq:Z}, respectively. \qed
\end{definition}

Next, we derive the optimal consumption policy by considering the value function
\begin{align}
	\Vh(\wh,z) := \sup_{C\in\widehat{\Ac}(\wh,z)} 
	\int_{0}^{+\infty}\ee^{-\del t} \frac{1}{1-\gam}\left(\frac{C(t)}{Z(t)}\right)^{1-\gam}\, \dd t;\quad z>0, \wh\ge \xu z-\frac{\eta}{r}.
\end{align}
The following corollary of Theorem \ref{thm:VF-riskless} yields the solution of this stochastic control problem.

\begin{corollary}\label{corol:VF-riskless}
	Let $V$ and $\cso$ be as in Theorem 3.1. We have $\Vh(\wh,z) = V\left(\frac{\wh+\eta/r}{z}\right)$ for $z>0$ and $\wh\ge \xu z-\frac{\eta}{r}$. Furthermore, an optimal consumption policy is given by $C^*(t):=Z^*(t)\cso\left(\frac{\Wh^*(t)+\eta/r}{Z^*(t)}\right)$, $t\ge0$, in which $\big\{\Wh^*(t)\big\}_{t \ge 0}$ and $\big\{Z^*(t)\big\}_{t \ge 0}$ are the solution of the ODE system
	\begin{align}\label{eq:WZ-income}
		\begin{cases}
			\displaystyle
			\frac{\dd \Wh^*(t)}{\dd t} = r\,\Wh^*(t) - \cso\left(\frac{\Wh^*(t)+\eta/r}{Z^*(t)}\right) + \eta,\quad t>0\vspace{1ex}\\
			\displaystyle
			\frac{\dd Z^*(t)}{\dd t} = -\rho\left(Z^*(t) - \cso\left(\frac{\Wh^*(t)+\eta/r}{Z^*(t)}\right)\right),\quad t>0
		\end{cases}
	\end{align}
	for $t\ge0$, and with the initial conditions $\Wh^*(0) = \wh$ and $Z^*(0)=z$.\qed
\end{corollary}
\vspace{1em}

In particular, assuming a wealth $w$ and habit $z$, the addition of income and borrowing limit increases consumption rate from $z\cso(w/z)$ to $z\cso\big((w+\eta/r)/z\big)$. Impatient and patient consumption patterns still exist and are identified by whether or not $\del\ge r+\rho(1-\al)$. However, consumption levels are generally higher. For instance, a patient individual requires a smaller wealth to increase her consumption above it minimum $\al Z(t)$ (i.e. $W(t)>\xa Z(t)-\eta/r$ instead of $W(t)>\xa Z(t)$). Their long term consumption-to-habit ratio is unchanged $\lim_{t\to+\infty} C^*(t)/Z^*(t) = c_0=\frac{r+\rho-\del}{\rho}$, while their long term wealth is generally lower since $\lim_{t\to+\infty}\frac{\Wh^*(t)+\eta/r}{Z^*(t)} = x_0=\frac{r+\rho-\del}{\del\rho}$.

%-----------------------------------------------------------------------------------
%
%         4- Properties of the optimal solution
%
%-----------------------------------------------------------------------------------

\section{Properties of the optimal solution}\label{sec:4}

We end this paper with a discussion of the behavior of the optimal consumption and wealth functions. To prepare for this discussion, observe that, by \eqref{eq:X-riskless}, if $c(t)=\frac{(r+\rho)X(t)}{1+\rho X(t)}$, then the agent's wealth-to-habit ratio remains fixed. Similarly, if $c(t)>\frac{(r+\rho)X(t)}{1+\rho X(t)}$ (resp.\ $<$), then the wealth-to-habit ratio decreases (resp.\ increases).

We can interpret the optimal policy function $\cso$ as follows. If $\del\ge r+\rho(1-\al)$, then the individual is ``impatient'' and, by Corollary \ref{cor:optimconsum_riskless}(ii), she wishes to consume more now rather than later, that is, she chooses a consumption-to-habit ratio of $\cso(X(t))> \frac{(r+\rho)X(t)}{1+\rho X(t)}$, which implies that her wealth-to-habit ratio decreases with the eventual limit (that is, as $t \to +\infty$) of $\xu$. The top plot of Figure \ref{fig:c0_impatient} illustrates this scenario.
%{\bf From Jenny: I don't think of this  individual as aspiring to the wealth-to-habit ratio of $\xu$.  Rather, her impatience makes her want to consume more now instead of later, which leads to $\cso(X(t))> \frac{(r+\rho)X(t)}{1+\rho X(t)}$ and, then, $X(t)$ reaching $\xu$ asymptotically.}

Assume, on the other hand, that the individual is ``patient,'' meaning $0<\del<r+\rho(1-\al)$. Then, Corollary \ref{cor:optimconsum_riskless}(iii)-(v) implies that she ``aspires'' to achieve a wealth-to-habit ratio of $x_0:=\frac{r+\rho-\del}{\del\rho}$ and the consumption-to-habit ratio of $c_0:=\frac{r+\rho-\del}{\rho}$ in the following sense: If $X(t)<x_0$, then her optimal consumption-to-habit is $\cso(X(t))<\frac{(r+\rho)X(t)}{1+\rho X(t)}$.  In other words, because her wealth-to-habit ratio is low enough, it is optimal for her to consume less, which thereby increases that ratio.  Conversely, if $X(t)>x_0$, her optimal consumption-to-habit ratio is $\cso(X(t))>\frac{(r+\rho)X(t)}{1+\rho X(t)}$, which causes her wealth-to-habit ratio to decrease. Finally, if $X(t)=x_0$, then the optimal consumption-to-habit ratio equals $\cso(x_0)=c_0=\frac{(r+\rho)x_0}{1+\rho x_0}$, which maintains the level of wealth-to-habit ratio at $x_0$. In other words, the individual is content with the wealth-to-habit ratio of $x_0$ in this scenario. See the bottom plot of Figure \ref{fig:c0_impatient} for an illustration of this scenario.

Next, we illustrate the optimal consumption and wealth as a function of time. Figure \ref{fig:xcs_impatient} shows the sample paths for an impatient individual. The top left plot shows the optimal wealth-to-habit function $t\mapsto X^*(t)$, that is, the solution of \eqref{eq:XStar-riskless}. Note that $X^*$ is a decreasing function and approaches $\xu$ as $t \to +\infty$. The top right plot shows the corresponding consumption-to-habit path $t\mapsto \cs\big(X^*(t)\big)$, which decreases to $\al$ as $t \to +\infty$. The bottom plots show the corresponding optimal wealth paths $t\mapsto W^*(t)$ and $t\mapsto C^*(t)$ that are found by setting $c\equiv t\mapsto \cs\big(X^*(t)\big)$ in Proposition \ref{prop:Admiss}.

\begin{figure}[t]
	\centerline{
		% \fbox{
		\adjustbox{trim={0.0\width} {0.04\height} {0.2\width} {0.0\height},clip}
		{\includegraphics[scale=0.31, page=1]{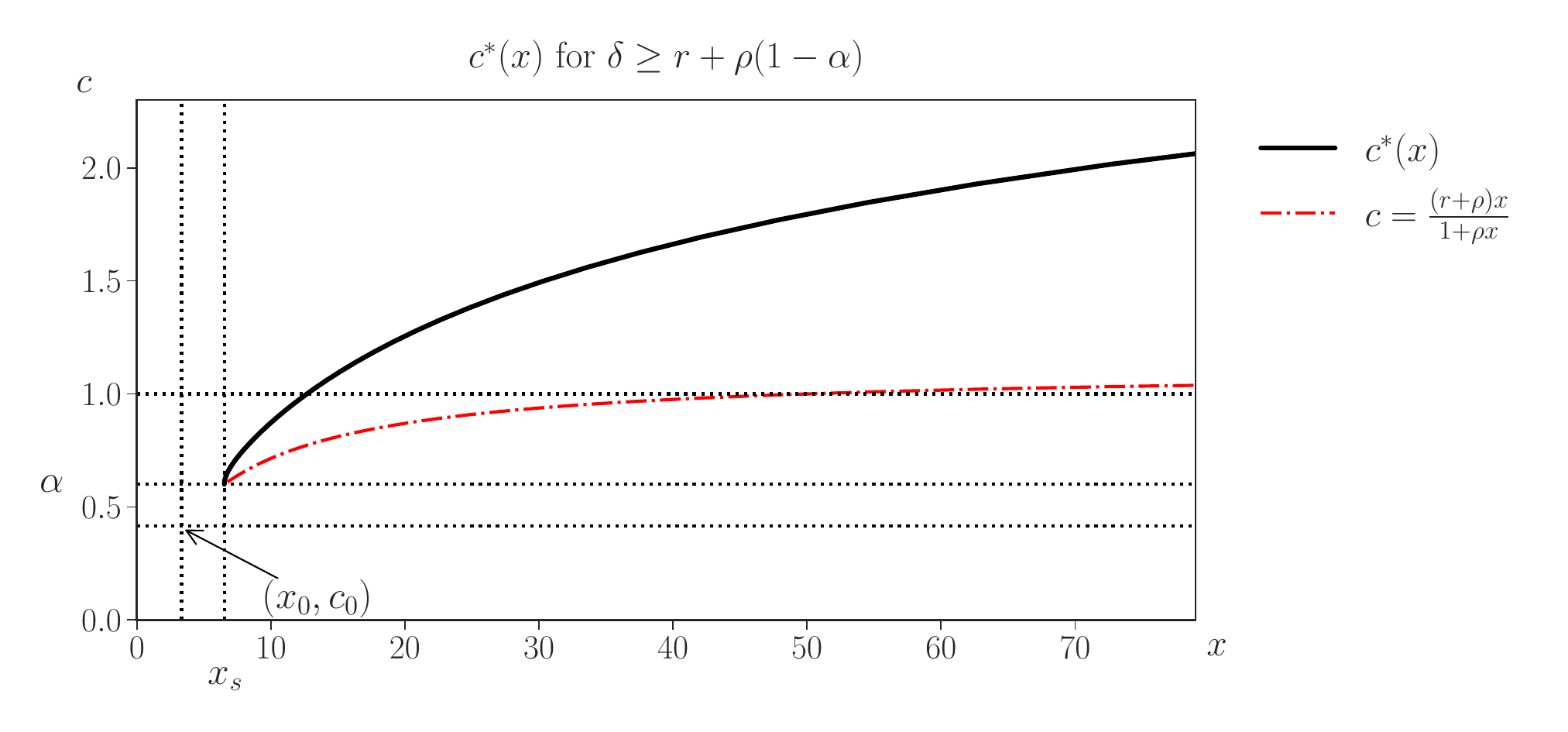}}
		% }
		\hspace{1em}
		\adjustbox{trim={0.0\width} {0.0\height} {0.0\width} {0.0\height},clip}
		{\includegraphics[scale=0.3, page=1]{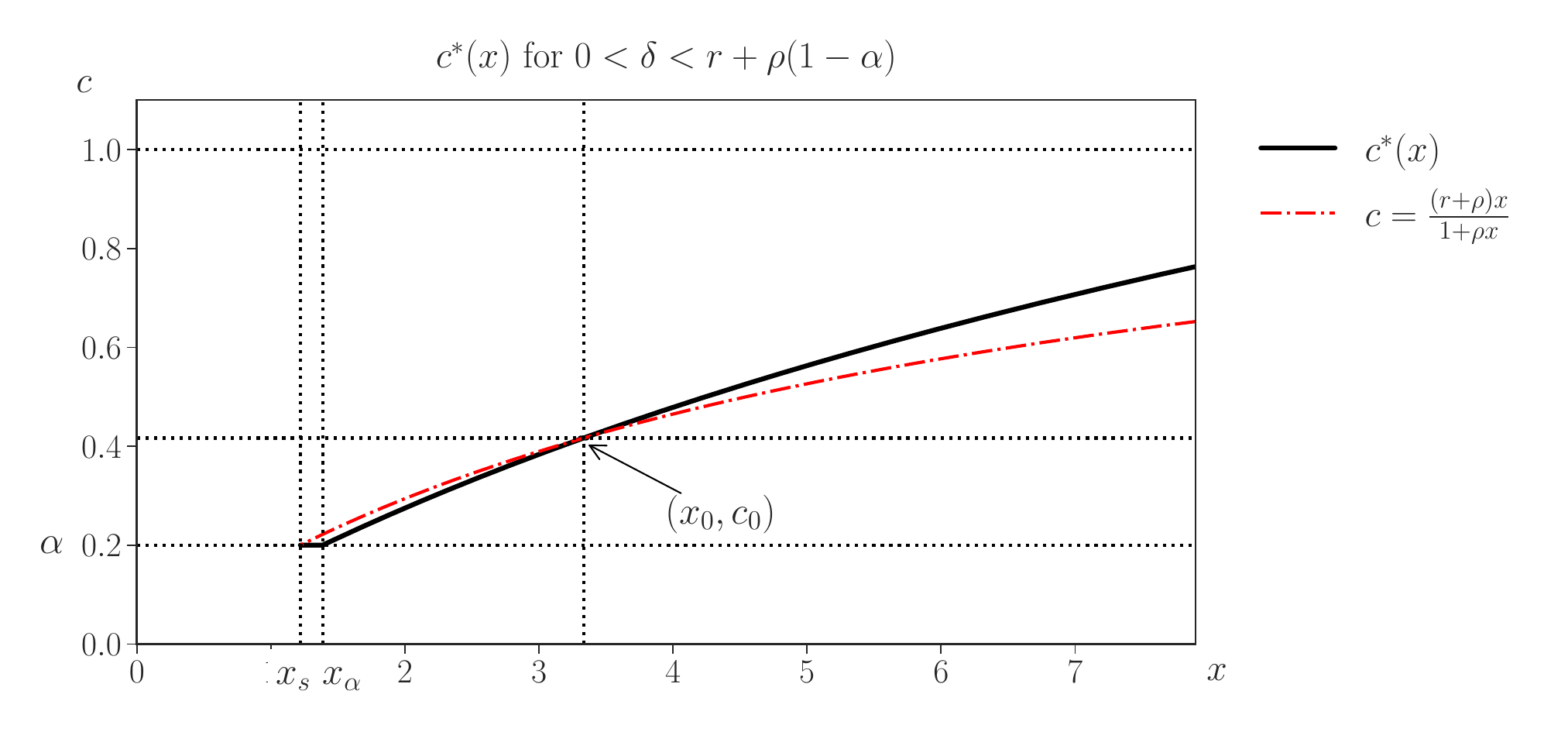}}
	}
	\caption{
		\textbf{Top:} The optimal consumption-to-habit function $\cso(x)$ (solid line)
		for an impatient individual, namely, one for whom $\del\ge r+\rho(1-\al)$.  In this case, $\cso(x)>\frac{(r+\rho)x}{1+\rho x}$ for all $x \ge \xu$ because an impatient person prefers to consume more now rather than later.  The effect is that it is optimal to consume in such a way that to the wealth-to-habit ratio continually decreases and eventually reaches the lowest possible ratio of $\xu$. 
		\textbf{Bottom:} The optimal consumption-to-habit function $\cso(x)$ (the solid line)
		for a patient individual, namely, one for whom $0<\del<r+\rho(1-\al)$. Note that, by \eqref{eq:X-riskless}, if $c(t)>\frac{(r+\rho)X(t)}{1+\rho X(t)}$ (resp.\ $<$), then the wealth-to-habit ratio decreases (resp.\ increases). Since $\cso(x)<\frac{(r+\rho)x}{1+\rho x}$ for $\xu\le x<x_0$, the optimal consumption causes the wealth-to-habit ratio to increase. If, on the other hand, $x>x_0$, then $\cso(x)>\frac{(r+\rho)x}{1+\rho x}$, which causes the wealth-to-habit ratio to decrease. In other words, the optimal consumption policy of a patient individual moves towards the consumption-to-habit ratio of $c_0$, which corresponds to a wealth-to-habit ratio of $x_0$. We use the following set of values for the parameters: $r=0.02$, $\rho=0.18$, $\del=0.125$, and $\gam=2$. We have chosen $\al=0.6$ for the case $0<\del<r+\rho(1-\al)$, and $\al=0.2$ for the case $\del>r+\rho(1-\al)$.
		\label{fig:c0_impatient}}
\end{figure}

\begin{figure}[p]
	\centerline{
		\adjustbox{trim={0.0\width} {0.0\height} {0.0\width} {0.0\height},clip}
		{\includegraphics[scale=0.3, page=1]{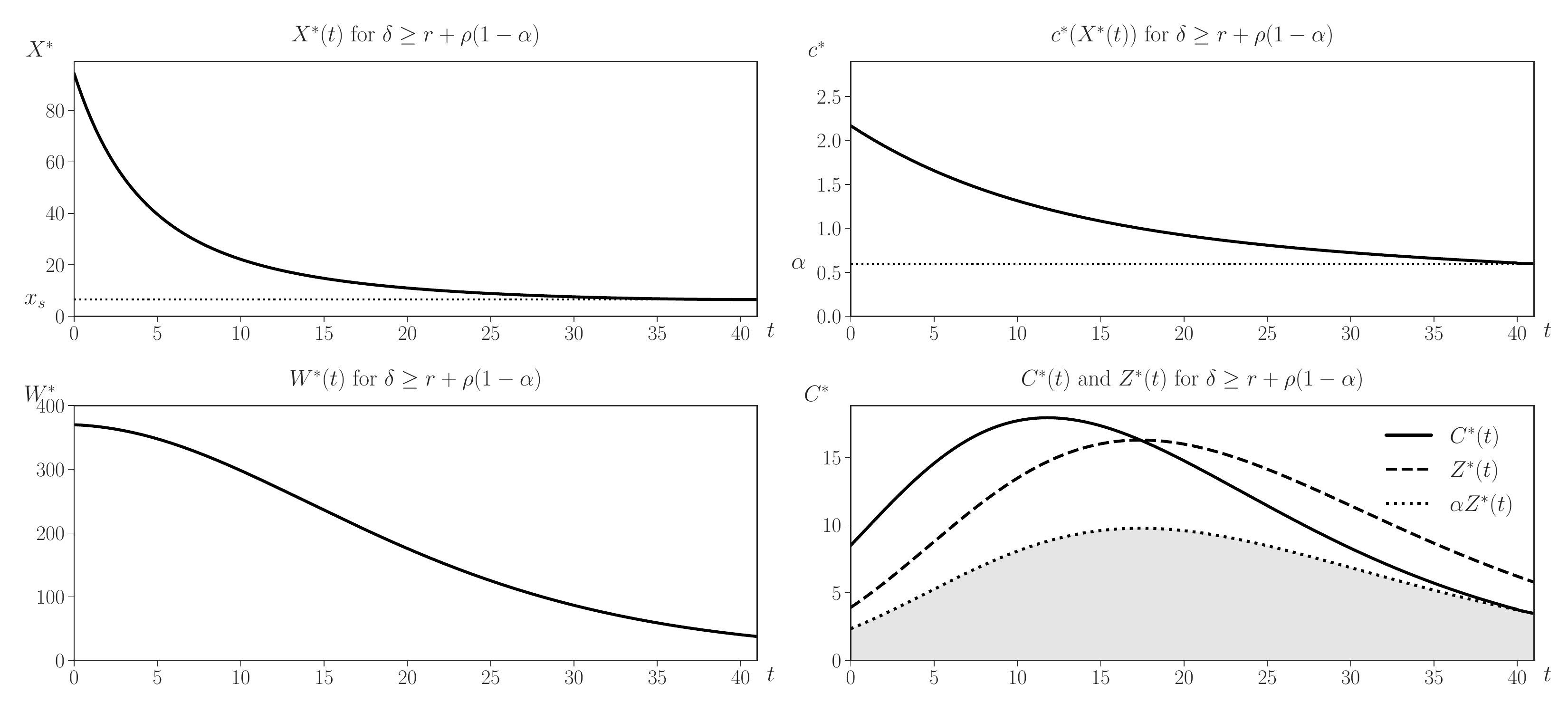}}
	}
	\caption{
		Sample paths of the optimal relative and absolute wealth and consumption for an impatient individual, that is, $\del>r+\rho(1-\al)$. \textbf{Top left:} The optimal wealth-to-habit ratio $X^*(t)$, $t\ge0$, is decreasing and approaches its minimum value $\xu$ as $t \to +\infty$. \textbf{Top right:} The consumption-to-habit ratio is $\cs\big(X^*(t)\big)$, $t\ge0$ is decreasing and approaches its minimum value $\al$ as $t \to +\infty$. \textbf{Bottom left:} Path of the optimal wealth $W^*$. \textbf{Bottom right:} Paths of the optimal consumption rate $C^*$ (solid line) and the corresponding habit $Z^*$ (the dashed line). Note that there is a consumption hump at about 15 years. The shaded region represent infeasible consumption rates, that is, values below $\al Z^*$. We have chosen $w=370$ and $z=3.92$. The remaining parameters are as in Figure \ref{fig:c0_impatient}.
		\label{fig:xcs_impatient}}
	\vspace{1em}
	
	\centerline{
		\adjustbox{trim={0.0\width} {0.0\height} {0.0\width} {0.0\height},clip}
		{\includegraphics[scale=0.3, page=1]{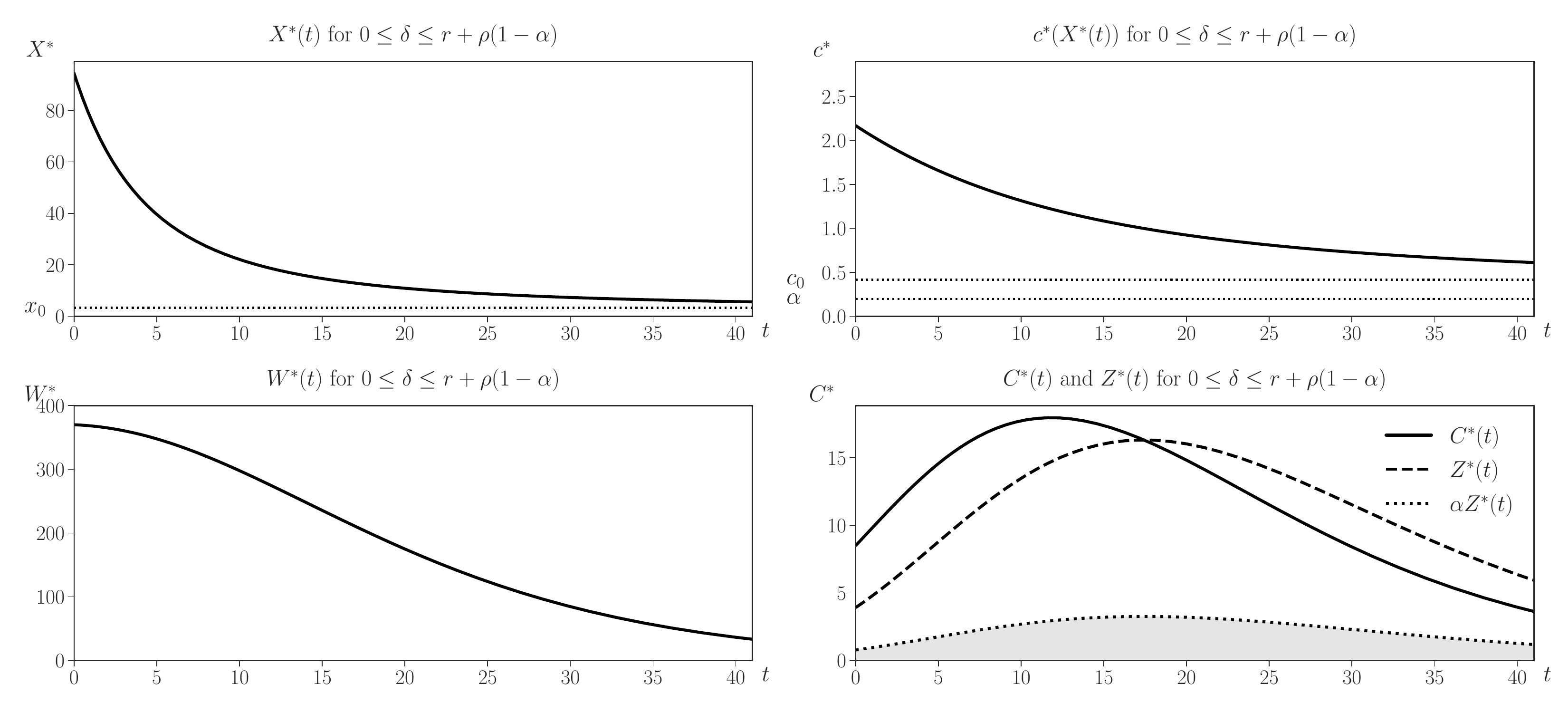}}
	}
	\caption{Sample paths of the optimal relative and absolute wealth and consumption for a patient individual, that is, $0<\del<r+\rho(1-\al)$, with an initial wealth-to-habit ratio $x>x_0:=(r+\rho-\del)/(\del\rho)$.
		\vspace{1em}
	\label{fig:xcs_patient1}}
\end{figure}

\begin{figure}[p]
	\centerline{
		\adjustbox{trim={0.0\width} {0.0\height} {0.0\width} {0.0\height},clip}
		{\includegraphics[scale=0.3, page=1]{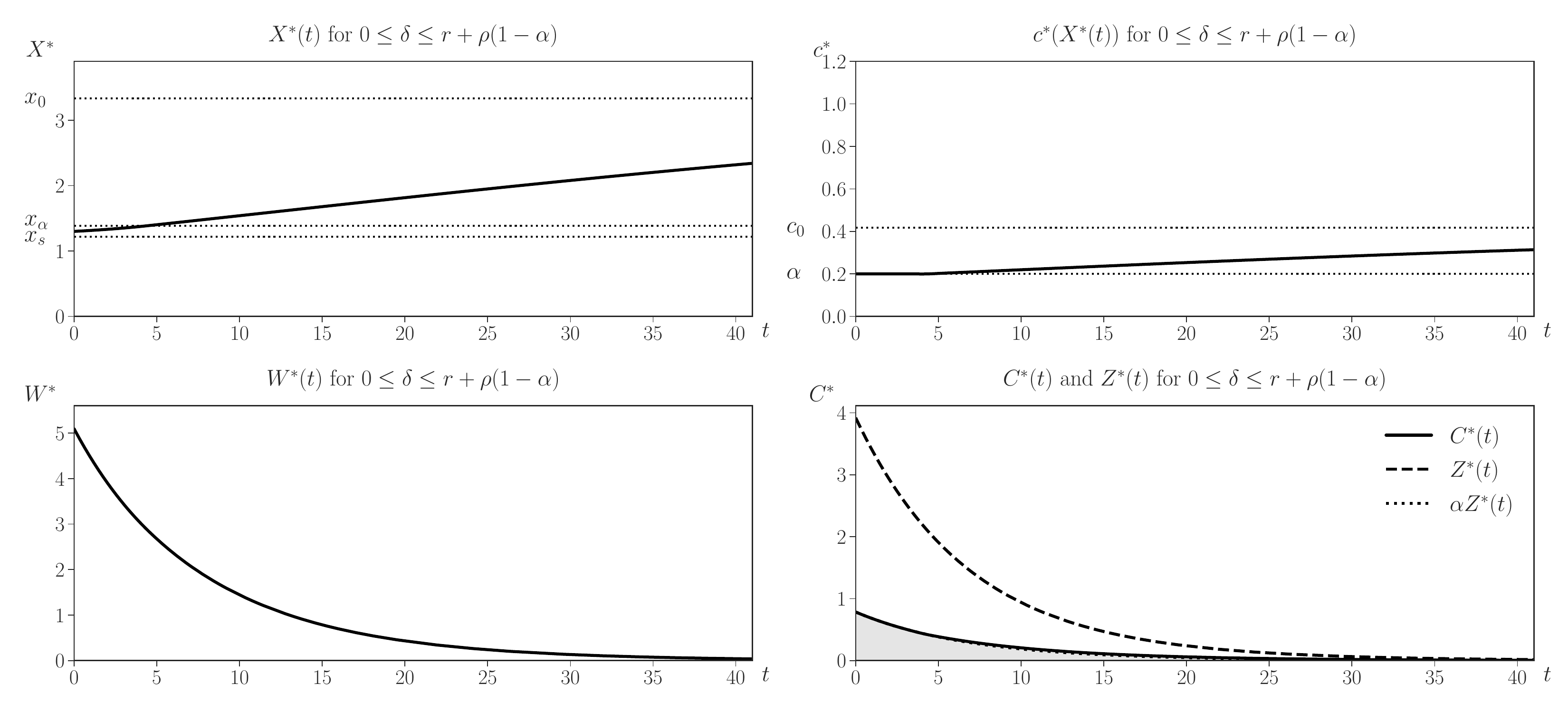}}
	}
	\caption{
		Counterpart of Figure \ref{fig:xcs_patient1} for an initial wealth-to-habit ratio $x<\xs$.
		%		\textbf{Top left:} The optimal wealth-to-habit ratio $X^*(t)$, $t\ge0$, is increasing and approaches its maximum value $x_0$ at $t=+\infty$. \textbf{Top right:} At first, the consumption-to-habit ratio is $\cs\big(X^*(t)\big)$, $t\ge0$, is the constant $\al$ while $X^*<\xs$. Then, it increases to its maximum value $c_0:=(r+\rho-\del)/\rho$ as $t \to +\infty$. \textbf{Bottom:} Paths of the optimal wealth $W^*$, the optimal consumption rate $C^*$, and the corresponding habit $Z^*$.
		\label{fig:xcs_patient2}}
	\vspace{1em}
	
	\centerline{
		\adjustbox{trim={0.0\width} {0.0\height} {0.0\width} {0.0\height},clip}
		{\includegraphics[scale=0.3, page=1]{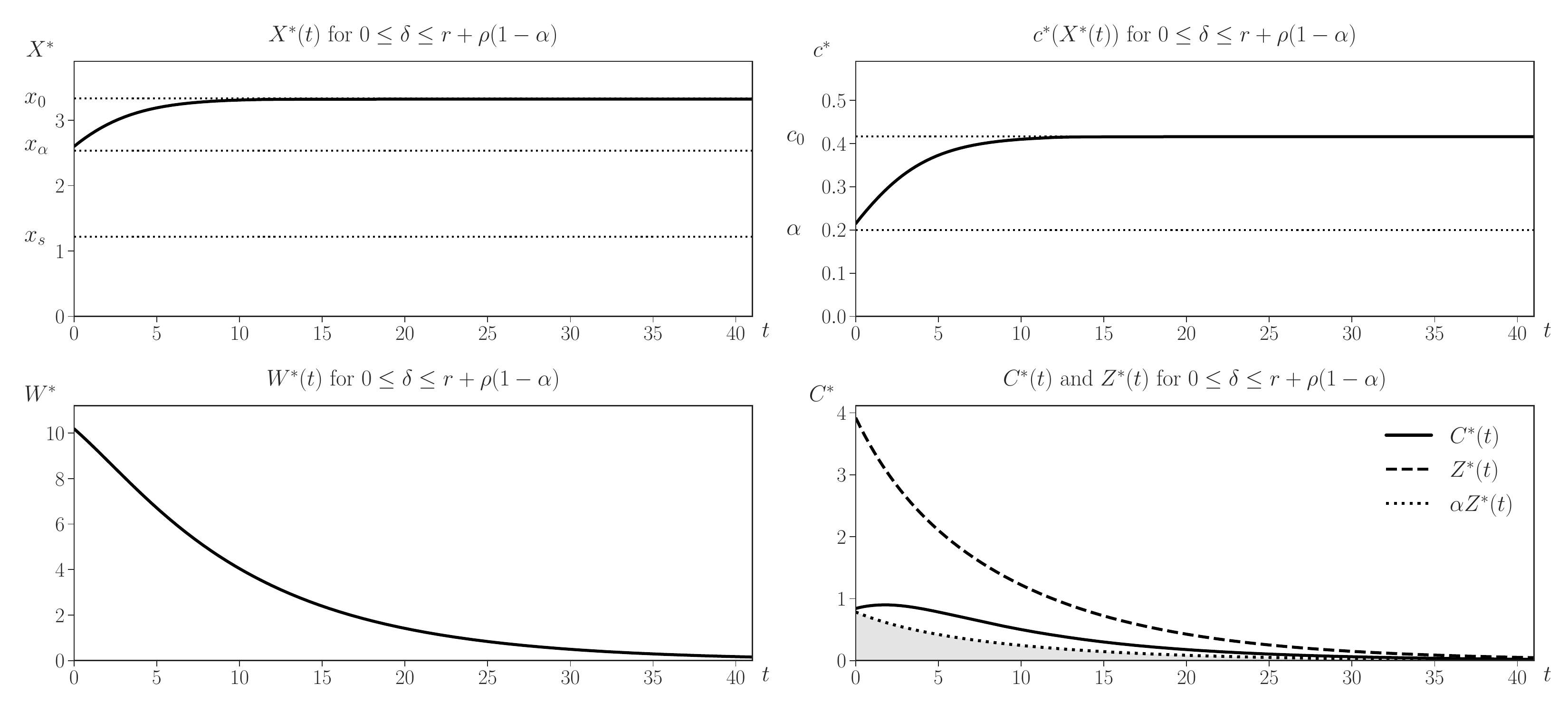}}
	}
	\caption{
		An example for presence of a consumption hump with low level of wealth-to-habit ratio, according to Proposition \ref{prop:inverseU}.(ii). As the bottom right plot indicates, there is a consumption hump at about 2 years. The values of the parameters are as follows: $r=0.02$, $\rho=0.18$, $\al=0.2$, $\del=0.125$, $\gam=0.05$, $z=3.92$, and $w=2.6 z=10.2$. For this case, $\xa=2.5371$ and $x_0=3.3333$. Therefore, $1+\frac{\del}{\gam}\frac{\xa-x_0}{1+\rho \xa} -\al = -0.567 < 0$. By numerically solving \eqref{eq:xHp}, we obtain $\xH'=2.90145$. Thus, the conditions of Proposition \ref{prop:inverseU}.(ii) hold as long as $w/z\in[\xa=2.5371, \xH'=2.90145)$, and we have chosen $w/z=2.6$.
		\vspace{1em}
	\label{fig:conshump_lowWealth}}
\end{figure}

\begin{figure}[p]
	\centerline{
%		\fbox{
		\adjustbox{trim={0.0\width} {0.0\height} {0.225\width} {0.0\height},clip}
		{\includegraphics[scale=0.3, page=1]{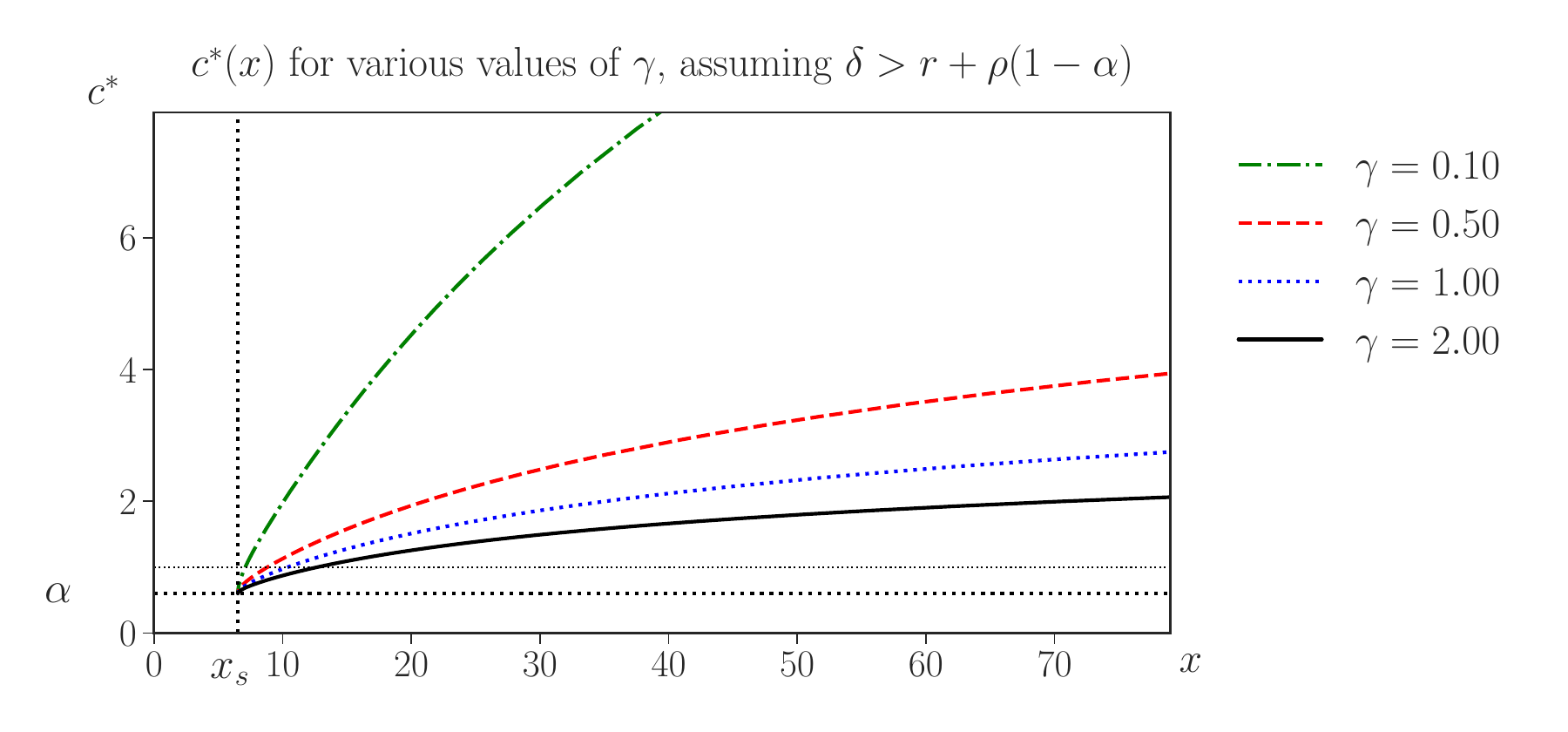}}
		% }
		\hspace{1em}
		\adjustbox{trim={0.0\width} {0.0\height} {0.0\width} {0.0\height},clip}
		{\includegraphics[scale=0.33, page=1]{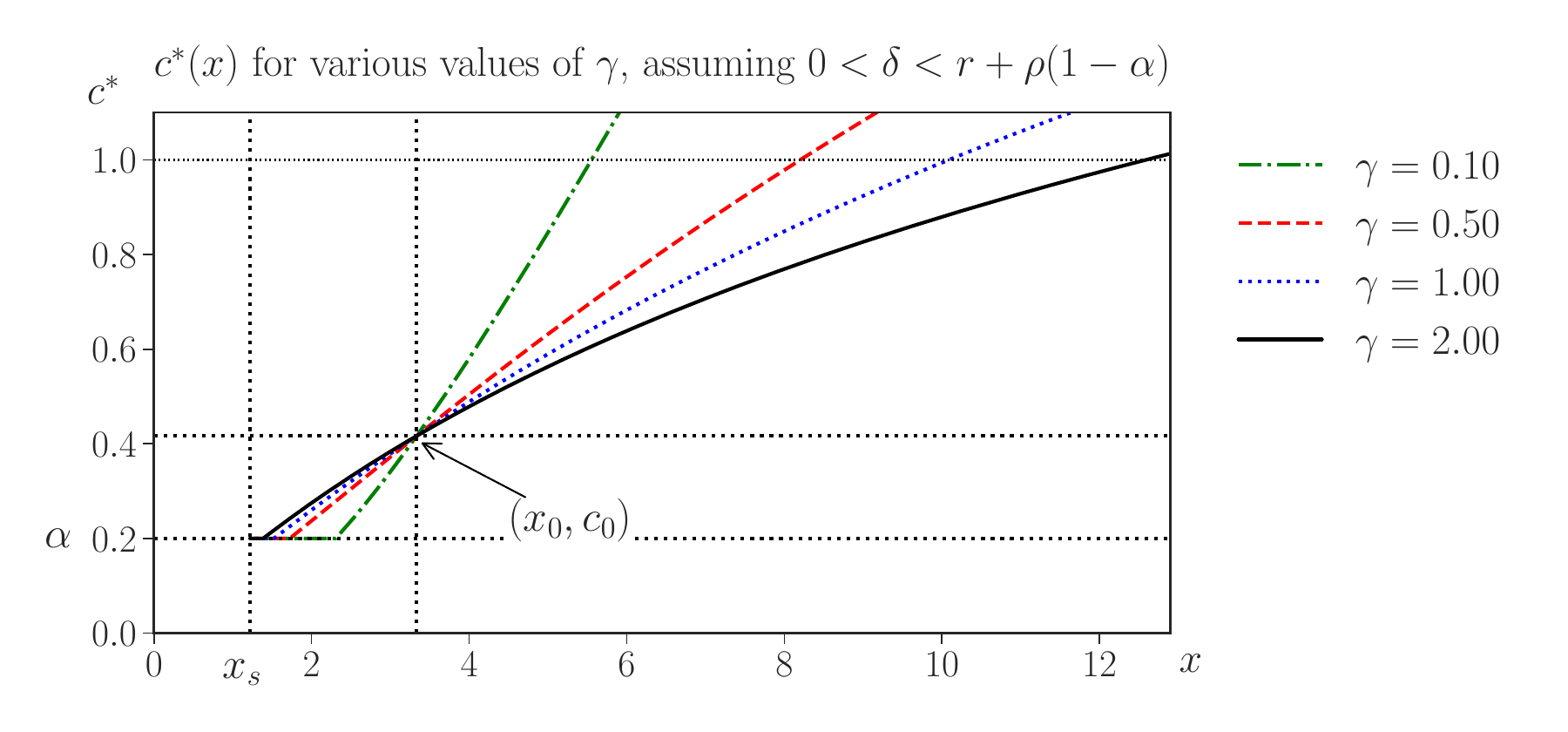}}
%	}
	}
	\caption{
		Sensitivity of the optimal consumption-to-habit function $\cso(x)$ to the risk aversion parameter $\gam$. The case $\gam=1$ corresponds to the logarithmic utility function as discussed in Subsection \ref{sub:Log}.
		\textbf{Top:} For impatient individuals (those with $\del>r+\rho(1-\al)$), higher risk aversion decreases optimal consumption at all levels of wealth-to-habit ratio.	\textbf{Bottom:}  For patient individuals (those with $0<\del<r+\rho(1-\al)$), higher risk aversion decreases (resp.\ increases) optimal consumption if wealth-to-habit ratio is above (resp.\ below) $x_0$. Note, that $\xa$ is decreasing in $\gam$. 
		\label{fig:gam_sens}}
	\vspace{1em}
\end{figure}

\begin{figure}[p]
	\centerline{
%		\fbox{
		\adjustbox{trim={0.0\width} {0.0\height} {0.0\width} {0.0\height},clip}
		{\includegraphics[scale=0.325, page=1]{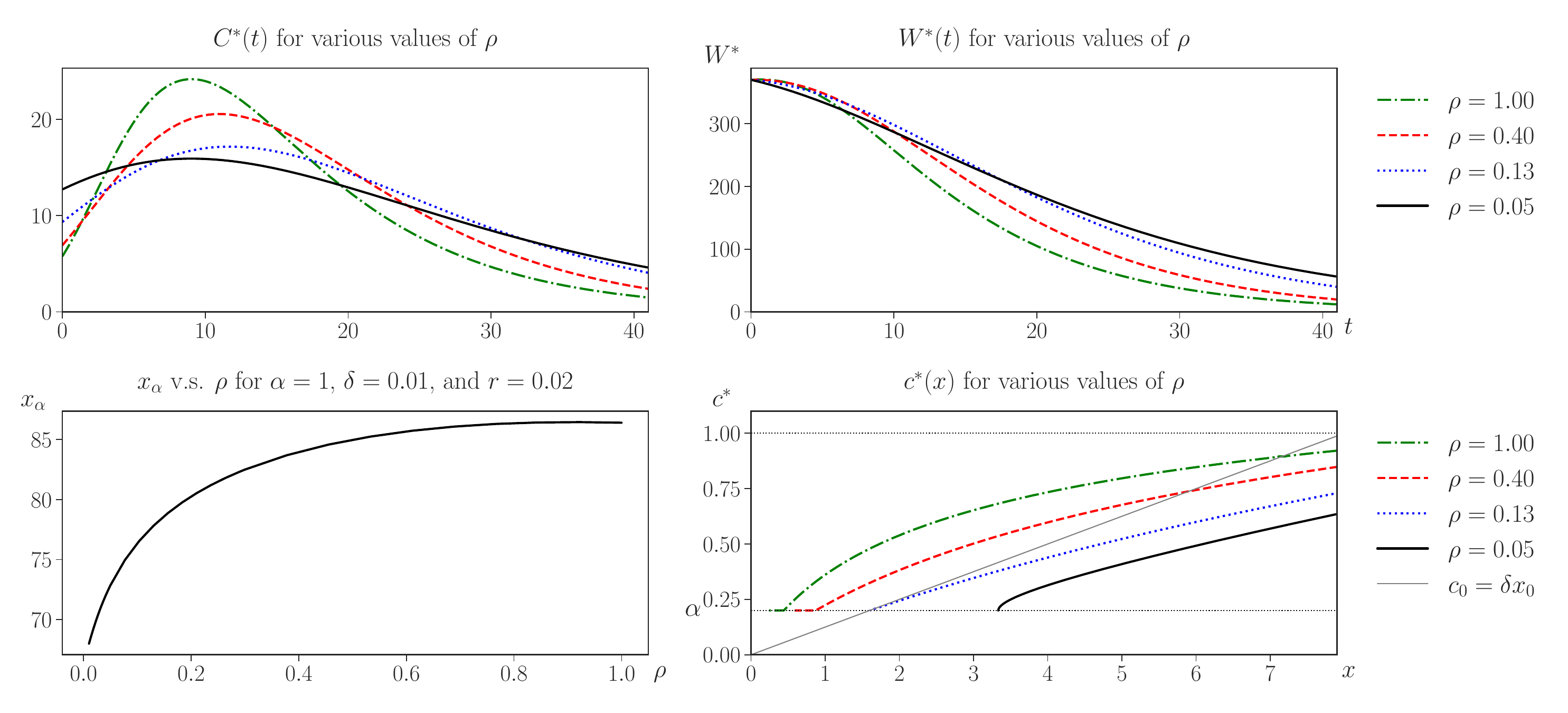}}
%	}
	}
	\caption{
		Sensitivity of the optimal consumption and wealth to the habit-formation parameter $\rho$.  \textbf{Top left:} Paths of optimal consumption $C^*(t)$ for various values of $\rho$. Higher $\rho$ amplify the consumption hump. \textbf{Top right:} The corresponding paths of of the optimal wealth $W^*(t)$. Higher $\rho$ leads to higher spending. \textbf{Bottom right:} The consumption-to-habit optimal feedback policy function $\cso(x)$ for various $\rho$. Note that the safe level $\xu=\al/(r+\rho(1-\al))$ is decreasing in $\rho$. Thus, the consumption function shifts to left as $\rho$ increases. \textbf{Bottom left:} Values of the free-boundary $\xs$ as a function of $\rho$ and assuming $\al=1$ and $\del<r$ (for a patient individual). Note that for this case, the safe level $\xu=1/r=50$ is independent of $\rho$. The free-boundary $\xs$ however depends on $\rho$ since $\rho$ is still present in the dynamics \eqref{eq:X-riskless}.
		\label{fig:rho_sense}}
%	\vspace{1em}
\end{figure}

For a patient individual, the relative wealth and habit have two different regimes. If the initial wealth-to-habit ratio is sufficiently large, namely $x>x_0:=(r+\rho-\del)/(\del\rho)$, then the optimal wealth-to-habit ratio decreases to $x_0$ as $t \to +\infty$, while the optimal consumption-to-habit ratio decreases to $c_0:=(r+\rho-\del)/\rho$ as $t \to +\infty$. Figure \ref{fig:xcs_patient1} shows this scenario. If, on the hand, $x<x_0:=(r+\rho-\del)/(\del\rho)$, then the optimal wealth-to-habit (resp.\ consumption-to-habit) ratio increases to $x_0$ (resp.\ $c_0:=(r+\rho-\del)/\rho$) as $t \to +\infty$. In particular, if $x<\xs$, then the optimal consumption-to-habit ratio is kept at $\al$ while $X^*(t)<\xs$. Figure \ref{fig:xcs_patient2} illustrates this second scenario. The bottom plots of Figures \ref{fig:xcs_patient1} and \ref{fig:xcs_patient2} show the paths of the optimal (absolute) wealth and consumption obtained by Proposition \ref{prop:Admiss}.

The bottom right plots in Figures \ref{fig:xcs_impatient} and \ref{fig:xcs_patient1} show consumption humps for high levels of wealth-to-habit ratios according to Proposition \ref{prop:inverseU}.(i). Figure \ref{fig:conshump_lowWealth} illustrate a numerical example of a consumption hump for low level of wealth-to-habit ratio according to Proposition \ref{prop:inverseU}.(ii). For this case, we have chosen a low level of risk aversion $\gam=0.05$, as Condition (ii) requires it. Other values of the parameters are as follows: $r=0.02$, $\rho=0.18$, $\al=0.2$, $\del=0.125$, $z=2.253$, and $w=2.6 z$. For this case, we find that $\xa=2.5371$ and $x_0=3.3333$. Therefore, $1+\frac{\del}{\gam}\frac{\xa-x_0}{1+\rho \xa} -\al = -0.567 < 0$. Furthermore, we obtain that $\xH'=2.90145$ by numerically solving \eqref{eq:xHp}. Thus, the conditions of Proposition \ref{prop:inverseU}.(ii) holds as long as $w/z\in[\xa=2.5371, \xH'=2.90145)$. Note that we have chosen $w/z=2.6$.

Figure \ref{fig:gam_sens} shows the dependence of the optimal relative consumption policy on the risk aversion parameter $\gam$. The top plot indicates that, for impatient individuals, the optimal relative consumption decreases as $\gam$ increases. In other words, more risk averse impatient individuals optimally consume less. The bottom plot shows a different story for patient individuals. If their wealth-to-habit ratio is above $x_0$, the more risk averse they are, the less they consume. However, for wealth-to-habit ratio below $x_0$, the opposite is true: more risk aversion increases consumption. The bottom plot also shows that $\xa$ is decreasing in $\gam$. Note that, for both cases, we have included the logarithmic utility function (i.e. $\gam=1$), which was discussed in subsection \ref{sub:Log}.

Figure \ref{fig:rho_sense} shows dependence of the optimal consumption and wealth to the habit-formation parameter $\rho$ in \eqref{eq:Z}. As the top plots illustrates, increasing $\rho$ amplifies the consumption hump and increases spending. The bottom left plot illustrates that increasing $\rho$ shifts the optimal relative-consumption feedback function $\cso(x)$ to left. That is, for a given level of the wealth-to-habit ratio $x$, increasing $\rho$ increases the optimal consumption-to-habit $\cso(x)$. These effects are expected since increasing $\rho$ strengthens the individual's habit-formation, and all the aforementioned effects are a consequence of habit-formation. Note that for $\al=1$, the safe level becomes $\xu=1/r$, and is independent of $\rho$. The threshold $\xs$ for patient individuals (which in this case is when $\del<r$), however, still depends on $\rho$, as the bottom right plot of Figure \ref{fig:rho_sense} illustrates. The reason is that, even for $\al=1$, the parameter $\rho$ still effects the dynamics of $(X_t)_{t\ge0}$ in \eqref{eq:X}.

Next, we investigate how the optimal consumption-to-habit function $x\mapsto c^*(x)$ depends on the parameter $\al$, assuming other parameters are fixed. Figure \ref{fig:alpha_sens} illustrates this dependence. First, note that the domain of $\cs$ depends on $\al$. Specifically, by Lemma \ref{lem:NoRuin}, $\cs(x)$ is only defined for values of $x\ge\xu:= \al/(r+\rho(1-\al))$ with $\cs(\xu)=\al$. Thus, the graph of $x\mapsto c^*(x)$ starts at the point $(\xu, \al=(r+\rho)\xu/(1+\rho \xu))$. The function $\xu\mapsto (r+\rho)\xu/(1+\rho \xu)$ is represented by the red dashed-dotted line in Figure \ref{fig:alpha_sens}. Second, note that, depending on the values of $\del$, $\rho$, and $r$, we can identify the following three scenarios:
\begin{enumerate}
	\item[(i)] If $0<\del<r$, then $0<\del<r+\rho(1-\al)$ for all $\al\in(0,1)$. Thus, the individual is patient for all values of $\al$. The top plot of Figure \ref{fig:alpha_sens} illustrates this case.
	
	\item[(ii)] If $\del>r+\rho$, then $\del>r+\rho(1-\al)$ for all $\al\in(0,1)$. Thus, the individual is impatient for all values of $\al$. The bottom plot of Figure \ref{fig:alpha_sens} illustrates this case.
	
	\item[(iii)] If $r<\del<r+\rho$, then $\del<$ (resp.\ $>$) $r+\rho(1-\al)$ for $\al<$ (resp.\ $>$) $1-(\del-r)/\rho$. Thus, the individual is impatient when $\al$ is near 1 and patient when $\al$ is near 0. The middle plot of Figure \ref{fig:alpha_sens} illustrates this case.
\end{enumerate}

Figure \ref{fig:alpha_sens} highlights that the dependence of the optimal consumption-to-habit function to $\al$ is quite different between patient and impatient individuals. For impatient individuals, increasing $\al$ decreases $\cs(x)$. In contrast, $\cs(x)$ is a non-decreasing function of $\al$ for patient individuals. In particular, if $x>\xs$, then $\cs(x)$ does not change with a small change in $\al$. If, on the other hand, $x<\xs$, then $\cs(x)=\al$ is increasing in $\al$.

\begin{figure}[p]
	\centerline{
		\adjustbox{trim={0.0\width} {0.0\height} {0.0\width} {0.0\height},clip}
		{\includegraphics[scale=0.475, page=1]{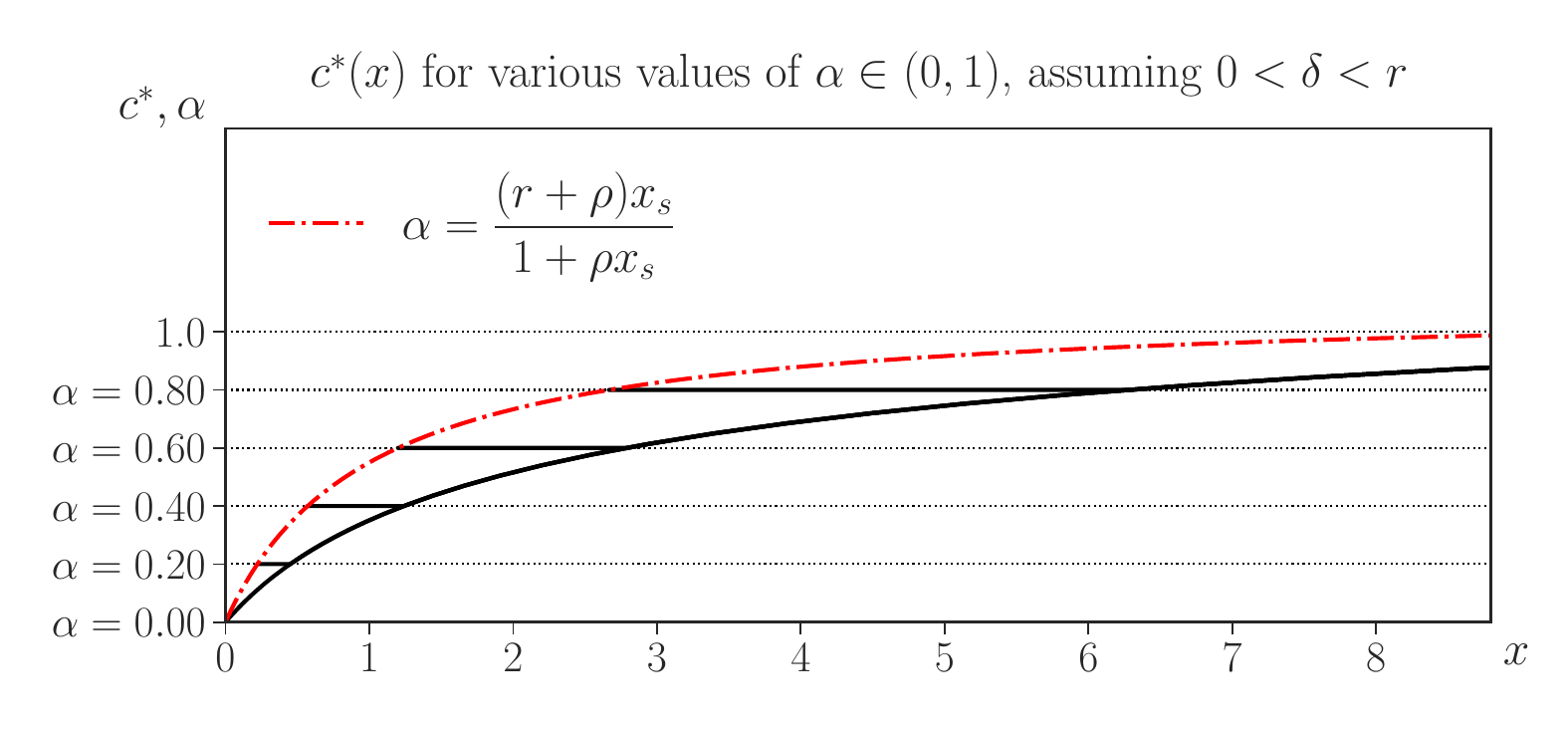}}
	}
	\centerline{
		\adjustbox{trim={0.0\width} {0.0\height} {0.0\width} {0.0\height},clip}
		{\includegraphics[scale=0.475, page=1]{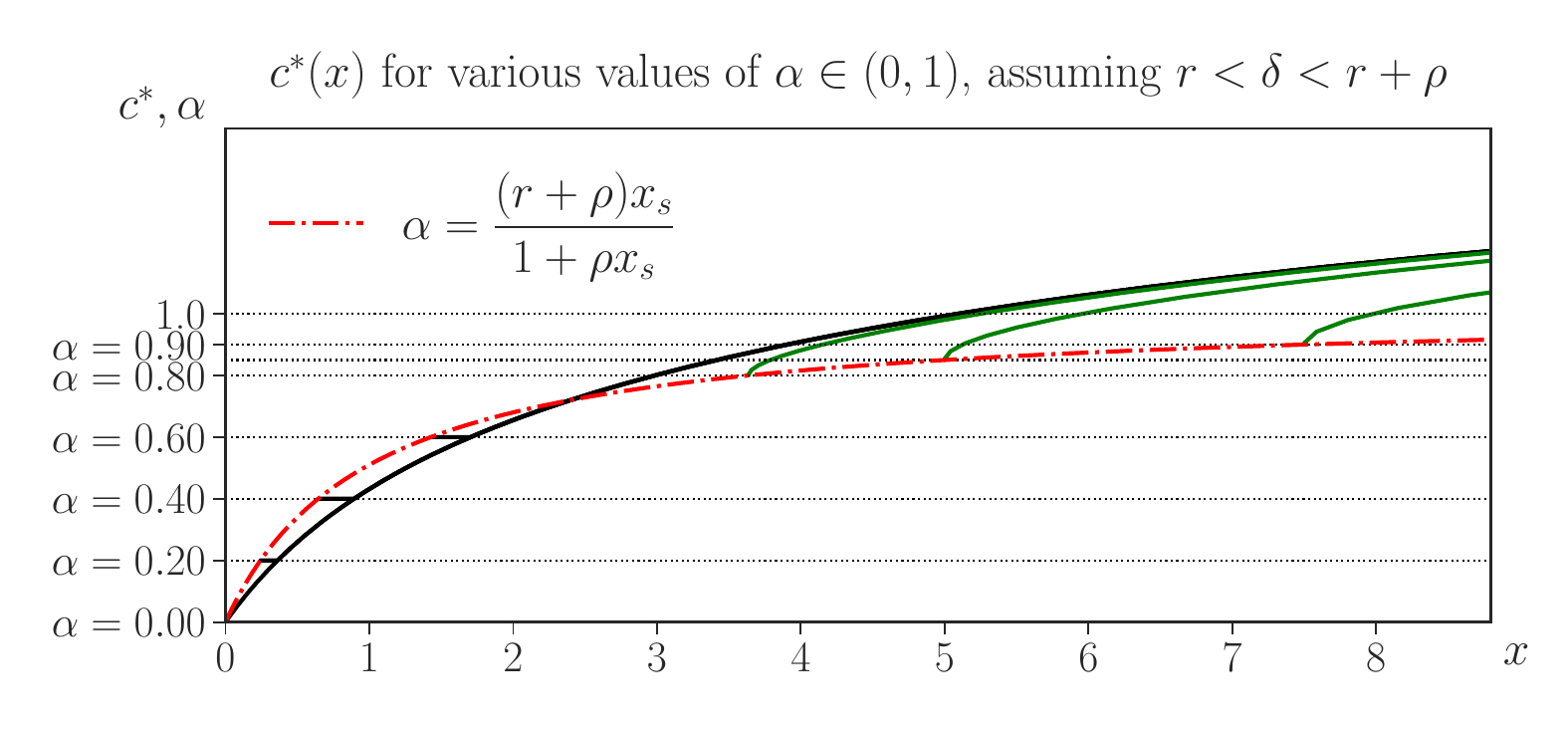}}
	}
	\centerline{
		\adjustbox{trim={0.0\width} {0.0\height} {0.0\width} {0.0\height},clip}
		{\includegraphics[scale=0.475, page=1]{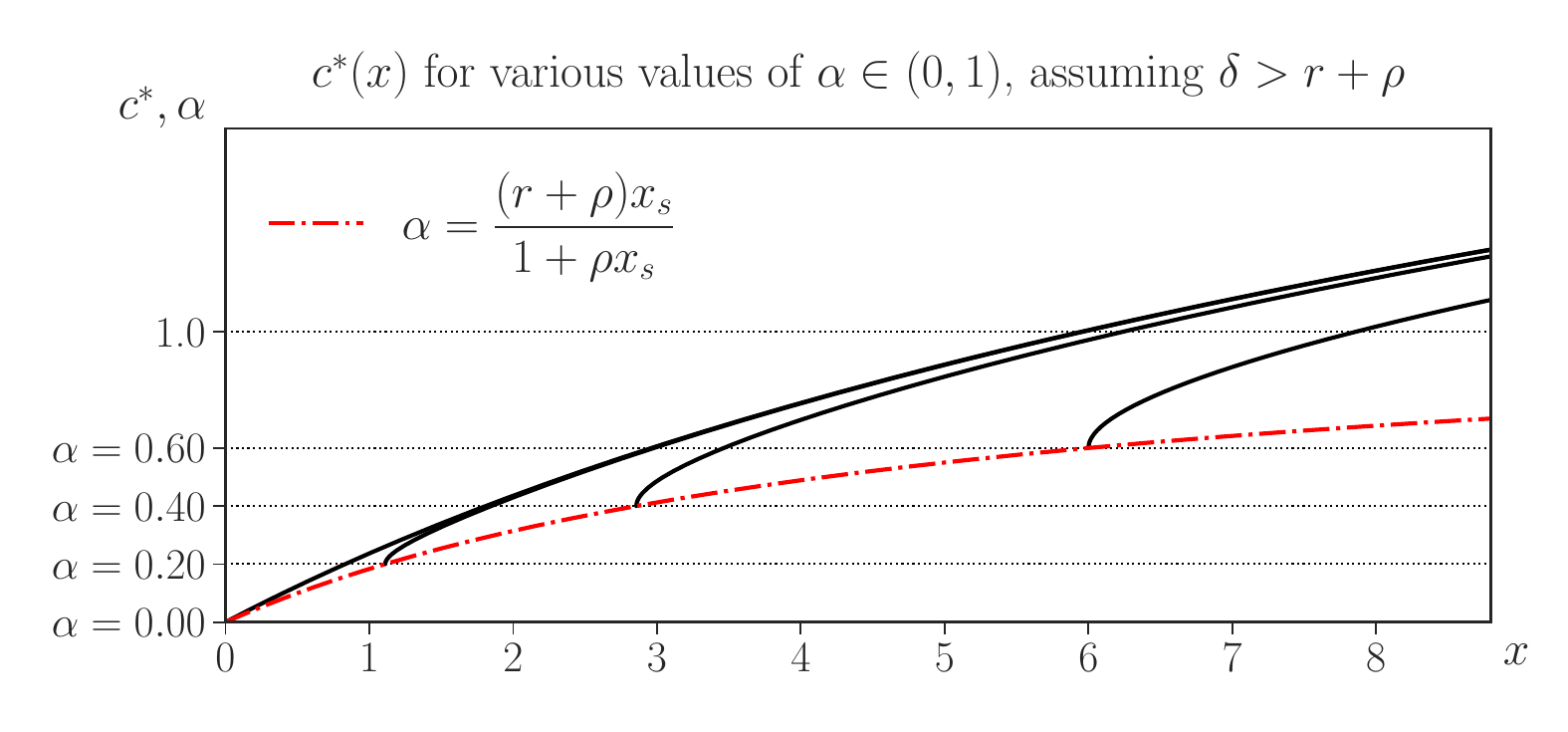}}
	}
	\caption{The optimal consumption-to-habit function $x\mapsto c^*(x)$ for various values of $\al\in(0,1)$ and with other parameters fixed. \textbf{Top:} For $0<\del<r$, the individual is patient for all values of $\al\in(0,1)$. \textbf{Middle:} For $r<\del<r+\rho$, the individual is patient (resp.\ impatient) for $\al$ near 0 (resp.\ 1). \textbf{Bottom:} For $\del>r+\rho$, the individual is impatient for all values of $\al\in(0,1)$. For patient individuals, the optimal wealth-to-habit threshold $\xs$ is increasing in $\al$. Furthermore, for patient individuals with only different $\al$, the optimal consumption-to-habit function $x\mapsto c^*(x)$ coincide for sufficiently large $x$. For impatient individuals, $c^*(x)$ is decreasing in $\al$. 
		\vspace{1em}
		\label{fig:alpha_sens}}
\end{figure}

\begin{figure}[p]
	\centerline{
		%		\fbox{
		\adjustbox{trim={0.01\width} {0.5\height} {0.01\width} {0.0\height},clip}
		{\includegraphics[scale=0.35, page=1]{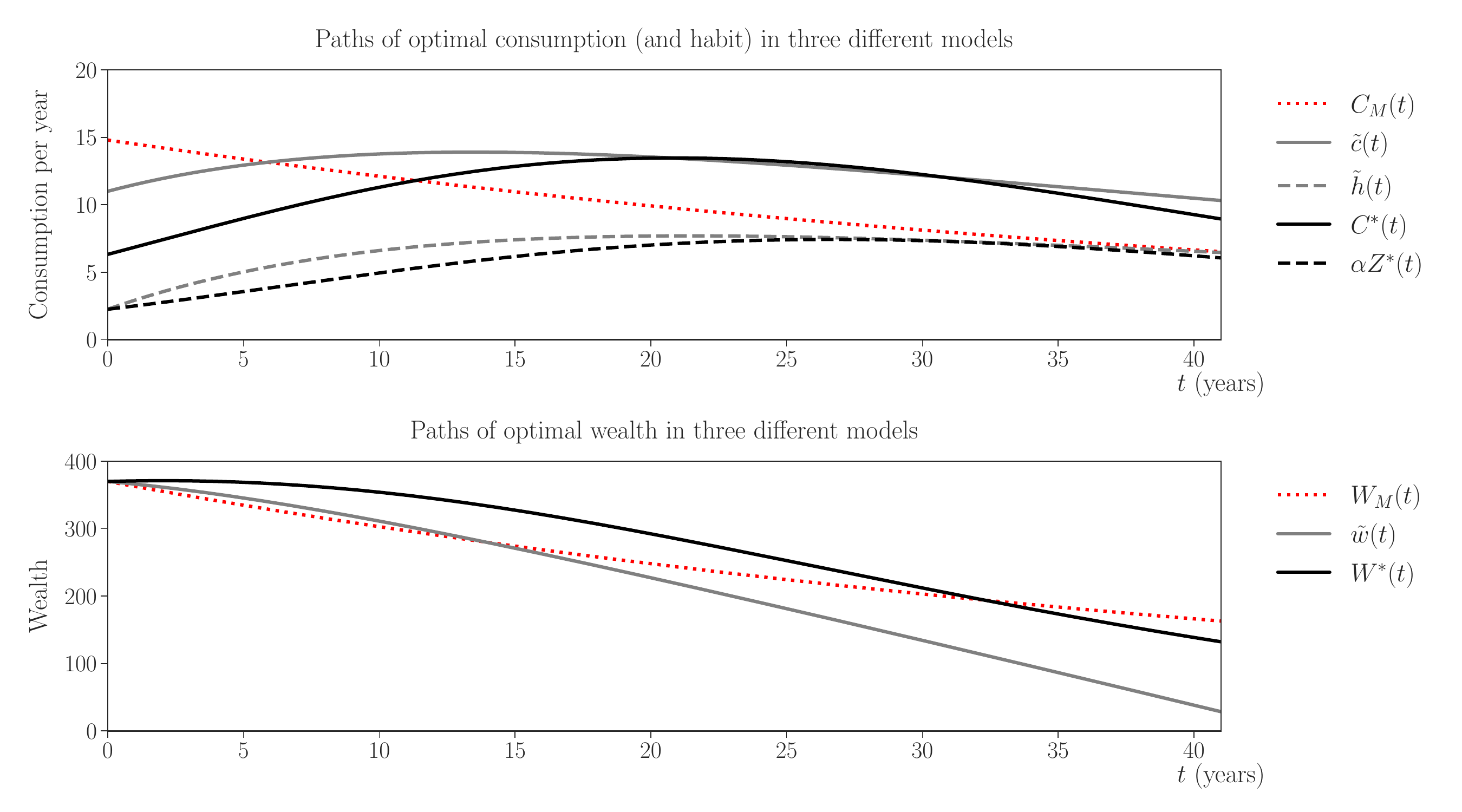}}
		%		}
	}\vspace{1em}
	\centerline{
		%		\fbox{
		\adjustbox{trim={0.01\width} {0.0\height} {0.01\width} {0.5\height},clip}
		{\includegraphics[scale=0.35, page=1]{xcs_impatient_Kraft.pdf}}
		%		}
	}
	\caption{
		\textbf{Top:} The red doted line $C_M(t)$ is the optimal consumption in infinite horizon Merton's model with only riskless investments given by \eqref{eq:CW_Merton}). The gray line (respectively, the gray dashed line) is the optimal consumption $\ct(t)$ (respectively, optimal habit $\htt(t)$) in the infinite horizon model of \cite{KraftMunkSeirfriedWagner2017} given by \eqref{eq:ctht_Kraft}. The black line (respectively, the dashed black line) is the optimal consumption $C^*(t)$ (respectively, the optimal habit reference $\al Z^*(t)$). 
		\textbf{Bottom:} The corresponding paths of the optimal wealth in the three models.
		\cite{KraftMunkSeirfriedWagner2017}'s policy spends the most wealth after 40 years, while our policy spends almost the same amount of wealth as Merton's (although not at a constant rate as Merton's).
		\textbf{Parameter values:} $r=0.02$, $\rho=0.174$, $\al=0.575$, $\del=0.1$, $\gam=4$, $w=370$, and $z=3.92$. Except for values of $\al$ and $z$, these values are taken from Table 1 of \cite{KraftMunkSeirfriedWagner2017} for singles. As explained in Remark \ref{rem:Kraft}, the parameter $\al$ (respectively, initial habit $h_0$) in \cite{KraftMunkSeirfriedWagner2017} corresponds to $\al\rho$ (respectively $\al z$) in our model. Thus, the values $\al=0.1$ and $h_0=2.253$ in \cite{KraftMunkSeirfriedWagner2017} translate to $\al=\frac{0.1}{0.174}=0.575$ and $z=\frac{2.253}{0.174}=3.92$ in our model.
		\label{fig:Merton_Kraft}}
\end{figure}

We end this section by comparing the optimal paths of consumption and wealth in our model with two other models. The first model is the classical infinite-horizon consumption problem in \cite{Merton1969} with only risk-free assets. In this model, the value function is
\begin{align}
	V_{M}(w) := \sup_{C(\cdot)\in\widetilde{\Ac}_0} \int_{0}^{+\infty}\ee^{-\del t} \frac{\big(C(t)\big)^{1-\gam}-1}{1-\gam}\, \dd t;\quad w\ge 0, \gam\in(0,1)\cup(1,+\infty),
\end{align}
in which $\widetilde{\Ac}_0$ is the set of consumption policies that avoid bankruptcy (i.e. set $\al=\rho=0$ in Definition 2.1). It can be shown that the value function is
$
	V_M(w) = \left(\frac{\gam}{\del+r(\gam-1)}\right)^\gam \frac{w^{1-\gam}}{1-\gam} - \frac{1}{\del(1-\gam)}, w>0;
$ 
that the optimal consumption feedback function is
$
	c_M(w) = \frac{\del+r(\gam-1)}{\gam}\; w, w>0;
$ 
and that the optimal wealth $W_M(t)$ and consumption $C_M(t)$ are given by
\begin{align}\label{eq:CW_Merton}
	W_M(t) = w\ee^{\frac{r-\del}{\gam}\,t},\quad C_M(t)=\frac{\del+r(\gam-1)}{\gam}\,w\,\ee^{\frac{r-\del}{\gam}\,t},
\end{align}
for $t\ge0$, in which $w>0$ is the initial wealth. For this model, if $\del> r$, then the wealth and consumption are decreasing. This is impatient consumption, as the individual has larger $\del$ and prefers to consume more now than in the future. If $\del= r$, then the wealth and consumption are constant. If $\del< r$, then the wealth and consumption are increasing. This patient consumption. Note that to have a finite value function, we must have $\del>r(1-\gam)$, otherwise, the value function $V_M$ explodes.

The second model is the infinite horizon optimal consumption policy in \cite{KraftMunkSeirfriedWagner2017}, which uses the classical habit formation utility. For $T=+\infty$, equations (20) and (21) therein yield the following dynamics for the optimal consumption path $\ct(t)$ and habit $\htt(t)$
\begin{align}
	\begin{cases}
		\ct'(t)-\htt'(t)=-\kappa\big(\ct(t)-\htt(t)\big),\\
		\ct'(t)+(\beta-\alpha)\ct(t)-(\beta-\kappa)\big(\ct(t)-\htt(t)\big)=0,
	\end{cases}
\end{align}
for $t\ge0$ and with $\ct(0)=c_0$ and $\htt(0)=h_0$ the initial values of consumption and habit. Note that \cite{KraftMunkSeirfriedWagner2017} provide the feedback form of the optimal consumption policy in terms of wealth (see equation (11) therein), which can be used to obtain $c_0$. Solving these ODEs yields explicit formulas for the optimal consumption path $\ct(t)$ and habit $\htt(t)$, namely,
\begin{align}\label{eq:ctht_Kraft}
	\begin{cases}
		\ct(t)=c_0 e^{(\alpha-\beta)t} - \frac{\beta-\kappa}{\kappa+\alpha-\beta}(c_0-h_0)\left(e^{-\kappa t}-e^{(\alpha-\beta)t}\right),\\
		\htt(t)=\ct(t)-(c_0-h_0)e^{-\kappa t},
	\end{cases}
\end{align}
for $t\ge0$. The optimal wealth $\wt(t)$ satisfies $\wt'(t)=r\wt(t)-\ct(t)$, which yields
$
	\tilde{w}(t)= e^{r t}\left(w +\int_0^t e^{-rs} \tilde{c}(s) d s\right).
$

The top plot in Figure \ref{fig:Merton_Kraft} compares the paths of the optimal consumption in our model (the black lines), the \cite{Merton1969} model with only riskless investments (the red dotted line), and the infinite-horizon model in \cite{KraftMunkSeirfriedWagner2017} (the gray lines). We have used parameter values as in Table 1 of \cite{KraftMunkSeirfriedWagner2017} for singles (see the caption in Figure \ref{fig:Merton_Kraft} for details). Merton's policy has the highest initial consumption rate and the lowest rate after 40 years, and it is a monotonically decreasing policy (i.e. it has no hump). Both ours and \cite{KraftMunkSeirfriedWagner2017}'s consumption are hump-shaped, but, \cite{KraftMunkSeirfriedWagner2017}'s consumption rates are generally higher than ours. The dashed lines show consumption habit reference levels in \cite{KraftMunkSeirfriedWagner2017}'s (i.e. $\htt(t)$ in \eqref{eq:ctht_Kraft}) and our model (i.e. $\al Z^*(t)$), which more or less match between the two model. Finally, the bottom plot in Figure \ref{fig:Merton_Kraft} shows the corresponding paths of the optimal wealth in the three models. \cite{KraftMunkSeirfriedWagner2017}'s policy spends almost all wealth after 40 years. Merton's policy spends wealth at a constant rate, with almost half of wealth remaining at the end. Our policy spends almost the same amount of wealth as Merton's, although not at a constant rate (more wealth is spend at the second half of the time period).

\section{Conclusion}\label{sec:5}
We considered an optimal consumption model for an individual who is unwilling to consume below a certain proportion $\alpha\in(0,1]$ of her consumption habit, in which $\al$ controls the degree of addictiveness. 
Assuming a risk-free market, we formulated and solved a deterministic control problem to maximize the discounted CRRA utility of the individual's consumption-to-habit process subject to the habit-formation constraint. We derived the optimal consumption policies explicitly in terms of the solution of a nonlinear free-boundary problem, which we analyzed in detail. Impatient individuals (or, equivalently, those with more addictive habits)  always consume above the minimum rate; thus, they eventually attain the minimum wealth-to-habit ratio.  Patient individuals (or, equivalently, those with less addictive habits) consume at the minimum rate if their wealth-to-habit ratio is below a threshold, and above it otherwise.  By consuming patiently, these individuals maintain a wealth-to-habit ratio that is greater than the minimum acceptable level. Additionally, we proved that the optimal consumption path is hump-shaped if the initial wealth-to-habit ratio is either: (1) larger than a high threshold; or (2) below a low threshold and the agent is more risk seeking (that is, less risk averse).

This paper complements \cite{AngoshtariBayraktarYoung2021} where we considered a similar model with the agent investing in a risky asset as well as the risk-free asset. The results presented herein for is not a special case of our other paper, however. The analysis in the stochastic case relies on randomness of the model and does not extend to the deterministic case considered here. Furthermore, there are structural differences between the consumption policies in the two settings. For instance, in the random setting, the individual consumes patiently (i.e. $\xa>\xu$) for all values of the parameters, while impatient consumption (i.e. $\xa=\xu$) can be optimal in the deterministic setting.

One way to extend our model is to allow bankruptcy by adapting the objective 
$
\int_0^{\tau_0} \frac{1}{1-\gam}\left(\frac{C(t)}{Z(t)}\right)^{1-\gam}\ee^{-\delta\,t}\,\dd t
$ 
in which $\tau_0$ is the time of bankruptcy (i.e. the first time $t$ where $W(t)=0$). Along the same direction, it will be interesting to consider a habit-formation constraint in the optimal dividend problem similar to how \cite{AlbrecherAzcueMuler2022} adapted the drawdown constraint to this problem. Another future research direction is to extend the setting of Subsection \ref{sub:Income} by considering non-constant income schedule. For instance, one may consider a (possibly random) retirement time $\tau_r$ such that the individual receives income $\eta>0$ over $[0,\tau_r]$ and zero thereafter.

%-------------------------------------------------------
%
%					Bibliography
%
%-------------------------------------------------------
%%\pagebreak
%\bibliographystyle{chicago}
%\bibliography{references}

\appendix

%-----------------------------------------------------------------------------------
%
%       SECTION: 		Proof of no-bankruptcy
%
%-----------------------------------------------------------------------------------

\section{Proof of Lemma \ref{lem:NoRuin}}\label{app:NoRuin}

The following lemma establishes a lower bound for the consumption habit process and is of use in later arguments.

\begin{lemma}\label{lem:Z-Lbound}
	Let $C = \{C(t)\}_{t \ge 0}$ be a consumption process satisfying \eqref{eq:Habit}, in which $Z$ is given by \eqref{eq:Z}. We, then, have
	\begin{align}\label{eq:Z-Lbound}
		Z(t)\ge Z(s) \ee^{-\rho(1-\al)(t-s)},
	\end{align}
	for all $0\le s \le t$. In particular, $Z(t) \ge z \ee^{-\rho(1-\al)t}$ for all $t\ge0$.
\end{lemma}
\begin{proof}
	For a fixed $s\ge0$, consider the consumption process $\Ct = \{\Ct(t)\}_{t\ge0}$, that coincides with $C$ over $[0, s)$, followed by consuming at the lowest rate allowed by \eqref{eq:Habit}.  In other words, $\Ct(t) = C(t)$ for $0\le t < s$, and $\Ct(t) = \al \Zt(t)$ for $t \ge s$. Here, $\{\Zt(t)\}_{t \ge 0}$ is the consumption habit process corresponding to $\Ct$, which satisfies
	\begin{align}\label{eq:Zt}
		\begin{cases}
			\dd \Zt(t) = -\rho\big(\Zt(t) - \Ct(t)\big)\dd t;\quad t \ge 0,\\
			\Zt(0)=z.
		\end{cases}
	\end{align}
	By definition, $\Zt(t) = Z(t)$ for $0\le t < s$. For $t \ge s$, on the other hand, \eqref{eq:Zt} yields
	\begin{align}\label{eq:Zt-IVP}
		\begin{cases}
			\frac{\dd \Zt(t)}{\dd t} = -\rho(1-\al) \Zt(t);\quad t \ge s,\\
			\Zt(s)=Z(s).
		\end{cases}
	\end{align}
	The solution to this initial-value problem is $\Zt(t) = Z(s) \ee^{-\rho(1-\al)(t-s)}$.
	
	% The lower bound in \eqref{eq:Z-Lbound} then follows from the fact that, for $t \ge s$, we have $C_t \ge \Ct_t$ and, thus, $Z_t \ge \Zt_t$.
	
	Next, we prove \eqref{eq:Z-Lbound}, that is, $Z(t) \ge \Zt(t)$ for $t\ge s$. For $k\in\{1,2,\dots\}$, define the process $Z^{(k)}$ by the recursive equation
	\begin{align}\label{eq:Ztk}
		\ee^{\rho t}Z^{(k)}(t) := \ee^{\rho s} Z(s) + \int_s^t \al\rho\ee^{\rho u} Z^{(k-1)}(u) \dd u;\quad t\ge s,
	\end{align}
	in which we have defined $Z^{(0)}(t):=Z(t)$ for $t\ge s$. Note, also, that
	\begin{align}\label{eq:ZtZs}
		\ee^{\rho t} Z(t) = \ee^{\rho s} Z(s) + \int_s^t \rho\ee^{\rho u} C(u) \dd u;\quad t\ge s,
	\end{align}
	by \eqref{eq:Z}. From \eqref{eq:Habit}, \eqref{eq:Ztk} (for $k=1$), and \eqref{eq:ZtZs}, we obtain $Z(t)\ge Z^{(1)}(t)$ for $t\ge s$. By using \eqref{eq:Ztk}, we deduce $Z^{(k)}(t)\ge Z^{(k+1)}(t)$, for $k\in\{1,2,\dots\}$. Furthermore, by definition, $Z^{(k)}(t)\ge 0$ for $t\ge s$ and $k\in\{1,2,\dots\}$. It, then, follows from the monotone convergence theorem that there exists a process $Z^{(+\infty)}$ such that, for $t\ge s$, $Z^{(+\infty)}(t) = \lim_{k\to\infty} Z^{(k)}(t)\le Z(t)$, and
	\begin{align}
		\ee^{\rho t}Z^{(+\infty)}(t) = \ee^{\rho s} Z(s) + \int_s^t \al\rho\ee^{\rho u} Z^{(+\infty)}(u) \dd u;\quad t\ge s.
	\end{align}
	The last integral equation is equivalent to \eqref{eq:Zt-IVP}. Therefore, $\Zt(t) = Z^{(+\infty)}(t) \le Z(t)$ for $t\ge s$, which proves \eqref{eq:Z-Lbound}. The last statement of the lemma follows trivially by setting $s=0$ in \eqref{eq:Z-Lbound}.
\end{proof}\vspace{1em}
 
\noindent\textbf{Proof of Lemma \ref{lem:NoRuin}:} Because $z>0$, it follows from Lemma \ref{lem:Z-Lbound} that $Z(t)>0$ for $t \ge 0$. Condition \eqref{eq:NoRuin}, then, implies that $W(t)>0$ for all $t\ge0$ (recall that we assumed $\al \in (0, 1]$). To show the reverse statement, assume that, at time $t\ge0$, the individual's wealth is $W(t)$ and her consumption habit is $Z(t)$. Assume that, thereafter, she consumes at the lowest rate, that is, $C(s)=\al Z(s)$ for $s\ge t$. From the proof of Lemma \ref{lem:Z-Lbound}, it follows that the consumption habit equals the lower bound in \eqref{eq:Z-Lbound}, that is, $Z(s)=Z(t) \ee^{-\rho(1-\al)(s-t)}$, for $s\ge t$. Note that we have exchanged the role of $t$ and $s$. From \eqref{eq:wealth}, we obtain
\begin{align}
	W(s) = W(t) + \int_t^s \left(r W(u) - \al Z(t) \ee^{-\rho(1-\al)(u-t)}\right)\dd u;\quad s\ge t,
\end{align}
which yields
\begin{align}
	W(s) = \ee^{r(s-t)} \left(W(t) - \frac{\al Z(t)}{r+\rho(1-\al)}\right) + \frac{\al Z(t)}{r+\rho(1-\al)} \ee^{-\rho(1-\al)(s-t)};\quad s\ge t.
\end{align}
The first term on the right dominates the second term as $s\to +\infty$. Thus,  if \eqref{eq:NoRuin} holds, the individual can avoid bankruptcy by setting $C(s) = \al Z(s)$ for all $s\ge t$. If \eqref{eq:NoRuin} does not hold, any consumption and investment policy leads to bankruptcy in finite time.

%-----------------------------------------------------------------------------------
%
%       SECTION: 		Proof of u impatient
%
%-----------------------------------------------------------------------------------

\section{Proof of Proposition \ref{prop:u-solution-impatient}}\label{app:u-solution-impatient}

The following lemma is used in the proof of Proposition \ref{prop:u-solution-impatient}.

\begin{lemma}\label{lem:u_is_convex}
	Let $\del\ge r+\rho(1-\al)$, and let $y$ be the solution of \eqref{eq:y-ODE-impatient}.  We, then, have
	\begin{align}\label{eq:y_lowerbound_imp}
		y(\psi)> \psi-\frac{\rho}{r+\rho}\psi^{1-\frac{1}{\gam}};\quad 0<\psi<\al^{-\gam}.
	\end{align}
\end{lemma}

\begin{proof}
	Define the function $w(\psi) := \psi-\frac{\rho}{r+\rho}\psi^{1-\frac{1}{\gam}}$ for $0<\psi\le\al^{-\gam}$. We want to show that $w(\psi)<y(\psi)$ for $0<\psi<\al^{-\gam}$. Let $\Pc$ be the defect of the ODE in \eqref{eq:y-ODE-impatient}, that is, $\Pc \phi(y)=\phi'(y)-f\big(y,\phi(y)\big)$, in which $f(y,\phi)$ is given by \eqref{eq:RHS-PSI-ODE} and $\phi(y)$ is an arbitrary function such that $\big(y,\phi(y)\big)$ is in the domain of $f$. We have
	\begin{align*}
		\Pc w(\psi) &= w'(\psi) - f\big(\psi, w(\psi)\big)
		= 1 - \frac{\rho}{r+\rho}\left(1-\frac{1}{\gam}\right)\psi^{-\frac{1}{\gam}}
		- \frac{\frac{\rho}{r+\rho}\left(\frac{r+\rho-\del}{\rho} -\psi^{-\frac{1}{\gam}}\right)\left(\psi-\frac{\rho}{r+\rho}\psi^{1-\frac{1}{\gam}}\right)}{\psi-\frac{\rho}{r+\rho}\psi^{1-\frac{1}{\gam}}-\frac{\del}{r+\rho}\psi}\\
		&= 1 - \frac{\rho}{r+\rho}\left(1-\frac{1}{\gam}\right)\psi^{-\frac{1}{\gam}} - 1 + \frac{\rho}{r+\rho}\psi^{-\frac{1}{\gam}} = \frac{1}{\gam} \, \psi^{-\frac{1}{\gam}} >0 = \Pc y(\psi),
	\end{align*}
	for $0<\psi< \al^{-\gam}$.  Furthermore, $w(\al^{-\gam}) = \al^{-\gam}\left(1-\frac{\rho}{r+\rho}\al\right)=\frac{\al^{-\gam}}{1+\rho\xu}=y(\al^{-\gam})$. Inequality \eqref{eq:y_lowerbound_imp}, then, follows from the comparison theorem for first-order ODEs.
\end{proof}
\vspace{1em}

\noindent\textbf{Proof of Proposition \ref{prop:u-solution-impatient}:} $(i)$ Let us first analyze the sign of the right side of the differential equation in \eqref{eq:y-ODE-impatient} for values of $y$ and $\psi$ in the region $0<y<\frac{\al^{-\gam}}{1+\rho\xu}$ and $0<\psi\le\al^{-\gam}$. To this end, define
\begin{align}\label{eq:RHS-PSI-ODE}
	f(\psi,y) := \frac{\frac{\rho}{r+\rho}\left(\frac{r+\rho-\del}{\rho} -\psi^{-\frac{1}{\gam}}\right)y}{y-\frac{\del}{r+\rho}\psi};\quad 0<\psi\le\al^{-\gam}, \, 0<y<\frac{\al^{-\gam}}{1+\rho\xu}.
\end{align}
Since $\del\ge r+\rho(1-\al)$ and $0<\psi\le\al^{-\gam}$, we have
\begin{align}\label{eq:yODE-denom}
	\frac{r+\rho-\del}{\rho}\le \al \le \psi^{-\frac{1}{\gam}}.
\end{align}
Thus, the numerator of $f$ is non-positive, and it follows that $f$ is non-negative in its domain if and only if $y\le \frac{\del}{r+\rho}\psi$. Thus, we look for a solution of \eqref{eq:y-ODE-impatient} in the domain
\begin{align}\label{eq:DC_eps}
	\Dc_\eps = \left\{(\psi,y): 0<\psi<\al^{-\gam}+\eps, \, 0<y<\frac{\del\psi}{r+\rho} \right\},
\end{align}
for $\eps > 0$.   The shaded regions in Figure \ref{fig:y_imp} represent the limiting domain $\Dc_0$. Consider the case $\del>r+\rho(1-\al)$ (the right plot in Figure \ref{fig:y_imp}). In this case, for a sufficiently small $\eps$, $f$ is positive and locally Lipschitz (with respect to y) in $\Dc_\eps$. Since the terminal value $\big(\al^{-\gam}, \frac{\al^{-\gam}}{1+\rho\xu} \big)$ is in $\Dc_\eps$, it follows that \eqref{eq:y-ODE-impatient} has a unique solution that extends to the left of $\Dc_\eps$. However, by the comparison theorem for first-order ODEs, we must have $0<y(\psi)<\del\psi/(r+\rho)$ for $0<\psi<\al^{-\gam}$. Thus, \eqref{eq:y-ODE-impatient} has a unique increasing solution satisfying $0<y(\psi)<\del\psi/(r+\rho)$ for $0<\psi<\al^{-\gam}$. Finally, we obtain the result for the case $\del=r+\rho(1-\al)$ (see the left plot of Figure \ref{fig:y_imp}) by letting $\del\to \big(r+\rho(1-\al)\big)^+$ and by using continuous dependence of the solution of \eqref{eq:y-ODE-impatient} with respect to $\del$ for the case $\del>r+\rho(1-\al)$.\vspace{1ex}

%-----------------------------------------------------------------------------------
%
%       Figure 1
%
%-----------------------------------------------------------------------------------

\begin{figure}[t]
	\centerline{
		% \fbox{
		\adjustbox{trim={0.0\width} {0.0\height} {0.0\width} {0.0\height},clip}
		{\includegraphics[scale=0.35, page=1]{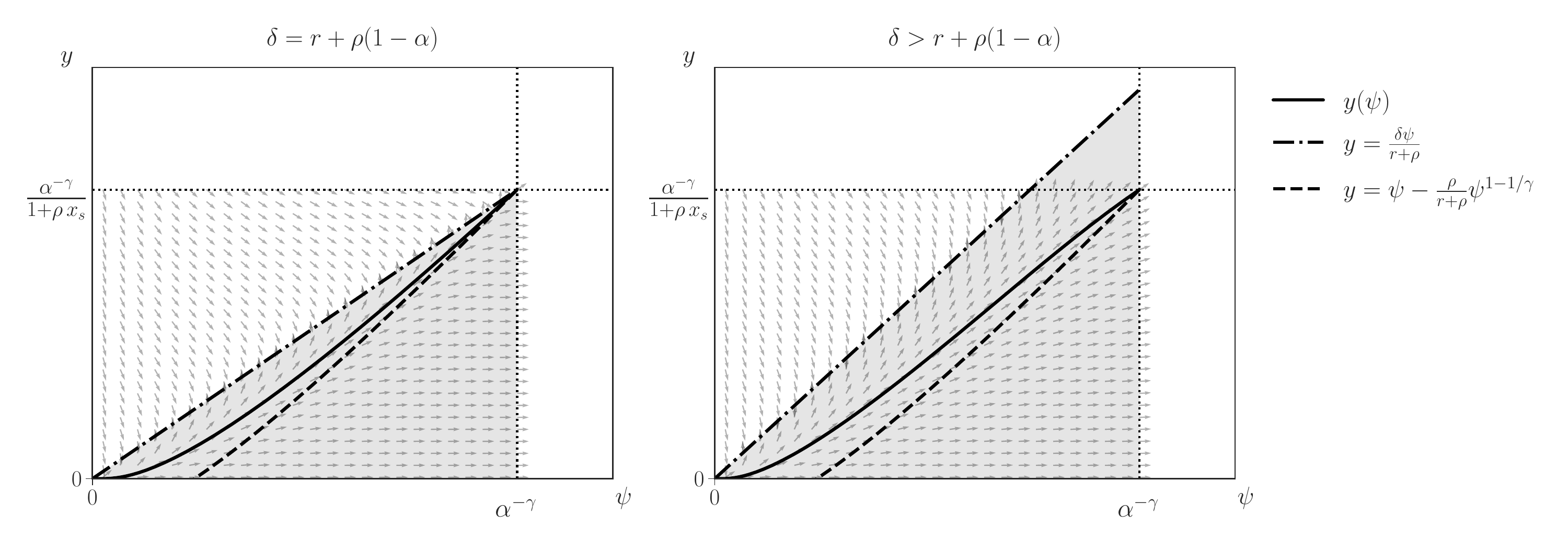}}
		% }
	}
	\caption{The direction fields and the solutions $y = y(\psi)$ of \eqref{eq:y-ODE-impatient} for the cases $\del>r+\rho(1-\al)$ (on the right) and $\del=r+\rho(1-\al)$ (on the left). The shaded areas are the domain $\Dc_0$ defined in \eqref{eq:DC_eps} where the right side of the differential equation in \eqref{eq:y-ODE-impatient} is positive. Note that any integral curve of the differential equation in the shaded regions approaches $(0,0)$.
		\vspace{1em}
		\label{fig:y_imp}}
\end{figure}

\noindent$(ii)$ Set $\yo=\ys = \frac{\al^{-\gam}}{1+\rho\xu}$. The FBP \eqref{eq:FBPDual2-riskless}--\eqref{eq:FBPDual4-riskless}, then, reduces to the terminal-value problem
\begin{align}\label{eq:u-TVP}
	\begin{cases}\displaystyle
		\big(r+\rho -\del\big)y u'(y) + \del u(y) = \frac{\gam}{1-\gam}\big(y-\rho y u'(y)\big)^{1-\frac{1}{\gam}};
		\qquad 0< y \le \ys,\vspace{1ex} \\
		u(\ys)=\dfrac{1}{\del}\left(\dfrac{\gam}{1-\gam}\al^{1-\gam}+(r+\rho-\del)\ys\xu\right) ~~ \Big\{\Longleftrightarrow~~u'(\ys)=-\xu\Big\}.
	\end{cases}
\end{align}
Let $\psi$ equal the inverse of $y$ obtained in (i); then, $\psi$ solves the terminal-value problem
\begin{align}\label{eq:psi-TVP}
	\begin{cases}\displaystyle
		\psi'(y) = \frac{y-\frac{\del}{r+\rho}\psi(y)}{\frac{\rho}{r+\rho}\left(\frac{r+\rho-\del}{\rho} -\big(\psi(y)\big)^{-\frac{1}{\gam}}\right)y};\quad 0<y\le\ys,\\
		\psi\left(\ys\right) = \al^{-\gam},
	\end{cases}
\end{align}
and it is straightforward to show that $u$ defined by
\begin{align}\label{eq:u-impatient2}
	u(y) = \frac{1}{\del}\left(\frac{\gam}{1-\gam}\al^{1-\gam}+(r+\rho-\del)\ys\xu\right)-\int_y^{\ys} \left(\frac{1}{\rho}-\frac{\psi(y')}{\rho y'}\right)\dd y'.
\end{align}
solves \eqref{eq:u-TVP}.

Next, we show that $u$ in \eqref{eq:u-impatient2} can be represented as in \eqref{eq:u-impatient}. To that end, define the operator
\begin{align}\label{eq:F}
	F(y) = \del u(y) - \frac{\gam}{1-\gam}\big(\psi(y)\big)^{1-\frac{1}{\gam}}-\frac{r+\rho-\del}{\rho}\big(\psi(y)-y\big);\quad 0<y\le\ys,
\end{align}
in which $u$ is given in \eqref{eq:u-impatient2}.  We want to show that $F(y)=0$ for $0<y\le \ys$. From \eqref{eq:u-impatient2} and the boundary condition in \eqref{eq:psi-TVP}, we have $F(\ys)=0$. Thus, it suffices to show that $F'(y)=0$ for $0<y<\ys$, which we demonstrate as follows:
\begin{align}\label{eq:Fcalcs}
	\begin{split}
		F'(y)&=\del u'(y) + \big(\psi(y)\big)^{-\frac{1}{\gam}}\psi'(y) -\frac{r+\rho-\del}{\rho}\big(\psi'(y)-1\big)\\
		&=\del\left(\frac{1}{\rho} - \frac{\psi(y)}{\rho y}\right) + \left(\big(\psi(y)\big)^{-\frac{1}{\gam}}-\frac{r+\rho-\del}{\rho}\right)\psi'(y)
		+\frac{r+\rho-\del}{\rho}\\
		&=\del\left(\frac{1}{\rho} - \frac{\psi(y)}{\rho y}\right)
		- \frac{y-\frac{\del}{r+\rho}\psi(y)}{\frac{\rho}{r+\rho}y}
		+\frac{r+\rho-\del}{\rho}=0,
	\end{split}
\end{align}
in which we used \eqref{eq:u-impatient2} to get the second equation and \eqref{eq:psi-TVP} to get the third equation.

We, now, complete the proof of Proposition \ref{prop:u-solution-impatient} by showing that $u = u(y)$ is decreasing and strictly convex with respect to $y \in (0, \ys]$. By \eqref{eq:u-impatient2}, we have $u'(y) = \frac{1}{\rho} - \frac{\psi(y)}{\rho y}$. By using \eqref{eq:psi-TVP}, we obtain that, for $0<y<\ys$,
\begin{align}
	u''(y) &= \frac{1}{\rho y}\left(\frac{\psi(y)}{y}-\psi'(y)\right)
	=\frac{1}{\rho y}\left(\frac{\psi(y)}{y}-\frac{y-\frac{\del}{r+\rho}\psi(y)}{\frac{\rho}{r+\rho}\left(\frac{r+\rho-\del}{\rho} -\big(\psi(y)\big)^{-\frac{1}{\gam}}\right)y}\right)\\
	&=\frac{(r+\rho)\left(\psi(y)-\frac{\rho}{r+\rho}\big(\psi(y)\big)^{1-\frac{1}{\gam}}-y\right)}{\rho^2 y^2\left(\frac{r+\rho-\del}{\rho} -\big(\psi(y)\big)^{-\frac{1}{\gam}}\right)}.
\end{align}
Note that, by \eqref{eq:yODE-denom}, we have $\frac{r+\rho-\del}{\rho} -\big(\psi(y)\big)^{-\frac{1}{\gam}}< 0$. Furthermore, it follows from Lemma \ref{lem:u_is_convex} that $\psi(y)-\frac{\rho}{r+\rho}\big(\psi(y)\big)^{1-\frac{1}{\gam}}-y < 0$. See Figure \ref{fig:y_imp} for an illustration. Thus, $u''(y)> 0$ for $0<y<\ys$. Finally, $u'(y)<0$ for $0<y<\ys$, since $u(y)$ is convex and that $u'(\ys)=-\xu<0$ by \eqref{eq:u-TVP}.

%-----------------------------------------------------------------------------------
%
%       SECTION: 		Proof of u patient
%
%-----------------------------------------------------------------------------------

\section{Proof of Proposition \ref{prop:u-solution-patient}}\label{app:u-solution-patient}

The following lemma is used in the proof of Proposition \ref{prop:u-solution-patient}. Its proof is similar to the proof of Lemma \ref{lem:u_is_convex} and is, thus, omitted.

\begin{lemma}\label{lem:y_bounds}
	Let $\del < r+\rho(1-\al)$, and let $y$ be as defined in \eqref{eq:y_patient}.  For $0<\psi<\psi_0$, we have
	\begin{align}\label{eq:yminus_bounds}
		\max\left(0, \psi-\frac{\rho}{r+\rho}\psi^{1-\frac{1}{\gam}}\right) < y(\psi) < \frac{\del}{r+\rho}\psi,
		\intertext{and, for $\psi > \psi_0$, we have}\label{eq:yplus_bounds}
		\frac{\del}{r+\rho}\psi < y(\psi) < \psi-\frac{\rho}{r+\rho}\psi^{1-\frac{1}{\gam}}.
	\end{align}
\end{lemma}\vspace{1ex}

\noindent\textbf{Proof of Proposition \ref{prop:u-solution-patient}:} $(i)$	We, first, analyze the sign of the right side of the differential equation in \eqref{eq:y-FBP}, that is, the sign of the function $f(\psi, y)$ of \eqref{eq:RHS-PSI-ODE} for $y>0$ and $0<\psi\le\al^{-\gam}$. Because $0<\del<r+\rho(1-\al)$, we have $\frac{r+\rho-\del}{\rho}-\psi^{-\frac{1}{\gam}}<0$ for $0<\psi<\psi_0$ and $\frac{r+\rho-\del}{\rho}-\psi^{-\frac{1}{\gam}}>0$ for $\psi_0<\psi<\al^{-\gam}$. It follows that
\begin{align}
	f(\psi, y) >0 \quad \Longleftrightarrow\quad (\psi, y) \in \Dc_- \cup \Dc_+,
\end{align}
in which we have defined
\begin{align}\label{eq:D_PM}
	\Dc_- = \left\{(\psi,y): 0<\psi<\psi_0,\, 0<y<\frac{\del\psi}{r+\rho} \right\},\text{ and }
	\Dc_+ = \left\{(\psi,y): \psi_0<\psi<\al^{-\gam}, \, y>\frac{\del\psi}{r+\rho} \right\}.
\end{align}
The shaded region in Figure \ref{fig:y_pat} represents the domain $\Dc_- \cup \Dc_+$.  It follows that any increasing solution of the differential equation in \eqref{eq:y-FBP} over the interval $[0,\al^{-\gam}]$ must pass through the point $(\psi_0, y_0)$. This point, however, is a singularity of the differential equation. Indeed, there are two integral curves passing through $(\psi_0, y_0)$, with one an increasing function of $\psi$ and the other a decreasing function. Here, we are interested in the increasing curve, and we construct it in the following paragraph.

Since $f$ is locally Lipschitz in $\Dc_-$, for $\eps$ in a right neighborhood of $0$, the terminal-value problem
\begin{align}\label{eq:y-TVP}
	\begin{cases}\displaystyle
		y'(\psi) = f\big(\psi, y(\psi)\big);\quad 0<\psi\le\psi_0,\\
		y(\psi_0) = y_0-\eps,
	\end{cases}
\end{align}
has a unique increasing solution that continuously depends on $\eps$. By taking the limit $\eps\to 0^+$, we obtain an increasing left solution $y_-(\psi)$ for $0<\psi<\psi_0$, such that $\lim_{\psi\to\psi_0^-}y_-(\psi)=y_0$. By applying a similar procedure to 
\begin{align}\label{eq:y-IVP}
	\begin{cases}\displaystyle
		y'(\psi) = f\big(\psi, y(\psi)\big);\quad \psi\ge\psi_0,\\
		y(\psi_0) = y_0+\eps,
	\end{cases}
\end{align}
we obtain an increasing right solution solution $y_+(\psi)$ for $\psi > \psi_0$, such that $\lim_{\psi\to\psi_0^+}y_+(\psi)=y_0$. Then, to get a solution over the whole domain, we define $y$ by
\begin{align}\label{eq:y_patient}
	y(\psi) =
	\begin{cases}
		y_-(\psi); &\quad 0<\psi<\psi_0,\\
		y_0; &\quad \psi=\psi_0,\\
		y_+(\psi); &\quad y>\psi_0.
	\end{cases}
\end{align}

%-----------------------------------------------------------------------------------
%
%       Figure 2
%
%-----------------------------------------------------------------------------------

\begin{figure}[t]
	\centerline{
		% \fbox{
		\adjustbox{trim={0.0\width} {0.0\height} {0.0\width} {0.0\height},clip}
		{\includegraphics[scale=0.35, page=1]{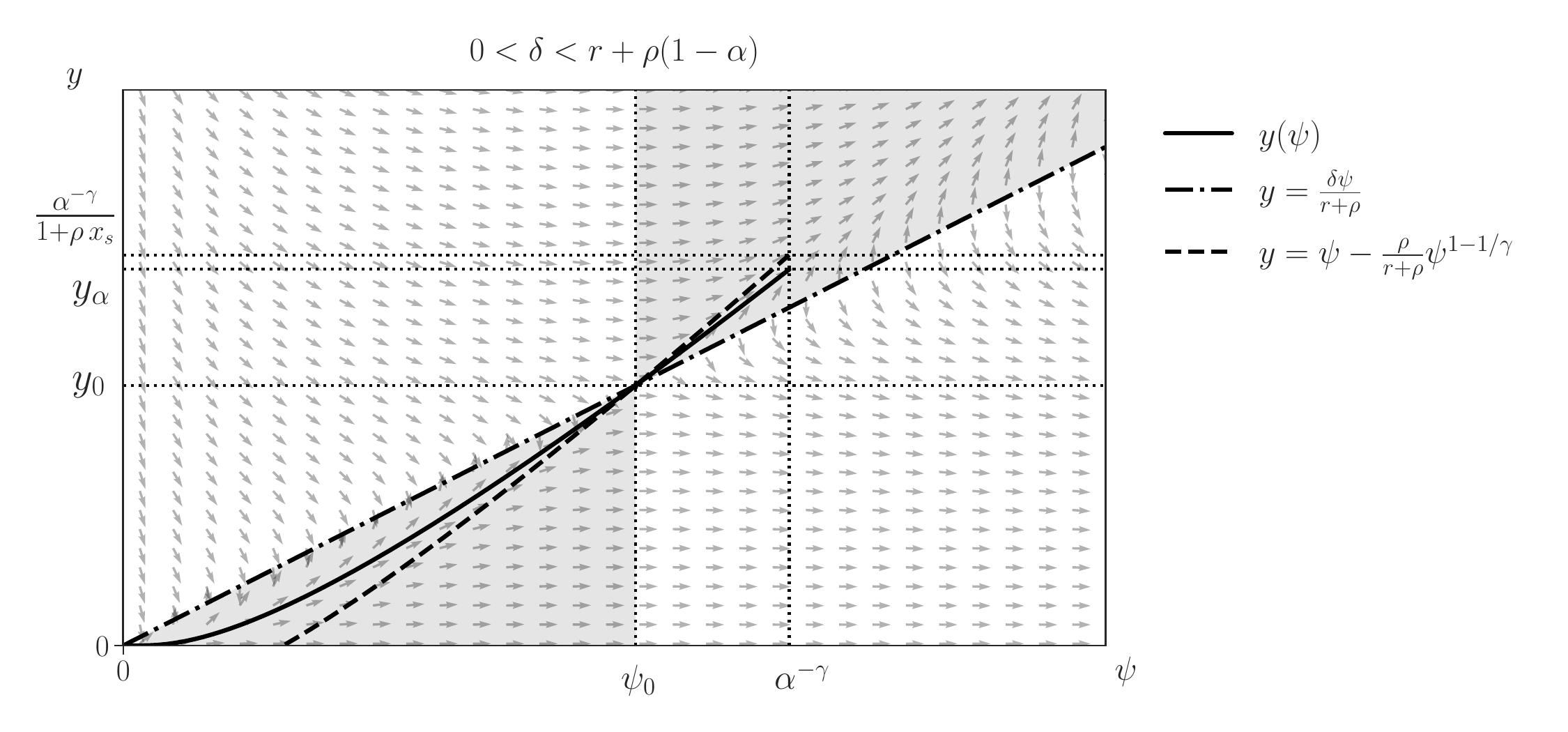}}
		% }
	}
	\caption{The direction field and the solution $y = y(\psi)$ of \eqref{eq:y-FBP}. The shaded area is the domain $\Dc_-\cup\Dc_+$ defined in \eqref{eq:D_PM} where the right side of the differential equation in \eqref{eq:y-FBP} is positive. Note that any increasing integral curve over $[0,\al^{-\gam}]$ has to pass through the point $(\psi_0, y_0)$.
		\vspace{1em}
		\label{fig:y_pat}}
\end{figure}

To show that \eqref{eq:y_patient} defines an increasing solution of the differential equation in \eqref{eq:y-FBP} over the interval $(0,\al^{-\gam}]$, it only remains to show: (a) $y_+(\psi)$ is defined for all $\psi \in [\psi_0, \al^{-\gam}]$, and (b) $y$ is differentiable at $\psi_0$.

%	Both of these statements are obtained by using the following lemma. \red{[correct this]}

Statement (a) directly follows from \eqref{eq:yplus_bounds}. To prove statement (b), first note that, by Lemma \ref{lem:y_bounds}, the left- and the right-derivatives $y_-'(\psi_0)$ and $y_+'(\psi_0)$ exist.  Because $y_\pm(\cdot)$ satisfy $y'_\pm(\psi)=f\big(\psi,y_\pm(\psi)\big)$, we have by L'H\^{o}pital's rule
\begin{align}
	\frac{(r+\rho)y_\pm'(\psi_0)}{\rho y_0} &= \lim_{\psi\to\psi^\pm}\frac{(r+\rho)y_\pm'(\psi)}{\rho y(\psi)}
	= \lim_{\psi\to\psi^\pm}\frac{\frac{r+\rho-\del}{\rho} - \psi^{-\frac{1}{\gam}}}{y_\pm(\psi) - \frac{\del}{r+\rho}\psi}
	=\lim_{\psi\to\psi^\pm}\frac{\frac{1}{\gam} \psi^{-1-\frac{1}{\gam}}}{y_\pm'(\psi) - \frac{\del}{r+\rho}} = \frac{\frac{1}{\gam} \psi_0^{-1-\frac{1}{\gam}}}{y_\pm'(\psi_0) - \frac{\del}{r+\rho}},
\end{align}
which, in turn, yields that $y_\pm'(\psi_0)$ satisfy the quadratic equation:
\begin{align}
	y_\pm^{\prime 2} - \frac{\del}{r+\rho} y_\pm' - \frac{\rho y_0}{(r+\rho)\gam}\psi_0^{-1-\frac{1}{\gam}} = 0.
\end{align}
This quadratic equation only has one positive solution. Therefore, we must have $y'_-(\psi_0)=y'_+(\psi_0)$, and $y$ defined by \eqref{eq:y_patient} is differentiable at $\psi_0$.

Finally, to obtain an increasing solution of the FBP \eqref{eq:y-FBP}, we set $\ys=y(\al^{-\gam})$, in which $y$ is given in \eqref{eq:y_patient}. The bounds on $\ys$ and $y$ directly follow from Lemma \ref{lem:y_bounds}.\vspace{1em}

\noindent$(ii)$ For $\yo=+\infty$, FBP \eqref{eq:FBPDual1-riskless}--\eqref{eq:FBPDual4-riskless} becomes,
\begin{align}\label{eq:u-FPB2}
	\begin{cases}
		\displaystyle
		\big(r+\rho(1-\al)-\del\big)y u'(y) + \del u(y) = \frac{\al^{1-\gam}}{1-\gam} -\al y;
		\qquad y>\ys,\vspace{1ex}\\
		\displaystyle
		\big(r+\rho -\del\big)y u'(y) + \del u(y) = \frac{\gam}{1-\gam}\Big(y-\rho y u'(y)\Big)^{1-\frac{1}{\gam}};
		\qquad 0< y \le \ys,\vspace{1ex}\\
		\displaystyle
		\lim_{y\to+\infty} u'(y) = -\xu,\vspace{1ex}\\
		\ys-\rho \ys u'(\ys) = \al^{-\gam},
	\end{cases}
\end{align}
in which $\ys>0$ is unknown. The general solution of the first differential equation in \eqref{eq:u-FPB2} is
\begin{align}\label{eq:linearODE}
	u(y) = C y^{-\frac{\del}{r+\rho(1-\al)-\del}} - \xu y + \frac{\al^{1-\gam}}{\del(1-\gam)};\quad y>\ys,
\end{align}
in which $C$ is an arbitrary constant to be determined. Because $0<\del<r+\rho(1-\al)$, we have $u'(y)+\xu = -\frac{C\del }{r+\rho(1-\al)-\del} y^{-\frac{r+\rho(1-\al)}{r+\rho(1-\al)-\del}}\to 0$ as $y\to+\infty$. Thus, the first boundary condition in \eqref{eq:u-FPB2} holds regardless of the value of $C$. The second boundary condition yields that $C = \frac{r+\rho(1-\al) - \del}{\del\rho}\big(\al^{-\gam} - \ys(1+\rho\xu)\big)(\ys)^{\del/(r+\rho(1-\al)-\del)}$. By substituting this value of $C$ in \eqref{eq:linearODE}, we obtain that $u$ equals the expression in \eqref{eq:u_patient_sol1} for $y>\ys$.	Therefore, \eqref{eq:u-FPB2} reduces to the following terminal-value problem:
\begin{align}\label{eq:u-FBP3}
	\begin{cases}
		\displaystyle
		\big(r+\rho -\del\big)y u'(y) + \del u(y) = \frac{\gam}{1-\gam}\big(y-\rho y u'(y)\big)^{1-\frac{1}{\gam}};
		\qquad 0< y \le \ys,\vspace{1ex}\\
		\ys-\rho \ys u'(\ys) = \al^{-\gam},\vspace{1ex}\\
		\left\{\Longleftrightarrow~~u(\ys) = \frac{r+\rho(1-\al) - \del}{\del\rho}\big(\al^{-\gam} - \ys(1+\rho\xu)\big)
		-\xu y+\frac{\al^{1-\gam}}{\del(1-\gam)}\right\}.
	\end{cases}
\end{align}
Let $\psi$ equal the inverse of $y$ obtained in (i); then, $\psi$ solves the terminal-value problem	
\begin{align}\label{eq:psi-FBP}
	\begin{cases}\displaystyle
		\psi'(y) = \frac{y-\frac{\del}{r+\rho}\psi(y)}{\frac{\rho}{r+\rho}\left(\frac{r+\rho-\del}{\rho} -\big(\psi(y)\big)^{-\frac{1}{\gam}}\right)y};\quad 0<y\le\ys,\\
		\psi\left(\ys\right) = \al^{-\gam},
	\end{cases}
\end{align}
and it is straightforward to show that $u$ defined on $(0, \ys]$ by
\begin{align}\label{eq:u_patient_sol2-aux}
	u(y) = &\frac{r+\rho(1-\al) - \del}{\del\rho}\big(\al^{-\gam} - \ys(1+\rho\xu)\big)
	-\xu \ys\\
	&+\frac{\al^{1-\gam}}{\del(1-\gam)}-\int_y^{\ys} \left(\frac{1}{\rho}-\frac{\psi(y')}{\rho y'}\right)\dd y',
\end{align}
solves \eqref{eq:u-FBP3}.

Furthermore, by defining $F(y)$ as in \eqref{eq:F} and by repeating \eqref{eq:Fcalcs}, one can show that $F\equiv0$ and, therefore, \eqref{eq:u_patient_sol2-aux} is equivalent to \eqref{eq:u_patient_sol2}. Hence, we have established that a solution of the FBP \eqref{eq:FBPDual1-riskless}--\eqref{eq:FBPDual4-riskless} is given by $\yo=+\infty$, $\ys$ as in part (i), and $u$ given by \eqref{eq:u_patient_sol1} and \eqref{eq:u_patient_sol2}.

It only remains to show that $u = u(y)$ is decreasing and strictly convex with respect to $y > 0$. For $y\ge\ys$, these properties readily follow by differentiating \eqref{eq:u_patient_sol1}.  Consider the case $0<y<\ys$.  By \eqref{eq:u_patient_sol2-aux}, we have $u'(y) = \frac{1}{\rho} - \frac{\psi(y)}{\rho y}$. By using \eqref{eq:psi-FBP}, we obtain that, for $0<y<\ys$,
\begin{align}
	u''(y) &= \frac{1}{\rho y}\left(\frac{\psi(y)}{y}-\psi'(y)\right)
	=\frac{1}{\rho y}\left(\frac{\psi(y)}{y}-\frac{y-\frac{\del}{r+\rho}\psi(y)}{\frac{\rho}{r+\rho}\left(\frac{r+\rho-\del}{\rho} -\big(\psi(y)\big)^{-\frac{1}{\gam}}\right)y}\right)\\
	&=\frac{(r+\rho)\left(\psi(y)-\frac{\rho}{r+\rho}\big(\psi(y)\big)^{1-\frac{1}{\gam}}-y\right)}{\rho^2 y^2\left(\psi_0^{-\frac{1}{\gam}} -\big(\psi(y)\big)^{-\frac{1}{\gam}}\right)}>0.
\end{align}
To obtain the last inequality, consider the cases $0<y<y_0$ and $y_0<y<\ys$ separately and apply Lemma \ref{lem:y_bounds}. Finally, $u'(y)<0$ for $0<y<\ys$, because $u(y)$ is convex and $u'(\ys)<0$, which one can see by differentiating \eqref{eq:u_patient_sol1}.

%-----------------------------------------------------------------------------------
%
%       SECTION: 		Proof of Transversality for riskless case
%
%-----------------------------------------------------------------------------------
\section{Auxiliary lemmas for the proof of Theorem \ref{thm:VF-riskless}}\label{app:Transversality-riskless}

The following lemma proves the so-called \emph{transversality property} of the solution of the HJB equation and is used in the first step of the proof of Theorem \ref{thm:VF-riskless}.

\begin{lemma}\label{lem:Transversality-riskless}
	Let $c(\cdot)\in \Ac_0$, and let $\{X(t)\}_{t \ge 0}$ be the corresponding wealth-to-habit process given by \eqref{eq:X-riskless}. We, then, have
	\begin{align}\label{eq:transversality-riskless}
		\lim_{T\to+\infty}\ee^{-\del T} \, \vpt\big(X(T)\big) = 0.
	\end{align}
\end{lemma}
\begin{proof}
	We have the following two trivial cases:
	\begin{itemize}
		\item $\gam>1$: In this case, $\frac{\al^{1-\gam}}{\delta(1-\gam)}\le \vpt\big(x\big) <0$, and \eqref{eq:transversality-riskless} immediately follows.
		\item $x=\xu$: In this case the only admissible consumption-to-habit process is $c(t)=\al$, for all $t\ge0$. The corresponding wealth-to-habit process is $X(t)=\xu$, for all $t\ge0$, which clearly satisfies \eqref{eq:transversality-riskless}.
	\end{itemize}
	Let us consider the nontrivial case in which $0<\gam<1$ and $x>\xu$. Since $c(t)\ge \al$ for $t\ge 0$, from \eqref{eq:X-riskless} we obtain
	\begin{align}
		X(T) = x + \int_0^t\Big[(r+\rho)X(u) - \big(1+\rho X(u)\big) c(u)\Big]\dd u \le x - \al T + \int_0^T \rt\,X(u)\dd u;\quad T\ge 0,
	\end{align}
	in which we have defined $\rt=r+\rho(1-\al)>0$.	Gronwall's inequality (for example, see \cite{Walter1970}, page 14) yields
	\begin{align}
		X(T) \le \xu + (x-\xu)\ee^{\rt\,T},\quad T\ge 0.
	\end{align}
	Define
	\begin{align}
		t_0 :=
		\begin{cases}
			0;&\quad x\ge \xs,\\
			\displaystyle\frac{1}{\rt}\log\left(\frac{\xs - \xu}{x-\xu}\right);&\quad \xu< x < \xs.
		\end{cases}
	\end{align}
	Note that $\xu + (x-\xu)\ee^{\rt\,t} \le \xs$ for $t<t_0$, and that $\xu + (x-\xu)\ee^{\rt\,t} > \xs$ for $t>t_0$. Since $\vpt(\cdot)$ is increasing and $0<\gam<1$, we have that, for $T\ge t_0$,
	\begin{align}
		0< \frac{\al^{1-\gam}}{\del(1-\gam)} = \vpt(\xu)
		\le \vpt \big(X(T)\big),
	\end{align}
	and
	\begin{align}
		\vpt\big(X(T)\big) &\le \vpt\Big(\xu + (x-\xu)\ee^{\rt T}\Big) = \vpt(x) + \int_0^T \vpt'\Big(\xu + (x-\xu)\ee^{\rt \,t}\Big) \, \rt\,(x-\xu)\ee^{\rt\,t}\, \dd t\\*
		&\le \vpt(x) + \int_0^{t_0} \vpt'\Big(\xu + (x-\xu)\ee^{\rt \,t}\Big) \, \rt\,(x-\xu)\ee^{\rt\,t}\, \dd t 
		+ \int_{t_0}^T \frac{\al^{-\gam}\rt\,(x-\xu)\ee^{\rt\,t}}{1 + \rho \xu + \rho (x-\xu)\ee^{\rt\,t}}\, \dd t\\*
		&= \vpt(x) + \int_0^{t_0} \vpt'\Big(\xu + (x-\xu)\ee^{\rt \,t}\Big) \, \rt\,(x-\xu)\ee^{\rt\,t}\, \dd t + \frac{\al^{-\gam}}{\rho}\, \log\left(\frac{1+\rho \xu + \rho (x-\xu)\ee^{\rt\,T}}{1+\rho \xu + \rho (x-\xu)\ee^{\rt\,t_0}}\right).
	\end{align}
 
	To get the second inequality, we used
	\begin{align}
		0<\vpt'\Big(\xu + (x-\xu)\ee^{\rt \,t}\Big)<\frac{\al^{-\gam}}{1+\rho\xu + \rho(x-\xu)\ee^{\rt \,t}};\quad t>t_0,
	\end{align}
	which follows from \eqref{eq:VE-riskless} and the fact that for $t>t_0$, one has $\xu + (x-\xu)\ee^{\rt \,t}>\xs$. Finally, we obtain that 
	\begin{align}
		0<\ee^{-\del T} \, \vpt\big(X(T)\big) \le \ee^{-\del T} \, \vpt(x)
		&+ \ee^{-\del T}\int_0^{t_0} \vpt'\Big(\xu + (x-\xu)\ee^{\rt \,u}\Big) \, \rt\,(x-\xu)\ee^{\rt\,u}\, \dd u\\*
		&+ \frac{\al^{-\gam} \ee^{-\del T}}{\rho}\, \log\left(\frac{1+\rho \xu + \rho (x-\xu)\ee^{\rt\,T}}{1+\rho \xu + \rho (x-\xu)\ee^{\rt\,t_0}}\right),
	\end{align}
	for $T\ge t_0$ and letting $T\to\infty$ yields \eqref{eq:transversality-riskless}.
\end{proof}

The following lemma is used in the second step of the proof of Theorem \ref{thm:VF-riskless}.

\begin{lemma}\label{lem:CS-comparison}
	Let $\cso$ be as in \eqref{eq:CS-sol-riskless}, and define $x_0:=\frac{r+\rho-\del}{\del\rho}$ and $c_0:=\frac{r+\rho-\del}{\rho}$. The following statements hold$:$
	\begin{enumerate}
		\item[$(i)$] If $\del\ge r+\rho(1-\al)$, then $\cso(x)>\frac{(r+\rho)x}{1+\rho x}$ for $x>\xu$.
		\item[$(ii)$] If $0< \del < r+\rho(1-\al)$, then
		\begin{itemize}
			\item $\cso(x)<\frac{(r+\rho)x}{1+\rho x}$ for $\xu\le x < x_0$.
			\item $\cso(x_0)=\frac{(r+\rho)x_0}{1+\rho x_0} = c_0$.
			\item $\cso(x)>\frac{(r+\rho)x}{1+\rho x}$ for $x>x_0$.
		\end{itemize}
	\end{enumerate}
\end{lemma}
\begin{proof}
	$(i)$ For $0<y<\ys$, Lemma \ref{lem:u_is_convex} yields that $y> \psi(y)-\frac{\rho}{r+\rho}\psi(y)^{1-\frac{1}{\gam}}$. Therefore,
	\begin{align}
		\psi(y)^{-\frac{1}{\gam}}> \frac{r+\rho}{\rho}\left(1-\frac{y}{\psi(y)}\right)
		=-\frac{r+\rho}{\rho}\,\frac{y-\psi(y)}{\psi(y)} = - \frac{(r+\rho)u'(y)}{1 - \rho u'(y)};\quad 0<y<\ys,
	\end{align}
	in which the last equality follows from $\psi(y) = y - \rho y u'(y)$, as can be seen from the proof of Proposition \ref{prop:u-solution-impatient}(ii).  Finally, we obtain statement (i) by substituting $y=J(-x)$ for $x>\xu$.
	
	$(ii)$ The proof is parallel to the proof of part (i) but uses the bounds given in Proposition \ref{prop:u-solution-patient}(i).
\end{proof}

%-----------------------------------------------------------------------------------
%
%       SECTION: 		Proof of Transversality for riskless case
%
%-----------------------------------------------------------------------------------
\section{Proof of Proposition \ref{prop:inverseU}}\label{app:inverseU}

We need the following lemma.

\begin{lemma}\label{lem:f}
	For $t\ge0$, we have $\frac{\dd C^{*}(t)}{\dd t}= - Z^*(t) f\big(X^*(t)\big)$, in which
	\begin{align}\label{eq:f}
		f(x):=
		\begin{cases}
			\rho\al(1-\al); &\quad \xu\le x<\xa,\\
			\rho \cs(x)\left[1+\frac{\del}{\gam}\frac{x-x_0}{1+\rho x} -\cs(x)\right]; &\quad x\ge \xa,
		\end{cases}
	\end{align}
	with $x_0=\frac{r+\rho-\del}{\del\rho}$ as in \eqref{eq:x0c0}.
\end{lemma}
\begin{proof}
	By \eqref{eq:XStar-riskless} and \eqref{eq:ZSTAR2}, we have
	\begin{align}\label{eq:CSTAR_prime2}
		\frac{\dd C^{*}(t)}{\dd t}= Z^*(t)\cs^\prime\big(X^{*}(t)\big)
		\frac{\dd}{\dd t}X^*(t)+\cs\big(X^{*}(t)\big)\frac{\dd Z^*(t)}{\dd t}
		=- Z^*(t) f\big(X^*(t)\big);\quad t\ge0,
	\end{align}
	in which
	\begin{align}\label{eq:f1}
		f(x) := c^{*\prime}(x)(1+\rho x)\left(\cs(x) - \frac{(r+\rho)x}{1+\rho x}\right) + \rho \cs(x)\big(1-\cs(x)\big),
	\end{align}
	for $x\ge \xu$. It follows from \eqref{eq:CS-sol-riskless} that
	\begin{align}\label{eq:f2}
		f(x)\equiv\rho\al(1-\al);\quad \xu\le x<\xa.
	\end{align}
	For $x\ge\xa$, by \eqref{eq:cs-cand-riskless} and \eqref{eq:VE-riskless}, we have
	\begin{align}\label{eq:vp_cs}
		v'(x) = \frac{\cs(x)^{-\gam}}{1+\rho x}.
	\end{align}
	Differentiating with respect to $x$ yields
	\begin{align}\label{eq:vpp_cs}
		v''(x) = -\frac{\cs(x)^{-\gam}}{1+\rho x}\left(\frac{\gam}{\cs(x)} \cs'(x)+\frac{\rho}{1+\rho x}\right).
	\end{align}
	Furthermore,
	\begin{align}
		- (r+\rho) x \vpt'(x) + \del \vpt(x) = \frac{\gam}{1-\gam}\cs(x)^{1-\gam},
	\end{align}
	by \eqref{eq:cs-cand-riskless} and \eqref{eq:FBP-riskless}.
	By differentiating with respect to $x$ and then eliminating $v'(x)$ and $v''(x)$ via \eqref{eq:vp_cs} and \eqref{eq:vpp_cs}, we then obtain
	\begin{align}\label{eq:cSTAR-ODE}
		\left(\cs(x) - \frac{(r+\rho)x}{1+\rho x}\right) \cs'(x) = \frac{\del\rho}{\gam}\frac{x - x_0}{(1+\rho x)^2}\cs(x).
	\end{align}
	By substituting for $\cs'(x)$ from the last equation into \eqref{eq:f}, it follows that
	\begin{align}
		f(x)
		&= \rho \cs(x)\left[1+\frac{\del}{\gam}\frac{x-x_0}{1+\rho x} -\cs(x)\right];
		\quad x\ge \xa.\qedhere
	\end{align}
\end{proof}
\vspace{1ex}

To prove Proposition \ref{prop:inverseU}, we first show the sufficiency of Conditions (i) and (ii); then, we show the necessity of those conditions.\vspace{1ex}

\noindent\textbf{Sufficiency of Condition $(i)$:} Assume that $\del>r$ and let $x> \max\{\xu,x_0\}$. By \eqref{eq:f}, we have $f(x) = \rho \cs(x) g(x)$, in which
\begin{align}\label{eq:g}
	g(x) = 1+\frac{\del}{\gam}\frac{x-x_0}{1+\rho x} -\cs(x).
\end{align}
By \eqref{eq:cSTAR-ODE}, we have
\begin{align}\label{eq:g-prime}
	g'(x) 
	=\frac{\del}{\gam}\frac{1+\rho x_0}{(1+\rho x)^2} - \cs'(x)
	= \frac{\del\left(\frac{(1+\rho x_0)(r+\rho)x}{1+\rho x}+(1+\rho x)\cs(x)\right)}
	{\gam(1+\rho x)^2\left(\frac{(r+\rho)x}{1+\rho x}-\cs(x)\right)}.
\end{align}
Since $x>x_0$, Lemma \ref{lem:CS-comparison} yields $\cs(x)>\frac{(r+\rho)x}{1+\rho x}$. Thus, $g'(x)<0$ by \eqref{eq:g-prime}. Furthermore, $g(x_0) = 1-c_0 = (\del - r)/\rho >0$ and $\lim_{x\to+\infty} g(x) = \frac{\del}{\gam\rho} - \lim_{x\to+\infty} \cs(x) = -\infty$. It follows that there exists a unique $\xH>x_0$ satisfying \eqref{eq:xH} such that $g(x)>0$ (resp.\ $g(x)<0$) if $x\in(x_0, \xH)$ (resp.\ $x>\xH$).

Now, assume $X^*(0)=\frac{w}{z}>\xH$. By Corollary \ref{cor:optimconsum_riskless}, $X^*(t)$ is a decreasing process such that $\lim_{t\to+\infty} X^*(t) = x_0<\xH$.  Because $X^*(t)$ is continuous and decreasing, there is a unique $\tau_h>0$ such that $X^*(\tau_h)=\xH$, $X^*(t)>\xH$ for $t\in[0,\tau_h)$, and $x_0<X^*(t)<\xH$ for $t>\tau_h$.  From Lemma \ref{lem:f}, it follows that
\begin{align}
	\frac{\dd}{\dd t} C^*(t) = - Z^*(t)\rho \cs\big(X^*(t)\big) g\big(X^*(t)\big) >0,
\end{align}
for $t\in[0,\tau_h)$, and
\begin{align}
	\frac{\dd}{\dd t} C^*(t) = - Z^*(t)\rho \cs\big(X^*(t)\big) g\big(X^*(t)\big) <0,
\end{align}
for $t>\tau_h$. In particular, the graph of $t\mapsto C^*(t)$ is hump-shaped and attains its maximum at $\tau_h$, as claimed.\vspace{1ex}

\noindent\textbf{Sufficiency of Condition $(ii)$:} Assume $r<\del<r+\rho(1-\al)$, and note that  $\xu<\xa<x_0$ by Proposition \ref{prop:HJB-riskless}. Let $x\in(\xa, x_0)$, and define $g(x)$ by \eqref{eq:g}. From Lemma \ref{lem:CS-comparison}, we have $\cs(x)<\frac{(r+\rho)x}{1+\rho x}$. Thus, \eqref{eq:g-prime} yields $g'(x)>0$. From $g(x_0)= (\del - r)/\rho >0$ and $g(\xa)<0$ (by Condition (ii)), it follows that there exists a unique constant $\xH'\in(\xa, x_0)$ satisfying \eqref{eq:xHp} such that $g(x)>0$ (resp.\ $g(x)<0$) if $x\in(\xH', x_0)$ (resp.\ $x\in(\xa, \xH')$).

Next, assume that $X^*(0)=\frac{w}{z}\in(\xa,\xH')$. Since $X^*(0)<x_0$, Corollary \ref{cor:optimconsum_riskless} yields that $X^*(t)$ is increasing and $\lim_{t\to+\infty} X^*(t) = x_0>\xH'$. Since $X^*(t)$ is continuous and increasing, there is a unique $\tau_h'>0$ such that $X^*(\tau_h')=\xH'$, $\xa<X^*(t)<\xH'$ for $t\in[0,\tau_h')$, and $\xH'<X^*(t)<x_0$ for $t>\tau_h'$. From Lemma \ref{lem:f}, it follows that
\begin{align}
	\frac{\dd}{\dd t} C^*(t) = - Z^*(t)\rho \cs\big(X^*(t)\big) g\big(X^*(t)\big) >0,
\end{align}
for $t\in[0,\tau_h')$, and
\begin{align}
	\frac{\dd}{\dd t} C^*(t) = - Z^*(t)\rho \cs\big(X^*(t)\big) g\big(X^*(t)\big) <0,
\end{align}
for $t>\tau_h'$. In particular, the graph of $t\mapsto C^*(t)$ is hump-shaped and attains its maximum at $\tau_h'$, as claimed.\vspace{1ex}

\noindent\textbf{Necessity of Conditions $(i)$ and $(ii)$:} By Proposition \ref{prop:HJB-riskless}, $\xa=\xu\ge x_0$ $($resp.\ $\xu<\xa<x_0)$ if $\del\ge r+\rho(1-\al)$ $($resp.\ $0<\del <r+\rho(1-\al))$. Therefore, Conditions $(i)$ and $(ii)$ are false if and only if one of the following six scenarios is true: 
\begin{enumerate}
	\item[(a)] $r\ge \del$,
	
	\item[(b)] $\del\ge r+\rho(1-\al)$ and $\xu\le w/z \le \xH$,
	
	\item[(c)] $r <\del<r+\rho(1-\al)$ and $x_0\le w/z \le \xH$,
	
	\item[(d)] $r <\del<r+\rho(1-\al)$ and $\xu\le w/z \le \xa$,
	
	\item[(e)] $r <\del<r+\rho(1-\al)$, $1+\frac{\del}{\gam}\frac{\xa-x_0}{1+\rho \xa} - \al < 0$, and $\xH'\le w/z < x_0$,
	
	\item[(f)] $r <\del<r+\rho(1-\al)$, $1+\frac{\del}{\gam}\frac{\xa-x_0}{1+\rho \xa} -\al \ge 0$.
\end{enumerate}
Thus, to show the necessity of Conditions (i) and (ii) for the presence of a consumption hump, it suffices to show that $t\mapsto C^*(t)$ is not hump-shaped in scenarios (a)-(f) above.

Under scenario (a), we have $\del\le r\le r+\rho(1-\al)$. From Corollary \ref{cor:optimconsum_riskless}, we know $\lim_{t\to+\infty}X^*(t)= x_0>\xa$ and $\cs(x_0)=c_0:=(r+\rho-\del)/\rho$. Therefore, $\lim_{t\to+\infty} f\big(X^*(t)\big) = \rho c_0 (1-c_0) = c_0 (\del-r)\le 0$ by \eqref{eq:f}. From Lemma \ref{lem:f}, it follows that there exists a constant $T$ such that $\frac{\dd C^{*}(t)}{\dd t}\ge 0$ for $t\ge T$. Since $t\mapsto C^*(t)$ is asymptotically non-decreasing, it cannot be hump-shaped.

In scenarios $(b)$ and $(c)$, we can apply the argument used for the proof of sufficiency of Condition (i) to conclude that $t\mapsto C^*(t)$ is decreasing and, thus, not hump-shaped.

In scenario (d), $f(x)=\rho\al(1-\al)>0$ by \eqref{eq:f}. Lemma \eqref{lem:f} then yields that $\frac{\dd}{\dd t} C^{*}(0)<0$. Thus, $t\mapsto C^*(t)$ is initially decreasing and cannot be hump-shaped.

Finally, in scenarios (e) and (f), we can apply the argument used for the proof of sufficiency of Condition (ii) to conclude that $t\mapsto C^*(t)$ is decreasing and, thus, not hump-shaped.

To end the proof of Proposition \ref{prop:inverseU}, it only remains to show its last statement. The first part of the statement is clear. The second part (that is, Conditions $(ii)$ fails if $\gam > 1-\frac{\del-r}{\rho(1-\al)}$) follows from the inequality
\begin{align}
	1+\frac{\del}{\gam}\frac{\xa-x_0}{1+\rho \xa} - \al
	>1+\frac{\del}{\gam}\frac{\xu-x_0}{1+\rho \xu} - \al
	%	=\frac{\del-r+\rho(1-\al)(\gam-1)}{\rho\gam}
	=\frac{1-\al}{\gam}\left(\gam - 1+\frac{\del-r}{\rho(1-\al)}\right).
\end{align}
To get the inequality, we used $\xa>\xs$ and that $\frac{x-x_0}{1+\rho x}$ is increasing in $x$. To get the equality, we used the definitions $x_0=\frac{r+\rho-\del}{\del \rho}$ and $\xu=\frac{\al}{r+\rho(1-\al)}$.
\end{document}